\documentclass{tglat2e}

\usepackage[style=alphabetic, backend=bibtex, firstinits=true, doi=false, url=false, isbn=false, sorting=nyt]{biblatex}
\bibliography{biblist.bib}
\AtEveryBibitem{
  \ifentrytype{book}{%
    \clearfield{pages}
  }{%
  }%
}
\renewbibmacro{in:}{
  \ifentrytype{article}{}{\printtext{\bibstring{in}\intitlepunct}}}
\DeclareFieldFormat[article,incollection,misc]{title}{#1}  
\DeclareFieldFormat{pages}{#1}  
\DeclareFieldFormat[article]{volume}{\mkbibbold{#1}}  
\renewbibmacro*{volume+number+eid}{%
  \printfield{volume}%
  \setunit*{\addcolon}
  \printfield{number}%
  \setunit{\addcomma\space}%
  \printfield{eid}}

\usepackage[colorlinks=true, pdfstartview=FitV, linkcolor=blue, citecolor=blue, urlcolor=blue, breaklinks=true]{hyperref}
\usepackage{amsmath,amsfonts,amssymb,amscd,comment,euscript,mathrsfs}
\usepackage[usenames]{color}
\usepackage[all]{xy}

\newcommand\C{\mathbb{C}}
\newcommand\Z{\mathbb{Z}}

\newcommand\rat{\mathrm{rat}}
\newcommand\rss{\mathrm{ss}}
\newcommand\ab{\mathrm{ab}}
\newcommand\eval{\mathrm{eval}}
\newcommand\noneval{\mathrm{noneval}}
\newcommand\SL{\mathrm{SL}}
\newcommand\Sym{\mathrm{Sym}}
\newcommand\T{\mathrm{T}}
\newcommand\ex{\mathrm{ex}}

\newcommand\g{\ensuremath{\mathfrak{g}}}
\newcommand\frM{\mathfrak{M}}
\newcommand\frK{\mathfrak{K}}
\newcommand\fm{\mathfrak{m}}
\newcommand\frB{\mathfrak{B}}
\newcommand\frD{\mathfrak{D}}

\newcommand\bx{\mathbf{x}}

\newcommand\cF{\ensuremath{\mathcal{F}}}

\newcommand\cB{\ensuremath{\mathcal{B}}}

\newcommand\gq{/\!/} 

\newcommand\ot{\otimes}
 
\newcommand\ts{\textstyle}
\newcommand\ideal{\trianglelefteq}
\newcommand\pa{\partial}

\newcommand\NN{\mathbb{N}}

\newcommand\scF{\ensuremath{\mathcal{F}}}

\newcommand\scS{\ensuremath{\mathcal{S}}}

\newcommand\cC{\ensuremath{\mathcal{C}}}

\newcommand\frb{\mathfrak{b}}

\newcommand\frg{\g}
\newcommand\frh{\mathfrak{h}}
\newcommand\lsl{\ensuremath{\mathfrak{sl}}}
\newcommand\gl{\ensuremath{\mathfrak{gl}}}
\newcommand\frl{\ensuremath{\mathfrak{l}}}

\newcommand\frn{\mathfrak{n}}
\newcommand\frN{\mathfrak{N}}
\newcommand\frQ{\mathfrak{Q}}
\newcommand\frs{\mathfrak{s}}

\newcommand\frZ{\mathfrak{Z}}

\newcommand\al{\alpha}
\newcommand\be{\beta}

\newcommand\Ga{\Gamma}
\newcommand\de{\delta}

\newcommand\eps{\varepsilon}

\newcommand\la{\lambda}

\newcommand\vphi{\varphi}
\newcommand\rh{\rho}
\newcommand\si{\sigma}

\newcommand\ta{\tau}

\newcommand\om{\omega}

\DeclareMathOperator{\Hom}{Hom}
\DeclareMathOperator{\Ext}{Ext}

\DeclareMathOperator{\End}{End}
\DeclareMathOperator{\Aut}{Aut}
\DeclareMathOperator{\Span}{Span}
\DeclareMathOperator{\Spec}{Spec}
\DeclareMathOperator{\Der}{Der}
\DeclareMathOperator{\IDer}{IDer}
\DeclareMathOperator{\Id}{Id}

\DeclareMathOperator{\Supp}{Supp} 
\DeclareMathOperator{\Int}{Int}
\DeclareMathOperator{\Out}{Out}
\DeclareMathOperator{\maxSpec}{maxSpec}
\DeclareMathOperator{\wt}{wt}

\DeclareMathOperator{\ev}{ev}

\DeclareMathOperator{\rmH}{H}

\DeclareMathOperator{\Ker}{Ker}
\DeclareMathOperator{\res}{res}

\theoremstyle{plain}
\newtheorem{theo}{Theorem}[section]
\newtheorem{prop}[theo]{Proposition}
\newtheorem{lem}[theo]{Lemma}
\newtheorem{cor}[theo]{Corollary}

\theoremstyle{definition}
\newtheorem{defin}[theo]{Definition}
\newtheorem*{rem*}{Remark}
\newtheorem{example}[theo]{Example}

\theoremstyle{remark}
\newtheorem{rem}[theo]{Remark}
\newtheorem{rems}[theo]{Remarks}

\numberwithin{equation}{section}
\renewcommand{\theenumi}{\alph{enumi}}

\renewcommand{\theenumii}{\roman{enumii}}


\newcommand{\lv}[1]{}

\begin{document}
%

\title[Extensions and block decompositions for EMAS]{Extensions and block decompositions for finite-dimensional representations of\\ equivariant map algebras}

\authors{Erhard Neher\thanks{Supported by an NSERC Discovery Grant}
\address Department of Mathematics \& Statistics\\
University of Ottawa\\
Ottawa, ON, Canada
\email neher@uottawa.ca
\and
Alistair Savage\thanks{Supported by an NSERC Discovery Grant}
\address Department of Mathematics \& Statistics\\
University of Ottawa\\
Ottawa, ON, Canada
\email alistair.savage@uottawa.ca
}


\maketitle

\begin{abstract}
  Suppose a finite group acts on a scheme $X$ and a finite-dimensional Lie algebra $\g$.  The associated \emph{equivariant map algebra} is the Lie algebra of equivariant regular maps from $X$ to $\g$.  The irreducible finite-dimensional representations of these algebras were classified in \cite{NSS}, where it was shown that they are all tensor products of evaluation representations and one-dimensional representations.

  In the current paper, we describe the extensions between irreducible finite-dimensional representations of an equivariant map algebra in the case that $X$ is an affine scheme of finite type and $\g$ is reductive.  This allows us to also describe explicitly the blocks of the category of finite-dimensional representations in terms of \emph{spectral characters}, whose definition we extend to this general setting.  Applying our results to the case of generalized current algebras (the case where the group acting is trivial), we recover known results but with very different proofs.  For (twisted) loop algebras, we recover known results on block decompositions (again with very different proofs) and new explicit formulas for extensions.  Finally, specializing our results to the case of (twisted) multiloop algebras and generalized Onsager algebras yields previously unknown results on both extensions and block decompositions.
\end{abstract}

\tableofcontents

%
\section*{Introduction}
%

Equivariant map algebras are a large class of Lie algebras that include (twisted) loop or multiloop algebras, generalized current algebras, and generalized Onsager algebras, among others.  Suppose $X$ is a scheme and $\g$ is a finite-dimensional Lie algebra, both defined over an algebraically closed field of characteristic zero, and that $\Gamma$ is a finite group acting on both $X$ and $\g$ by automorphisms.  Then the \emph{equivariant map algebra} $\frM = M(X,\g)^\Gamma$ is the Lie algebra of equivariant algebraic maps from $X$ to $\g$.  One easily sees that $\frM \cong M(V, \g)^\Ga$ where $V=\Spec A$, and $A$ is the ring of global functions on $X$.  We will therefore assume throughout the paper that $X$ is affine.

In \cite{NSS}, the authors, together with P.~Senesi, gave a complete classification of the irreducible finite-dimensional representations of an arbitrary equivariant map algebra.  The main result there was that such representations are all tensor products of an irreducible evaluation representation and a one-dimensional representation.  Here by \emph{evaluation representation} we mean a representation of the form
\[ \textstyle
  \frM \xrightarrow{\ev_\bx} \bigoplus_{x \in \bx} \g^x \xrightarrow{\bigotimes_{x \in \bx} \rho_x} \End_k \left( \bigotimes_{x \in \bx} V_x \right),
\]
where $\bx$ is a finite subset of $X_\rat$, the set of rational points of $X$, $\ev_\bx$ is the map given by evaluation at the points of $\bx$, $\g^x$ is the subalgebra of $\g$ fixed by the isotropy group of $x$, and $\rho_x$, $x \in \bx$, are finite-dimensional representations $\g^x \to \End_k V_x$.   When all $\rho_x$, $x \in \bx$, are irreducible and no two points of $\bx$ lie in the same $\Gamma$-orbit, the corresponding evaluation representation is irreducible.  In many cases, including the generalized current algebras (for semisimple $\g$), multiloop algebras and generalized Onsager algebras, all irreducible finite-dimensional representations are in fact evaluation representations.  For generalized current algebras this was shown in \cite{CFK}, and for multiloop algebras in \cite{lau:mult} (different proofs were given in \cite{NSS}).  The isomorphism classes of irreducible evaluation representations are naturally parameterized by finitely-supported equivariant maps on $X_\rat$ taking values in the set of isomorphism classes of irreducible $\g^x$-modules at a point $x \in X_\rat$.

Except in rather trivial cases, the category of finite-dimensional representations of an equivariant map algebra is not semisimple.  It is therefore important to know the extensions between irreducibles.  These have been described for current algebras in \cite{CG05} and for generalized current algebras in \cite{kodera}.  In the current paper, we address the question of computing extensions in the general setting of equivariant map algebras.  Precisely, we determine the extensions between irreducible finite-dimensional representations of equivariant map algebras where the scheme $X$ is of finite type and the Lie algebra $\g$ is reductive.

One of our main results is that the problem of computing extensions between evaluation representations can be reduced to the case of single point evaluation representations (i.e.\ the case where $\bx$ above is a singleton) at the same point (see Theorem~\ref{theo:two-eval}).  We then show that the space of extensions between these single point evaluation representations is equal to certain spaces of homomorphisms of $\g^x$-modules (see Theorem~\ref{theo:ext-eval-point}).  These results generalize formulas previously obtained in \cite{CG05,kodera,S}.

Equipped with formulas for the extensions between irreducible objects in the category of finite-dimensional representations, we are able to determine the block decomposition of this category.  In \cite{CM}, these blocks were described for loop algebras in terms of \emph{spectral characters}.  These results were then extended to the cases of twisted loop algebras in \cite{S} and generalized current algebras in \cite{kodera}.  In the current paper, we generalize the notion of spectral characters to the setting of arbitrary equivariant map algebras ($X$ affine of finite type, $\g$ reductive).  In keeping with the classification of irreducibles in terms of finitely-supported equivariant functions on $X$, in many cases the spectral characters are finitely-supported equivariant functions on $X$ taking values in certain quotients of the weight lattice of $\g^x$ at a point $x \in X_\rat$ (see Sections~\ref{sec:block-decomps} and~\ref{sec:applications}).

Our results recover all the known results on extensions and block decomposition for Lie algebras that can be viewed as equivariant map algebras.  However, in such cases, our method is quite different.  Existing proofs in the literature use the concept of a Weyl module -- something which is not currently available for arbitrary equivariant map algebras.  In contrast, our approach uses results on the cohomology of Lie algebras, most importantly the Hochschild-Serre spectral sequence -- a technique that goes back to the paper \cite{FM} which studies extensions between irreducible evaluation modules of the current algebra $M(k,\g)$, $\g$ simple.
Hence our results give new proofs in the cases where the extensions and block decompositions were known.  In addition, we can describe the extensions between irreducible finite-dimensional representations and block decompositions for classes of equivariant map algebras for which these were not previously known.  This is the case, for example, for multiloop algebras and generalized Onsager algebras.

The organization of this paper is as follows.  In Section~\ref{sec:EMAs} we recall the definition of equivariant map algebras and the classification of their irreducible finite-dimensional representations in terms of evaluation representations.  In Section~\ref{sec:extgeneral} we collect some facts about extensions between representations of Lie algebras in general, and the relation between spaces of extensions and Lie algebra cohomology.  We specialize our discussion to equivariant map algebras in Section~\ref{sec:EMA-ext}, where we prove some of our main results on extensions between irreducible finite-dimensional representations. In Section~\ref{sec:abelian} we consider the special case where the group $\Gamma$ is abelian, in which case we are able to make our descriptions of extensions more explicit.  We use our results on extensions to describe the blocks of the category of finite-dimensional representations of an equivariant map algebra in Section~\ref{sec:block-decomps}.  Finally, in Section~\ref{sec:applications}, we specialize our general results to certain equivariant map algebras of particular interest.  In an appendix, we prove some results relating extensions to the weight lattice of a semisimple Lie algebra.  This allows us, in some cases, to describe the block decomposition in terms of explicit quotients of the weight lattice.

\subsection*{Notation} Throughout, $k$ is an algebraically closed field of characteristic $0$ and all algebras and tensor products are over $k$.  We denote by $X = \Spec A$ the prime spectrum of a unital associative commutative finitely generated $k$-algebra $A$.   Equivalently, $X$ is an affine scheme of finite type.  A point $x \in X$ is called a \emph{rational point} if $A/\fm_x \cong k$, where $\fm_x$ is the ideal of $A$ corresponding to $x$, and we abbreviate the subset of rational points of $X$ by $X_\rat$.  Since $A$ is finitely generated, the rational points correspond exactly to the maximal ideals of $A$.  Hence $X_\rat = \maxSpec A$.

The direct product of two algebras $A$ and $B$ is denoted $A \boxplus B$ to distinguish it from the direct sum of vector spaces.  For a Lie algebra $L$, we denote by $L'=[L,L]$ the derived
subalgebra and let $L_\ab = L/L'$ be the \emph{abelianization} of $L$, cf.\ \cite{wei}.  Throughout, $\g$ will denote a finite-dimensional reductive Lie algebra with semisimple part $\g_\rss = \g'$.  We identify $\g_\ab$ with the center of $\g$, so that $\g = \g_\rss \boxplus \g_\ab$.  We will denote the root and weight lattices of $\g_\rss$ by $Q$ and $P$, respectively.  The set of dominant weights with respect to some set of positive roots will be denoted by $P^+$, and $V(\la)$ is the irreducible finite-dimensional $\g_\rss$-module with highest weight $\la\in P^+$.  By $L_\ab^*$, we mean $(L_\ab)^*$ (and similarly for expressions such as $\g_\ab^*$ and $\g_{0,\ab}^*$).  By the usual abuse of notation, we use the terms module and representation interchangeably.

For a finite group $\Gamma$ and a $\Gamma$-module $M$, we let $M^\Gamma = \{m \in M : \gamma \cdot m = m \ \forall\ \gamma \in \Gamma\}$ denote the set of elements of $M$ fixed by $\Gamma$.  Similarly, if $M$ is an $L$-module, we let $ M^L = \{m\in M : l\cdot m = 0 \hbox{ for all $l\in L$}\}$.  In case $M=\Hom_k(M_1, M_2)$ for two $L$-modules $M_1$, $M_2$, the $L$-module $M^L$ coincides with the $L$-module homomorphisms $M_1 \to M_2$, and we therefore sometimes also employ the notation $(\Hom_k(M_1, M_2))^L = \Hom_L(M_1, M_2)$.

\subsection*{Acknowledgements} The authors thank V.~Chari, G.~Fourier, S.~Kumar, and G.~Smith for useful discussions.  In particular, they thank V.~Chari for pointing out the reference \cite{FM} and S.~Kumar for formulating and proving Proposition~\ref{prop:kumar} and providing a more direct proof of \eqref{eq:dirprodH1} than their original argument.  The second author would also like to thank the Hausdorff Research Institute for Mathematics, the Institut de Math\'ematiques de Jussieu, and the D\'epartement de Math\'ematiques d'Orsay for their hospitality during his stays there, when some of the writing of the current paper took place.

%
\section{Equivariant map algebras and their irreducible representations} \label{sec:EMAs}
%

In this section, we review the definition of equivariant map algebras and the classification of their irreducible finite-dimensional representations given in \cite{NSS}.  We recall the standing assumptions of this paper: $X$ is an affine $k$-scheme with finitely generated coordinate algebra $k[X]=A$, $\g$ is a reductive Lie $k$-algebra, and $\Ga$ is a finite group acting on $X$ (equivalently, on $A$) and on $\g$ by automorphisms. Let $M(X,\g)$ be the Lie $k$-algebra of regular maps from $X$ to $\g$, which we will often identify with $\g \ot A$.  This is a Lie algebra under pointwise multiplication.  The \emph{equivariant map algebra} $\frM=M(X,\g)^\Ga$ is the subalgebra of $\Ga$-equivariant maps.  In other words, $\frM$ consists of the $\Ga$-fixed points of the canonical (diagonal) action of $\Ga$ on $M(X,\g)=\g \ot A$.

For $x\in X$, we let
\[
  \Ga_x = \{ \gamma \in \Ga : \gamma \cdot x=x\}
\]
be its isotropy group and put
\[
  \g^x= \{u\in \g: \gamma \cdot u = u \hbox{ for all } \gamma \in \Ga_x\}.
\]
Since $\g$ is reductive, it is known that all isotropy subalgebras $\g^x$ are reductive (\cite[VII, \S1.5, Prop.~14]{Bou75}).  We denote by $X_*$ the set of finite subsets
$\bx\subseteq X_\rat$ for which $\Ga \cdot x \cap \Ga \cdot x' =
\varnothing$ for distinct $x,x'\in \bx$. For $\bx\in X_*$ we define
$\g^\bx = \boxplus_{x\in \bx}\, \g^x$.
The evaluation map
\[
  \ev_\bx : \frM \to \g^\bx,\quad \ev_\bx (\al) =  (\al(x))_{x\in \bx},
\]
is a Lie algebra epimorphism \cite[Cor.~4.6]{NSS} and we set
\[
  \frK_\bx = \Ker \ev_\bx.
\]
To $\bx\in X_*$ and a set $\{\rho_x: x \in \bx\}$ of (nonzero) representations $\rh_x : \g^x \to \End_k V_x$, we associate the \emph{evaluation representation} $\ev_\bx(\rho_x)_{x\in \bx}$ of $\frM$, defined as the composition
\[ \textstyle
  \frM \xrightarrow{\ev_\bx} \g^\bx \xrightarrow{\bigotimes_{x \in \bx} \rho_x} \End_k \left( \bigotimes_{x \in \bx} V_x \right).
\]
If all $\rh_x, x\in \bx$, are irreducible finite-dimensional representations, then this is also an irreducible finite-dimensional representation of $\frM$, \cite[Prop.~4.9]{NSS}. In this paper, we will always implicitly assume that evaluation representations are finite-dimensional (i.e.~the $\rho_x$ are all finite-dimensional).  The \emph{support}
of an evaluation representation $V=\ev_\bx(\rho_x)_{x\in \bx}$, abbreviated $\Supp V$, is the union of all $\Gamma \cdot x$, $x \in \bx$, for which $\rh_x$ is not the one-dimensional trivial representation of $\g^x$.  In a slight abuse of terminology, we will sometimes refer to $V$ as both a representation of $\frM$ and of $\g^\bx$.

For $x\in X_\rat$, let $\mathcal{R}_x$ denote the set of isomorphism classes of irreducible finite-dimensional representations of $\g^x$, and put $\mathcal{R}_X=\bigsqcup_{x \in X_\rat} \mathcal{R}_x$. Then $\Gamma$ acts on $\mathcal{R}_X$ by
\[
  \Gamma \times \mathcal{R}_X \to \mathcal{R}_X,\quad (\gamma,[\rho]) \mapsto \gamma  \cdot [\rho] := [\rho \circ \gamma^{-1}] \in \mathcal{R}_{\gamma \cdot x},
\]
where $[\rho] \in \mathcal{R}_x$ denotes the isomorphism class of a
representation $\rho$ of $\g^x$. Let $\mathcal{E}$ denote the set of
finitely supported $\Gamma$-equivariant functions $\psi : X_\rat \to
\mathcal{R}_X$ such that $\psi(x) \in \mathcal{R}_x$. Here the
support $\Supp \psi$ of $\psi \in \mathcal{E}$ is the set of all
$x \in X_\rat$ for which $\psi(x) \ne 0$, where $0$ denotes the
isomorphism class of the trivial one-dimensional representation.

For isomorphic representations $\rho$ and $\rho'$ of $\g^x$, the
evaluation representations $\ev_x \rho$ and $\ev_x \rho'$ are
isomorphic. Therefore, for $[\rho] \in \mathcal{R}_x$, we can define
$\ev_x [\rho]$ to be the isomorphism class of $\ev_x \rho$, and this
is independent of the representative $\rho$. Similarly, for a finite
subset $\bx \subseteq X_\rat$ and representations $\rho_x$ of $\g^x$
for $x \in \bx$, we define $\ev_\bx ([\rho_x])_{x \in \bx}$ to be
the isomorphism class of $\ev_\bx (\rho_x)_{x \in \bx}$.

For $\psi \in \mathcal{E}$, we define $\ev_\psi = \ev_\bx
(\psi(x))_{x \in \bx}$ where $\bx \in X_*$ contains one element of
each $\Gamma$-orbit in $\Supp \psi$.  By
\cite[Lem.~4.12]{NSS}, $\ev_\psi$ is independent
of the choice of $\bx$.  If $\psi$ is the map that is identically 0
on $X$, we define $\ev_\psi$ to be the isomorphism class of the
trivial representation of $\frM$.  Thus $\psi \mapsto \ev_\psi$
defines a map $\mathcal{E} \to \scS$, where $\scS$ denotes the set
of isomorphism classes of irreducible finite-dimensional
representations of $\frM$.  This map is injective by \cite[Prop.~4.14]{NSS}.  In other words, $\mathcal{E}$ naturally enumerates the isomorphism classes of irreducible evaluation representations of $\frM$.  We say that an evaluation representation is a \emph{single orbit evaluation representation} if its isomorphism class is $\ev_\psi$ for some $\psi \in \mathcal{E}$ whose support is contained in a single $\Gamma$-orbit.

We recall that the one-dimensional representations of a Lie algebra $L$ can be identified with the elements of $L_\ab^* \cong \{ \la \in L^*: \la(L') = 0 \}$, where to such a $\la$ we associate the one-dimensional representation on $k\equiv k_\la$ defined by $l \cdot a = \la(l) a$ for $l\in L$ and $a\in k$.

\begin{prop}[{\cite[Th.~5.5]{NSS}}] \label{prop:irred-classification}
The map
\[
  \frM_\ab^* \times \mathcal{E} \to \scS,\quad
  (\lambda, \psi) \mapsto k_\lambda \otimes \ev_\psi,\quad
  \lambda \in \frM_\ab^*,\quad \psi \in \mathcal{E},
\]
is surjective.  In particular, all irreducible finite-dimensional
representations of $\frM$ are tensor products of an irreducible evaluation
representation and a one-dimensional representation.
\end{prop}

\begin{rems} \label{rem:irred-classification}
\begin{enumerate}
  \item In \cite[Th.~5.5]{NSS}, a condition on when pairs $(\lambda,\psi)$ and $(\lambda',\psi')$ correspond to the same representation is given, thus obtaining an analogue of Proposition~\ref{prop:irred-classification} where the map is bijective.  However, we will not need the stronger result in the current paper.

  \item \label{rem-item:unique-decomp} By \cite[Cor.~5.4]{NSS}, every irreducible finite-dimensional representation of $\frM$ can be written as $V_\rss \otimes k_\lambda$ for $V_\rss$ an evaluation representation (unique up to isomorphism) factoring through some $\g^\bx_\rss$ and unique $\lambda \in \frM_\ab^*$.

  \item The results of \cite{NSS} apply for an arbitrary finite-dimensional Lie algebra $\g$.  However, in the current paper, we restrict our attention to the case where $\g$ is reductive.
\end{enumerate}
\end{rems}

\begin{example}[Untwisted map algebras]
When the group $\Gamma$ is trivial, $M(X,\g)$ is called an \emph{untwisted map algebra}, or \emph{generalized current algebra}.  These algebras arise also for a nontrivial group $\Ga$ acting trivially on $\g$ or on $X$. In the first case we have $M(X,\g)^\Ga \cong M(\Spec(A^\Ga), \g)$, and in the second $M(X,\g)^\Ga = M(X,\g^\Ga)$.
\end{example}

\begin{example}[Multiloop algebras] \label{eg:multiloop}
Fix positive integers $n, m_1, \dots, m_n$.  Let
\[
  \Gamma = \langle \gamma_1,\dots, \gamma_n : \gamma_i^{m_i}=1,\ \gamma_i \gamma_j = \gamma_j \gamma_i,\ \forall\ 1 \le i,j \le n \rangle \cong (\Z/m_1\Z) \times \dots \times (\Z/m_n\Z)
\]
and suppose that $\Gamma$ acts on a semisimple $\g$. Note that this is
equivalent to specifying commuting automorphisms $\sigma_i$,
$i=1,\dots,n$, of $\g$ such that $\sigma_i^{m_i}=\Id$. For $i =
1,\dots, n$, let $\xi_i$ be a primitive $m_i$-th root of unity.
Let $X=(k^\times)^n$ and define an action of $\Gamma$ on $X$ by
\[
  \gamma_i \cdot (z_1, \dots, z_n) = (z_1, \dots, z_{i-1}, \xi_i z_i, z_{i+1}, \dots, z_n).
\]
Then
\begin{equation} \label{eq:twisted-multiloop-def}
    M(\g,\sigma_1,\dots,\sigma_n,m_1,\dots,m_n) := M(X,\g)^\Gamma
\end{equation}
is the \emph{multiloop algebra} of $\g$ relative to $(\sigma_1,
\dots, \sigma_n)$ and $(m_1, \ldots, m_n)$.  In this case, all irreducible finite-dimensional representations are evaluation representations (see \cite[Cor.~6.1]{NSS} or \cite{lau:mult}).
\end{example}

\begin{example}[$\Ga$ of order $2$] \label{eg:order-two}
Let $\frM= M(X,\g)^\Ga$ be an equivariant map algebra with $\g$ simple and $\Ga=\{1,\si\}$ of order $2$, acting nontrivially on $\g$.  Thus we have $\Z/2\Z$-gradings on $\g$ and $A$, denoted $\g=\g_0 \oplus \g_1$ and $A=A_0 \oplus A_1$ with $\g_0 = \g^\Ga$ and $A_0 = A^\Ga$. Hence
\[
 \frM = (\g_0 \ot A_0) \oplus (\g_1 \ot A_1).
\]
We will use the following facts regarding the structure of $\g$, for which the reader is referred to \cite[Ch.~X, \S5]{Hel01} and \cite[Exercise~8.9]{kac}.
\begin{enumerate}
  \item \label{eg-item:order-two:a} We have $\g_0 = [\g_1, \g_1]$, $\g_1 = [\g_0, \g_1]$ and $\g_0$ acts faithfully on $\g_1$ (all of these claims are immediate from simplicity of $\g$). \lv{Details: $\frb = \{ x\in \g_0 : [x,\g_1]=0\}$ is an ideal of $\g_0$ and satisfies $[\frb, \g_1]=0$, so it is an ideal of $\g$. Since $\g_1 \ne 0$, $\frb=0$ follows.}

  \item \label{eg-item:order-two:b} The Lie algebra $\g_0$ is reductive, so $\g_0 = \g_{0,\rss} \oplus \g_{0, \ab}$, with $\dim \g_{0,\ab} \le 1$.

  \item \label{eg-item:order-two:c} Suppose $\dim \g_{0,\ab} = 1$. Then $\g_1 = V_1 \oplus V_{-1}$ is a direct sum of two irreducible dual $\g_0$-modules $V_1$ and $V_{-1}$ with $\g_{0,\ab}$ acting on $V_{\pm 1}$ by $\pm \rho$ for some $0 \ne \rho \in \g_{0,\ab}^*$.  In particular, $[\g_{0,\ab},\g_1]=\g_1$.  Moreover, also $\g_{0,\rss}$ acts irreducibly on $V_{\pm 1}$, \lv{If $\g=\lsl_2$ then $\g_{0,\rss} = 0$, but the $V_{\pm 1}$ are $1$-dimensional, so the $0$-algebra acts irreducibly on them. So $\g_{0,\rss}$ always acts irreducibly on $V_{\pm 1}$.} and we have: $\g_1^{\g_{0,\rss}} = 0 \iff \g \ne \lsl_2(k)$. \lv{As stated in \cite{kac}, the case with $\g_{0,\ab}\ne 0$ occurs precisely when $\g=V_1 \oplus \g_0 \oplus V_{-1}$ is a short $\Z$-grading. Thus $V=(V_1, V_{-1})$ is a simple Jordan pair. In case $\lsl_2(k)$ we therefore have $\dim \g_0 = 1$, so $\g_0=\g_{0,\ab}$ and therefore $\g_1^{\g_{0,\rss}} = \g_1^{0} = \g_1 \ne 0$. In case $\g\ne \lsl_2(k)$ one knows that $\g_{0,\rss} \ne 0$, $\dim V_{-1}= \dim V_1 > 1$. Since $\g_{0,\ab}$ acts by a scalar on $V_{\pm 1}$, the action of $\g_{0,\rss}$ on $V_{\pm 1}$ is irreducible, otherwise $\g_0$ would not act irreducibly. It is not irreducible in case $\g_=\lsl_2$ since then $V_{\pm 1}$ are one-dimensional.}

  \item \label{eg-item:order-two:sl2} If $\g=\lsl_2(k)$, then $\si$ acts by a Chevalley involution and $\g_0=\g_{0,\ab}\ne 0$. So $\g_1 = V_1 \oplus V_{-1}$ as in \eqref{eg-item:order-two:c}.

  \item \label{eg-item:order-two:d} If $\g_0$ is semisimple, the $\g_0$-module $\g_1$ is irreducible.
\end{enumerate}
In particular, \eqref{eg-item:order-two:a} and \eqref{eg-item:order-two:b} imply
\[
  \frM' = (\g_{0,\rss} \ot A_0) \oplus (\g_{0,\ab} \ot A_1^2) \oplus (\g_1 \ot A_1),\quad \frM_\ab \cong \g_{0,\ab} \ot (A_0/A_1^2).
\]
It is easy to see that the fixed point set $X_\rat^\Ga = \{ x\in X_\rat : \si \cdot x= x\}$ has the following description,
\begin{equation} \label{eq:order-two:1}
  X_\rat^\Ga = \{ x\in X_\rat: A_1 \subseteq \fm_x\} = \{ x\in X_\rat: A_1^n \subseteq \fm_x\} \quad \text{for any } n \in   \NN_+,
\end{equation}
where $\fm_x$ is the maximal ideal of $A$ corresponding to $x$. \lv{ If $x\in X_\rat^\Ga$ then for $a_1 \in A_1$ we have $a_1(x) = a_1 (\si \cdot x) = (\si \cdot a_1)(x) = - a_1(x)$, so $A_1 \subseteq \fm_x$, and
then also $A_1^n \subseteq \fm_x$. Conversely, if $A_1^n \subseteq \fm_x$, then for $a_1 \in A_1$ we have $0 = (a_1)^n(x)$ so $a_1(x) = 0$ and $A_1 \subseteq \fm_x$ follows, whence as above $a(x) = a(\si \cdot x)$ for all $a\in A$. Thus $\si(x) = x$, since $A$ separates points.} Hence, if $\g_{0,\ab} \ne 0$, then $\frM$ has nontrivial one-dimensional representations if and only if $A_1^2 \subsetneq A_0$, which in turn is equivalent to $\Ga$ acting on $X$ with fixed points.  These nontrivial one-dimensional representations are in general not evaluation representations, see \cite[Ex.~5.21]{NSS}.  However, for the generalized Onsager algebras, which are special cases of the example here and which we review next, it turns out that all one-dimensional representations are in fact evaluation representations, see \cite[Prop.~6.2]{NSS}.
\end{example}

\begin{example}[Generalized Onsager algebras] \label{eg:Onsager}
Let $X=k^\times = \Spec k[t^{\pm 1}]$, $\g$ be a simple Lie algebra,
and $\Gamma=\{1,\sigma\}$ be a group of order $2$. We suppose that
$\sigma$ acts on $\g$ by an automorphism of order $2$ and on $
k[t^{\pm 1}]$ by $\si \cdot t = t^{-1}$, inducing an action of
$\Gamma$ on $X$. We define the \emph{generalized Onsager algebra} to be the equivariant map algebra $M(k^\times,\g)^\Gamma$ associated to these
data.  The term ``generalized Onsager algebra'' was used in
\cite[Ex.~3.9]{NSS} in a more restrictive way ($\sigma$ was supposed
to be a Chevalley involution), while the algebra above was
considered in \cite[Ex.~3.10]{NSS} without a name. We have chosen
the new definition since all the results proven in \cite{NSS} and
here are true for the more general notion.

For $k=\C$ and $\si$ acting by a Chevalley involution, it was shown
in \cite{Roan91} that $M(X,\lsl_2)^\Gamma$ is isomorphic to the usual
\emph{Onsager algebra}.
\end{example}

We will return to the above examples in Section~\ref{sec:applications}, where we apply our general results on extensions and block decompositions.

%
\section{Extensions and Lie algebra cohomology}
\label{sec:extgeneral}
%

Our aim in the current paper is to determine extensions between irreducible finite-dimen\-sio\-nal representations of equivariant map algebras.  One of our main tools for computing such extensions will be Lie algebra cohomology.  In this section, we recall some basic facts about extensions between modules for Lie algebras and collect some results on Lie algebra cohomology that will be used in the sequel.  Throughout this section, $L$ is an arbitrary Lie algebra over $k$, not necessarily of finite dimension.

We will use the following easy and well-known lemmas.  The second is a straightforward consequence of Schur's Lemma.

\begin{lem}\label{lem:geniso}
Let $M, N$ and $P$ be finite-dimensional
$L$-modules. The following canonical vector space isomorphisms are
in fact $L$-module isomorphisms:
\begin{gather*}
  M \ot N \cong N \ot M, \quad (M \ot N)^* \cong M^* \ot N^*,\\
  N \cong N^{**},\quad \Hom_k(M,N) \cong M^* \ot N, \\
  \begin{split}
    \Hom_k(M\ot N,P) \cong \Hom_k(M, &N^* \ot P) \\
    &\cong \Hom_k(M \ot P^*, N^*) \cong M^* \ot \Hom_k(N,P).
  \end{split}
\end{gather*}
\end{lem}

\begin{lem}\label{lem:schur-tensor}
Let $M$ and $N$ be irreducible finite-dimensional $L$-modules where $L$ is an arbitrary Lie algebra. Then $\dim (M \ot N^*)^L \le 1$, and $\dim_k (M \ot N^*)^L = 1 \iff M \cong N$.
\end{lem}

Extensions of a Lie algebra $L$ can be described in terms of the first cohomology group $\rmH^1(L, V)$, for an $L$-module $V$, as we now describe.  We first recall the well-known fact, see for example \cite[Th.~7.4.7]{wei}, that
\begin{equation} \label{eq:H-Der}
  \rmH^1(L,V) \cong \Der(L,V) / \IDer(L,V),
\end{equation}
where
\[
  \Der(L, V) = \{ \partial \in \Hom_k(L,V) : \partial([l_1,l_2]) = l_1 \cdot \partial(l_2) - l_2 \cdot
  \partial(l_1) \ \forall\ l_1,l_2 \in L\}
\]
denotes the space of all derivations from $L$ to $V$ and
\[
  \IDer(L,V) = \{ \partial_v : v \in V\}, \text{ where } \partial_v(l) = l \cdot v\ \forall\ l \in L,
\]
is the subspace of inner derivations.\lv{
Recall that any $v\in V$ gives rise to an inner derivation $\pa_v$
defined by $\pa_v(l)= l \cdot v$ for $l\in L$.} The obvious maps
give rise to an exact sequence of $L$-modules:
\begin{equation} \label{eq:trivex}
  0 \to V^L \to V \to \IDer(L,V) \to 0.
\end{equation}
For example, if $V$ is a trivial $L$-module, then $\IDer(L,V) =
\{0\}$, $\Der(L,V) = \{ \pa\in \Hom_k(L, V): \pa(L')=0\}$ and
hence (\cite[Cor.~7.4.8]{wei})
\begin{equation} \label{eq:triv}
  \rmH^1(L, V) \cong \Hom_k(L_\ab, V) \qquad (\text{$V$ a trivial $L$-module}).
\end{equation}

The set $\Ext^1_L(V_1, V_2)$ of equivalence classes of extensions of $V_1$ by $V_2$ is in bijection with the first cohomology group of $L$ with coefficients in the $L$-module $\Hom_k(V_1, V_2)$ (see \cite[Expos\'{e}~4]{solie}, \cite[Ch.~1, \S4.5]{fuks}, or \cite[Exercise~7.4.5]{wei}):
\begin{equation} \label{eq:ext-homology}
  \Ext^1_L(V_1,V_2) \cong \rmH^1\big(L, \Hom_k(V_1,V_2)\big) \cong \rmH^1(L, V_1^* \otimes V_2),
\end{equation}
where in the second isomorphism we assume that $V_1$ and $V_2$ are finite-dimensional.
The first isomorphism is induced by assigning to the derivation $\pa : L \to \Hom_k(V_1, V_2)$ the extension $V_2 \hookrightarrow U \twoheadrightarrow V_1$, where $U=V_1 \oplus V_2$ with $L$-module structure given by $l \cdot (v_1 \oplus v_2) = (l \cdot v_1) \oplus (\pa(l)(v_1) + l \cdot v_2)$ for $l\in L$, $v_i \in V_i$, and where $V_2 \hookrightarrow U$ and $U \twoheadrightarrow V_1$ are the obvious maps.

Combining \eqref{eq:ext-homology} with Lemma~\ref{lem:geniso} yields the following.
\begin{cor} \label{cor:ext-moving-arguments}
For finite-dimensional $L$-modules $M,N,P$ we have
\begin{align*}
  \Ext_L^1(M\ot N, P) \cong \Ext^1_L(M, N^* \ot P )\cong \Ext^1_L(M \ot P^*, N^*).
\end{align*}\end{cor}
\lv{\begin{proof} This is immediate from Lemma~\ref{lem:geniso} and \eqref{eq:ext-homology}. \qed
\end{proof} }

\begin{lem} \label{lem:corone}
Suppose $L$ has a one-dimensional representation given by $\la \in L_\ab^*$.  Then
\begin{equation} \label{eq:cohom-one-dim-rep}
  \rmH^1(L, k_\la) \cong (\frK_\la / \frD_\la)^*,
\end{equation}
where
\[
  \frK_\la = \Ker \la \quad   \hbox{and} \quad \frD_\la  = \Span\, \{ \la(l)u - [l,u] : l\in L, u\in \frK_\la \} \subseteq \frK_\la.
\]
Furthermore, we have the following.
\begin{enumerate}
  \item \label{lem-item:coronea} If $\la=0$, then $\frD_\la = L' \ideal L=\frK_\la$ and so $\rmH^1(L, k_0) \cong L_\ab^*$.

  \item \label{lem-item:coroneb} If $\la \ne 0$, let $z\in L$ satisfy $\la(z) = 1$. Then
  \begin{equation} \label{eq:corone1}
    \frD_\la = \frK_\la' + \{ u - [z,u] : u\in \frK_\la\} \ideal L.
  \end{equation}
\end{enumerate}
\end{lem}

\begin{proof}
For $\la = 0$ we have $\frD_\la = L'$, so part \eqref{lem-item:coronea} follows from \eqref{eq:triv}.
We therefore assume $\la \ne 0$ and prove the result using \eqref{eq:H-Der}. First, one easily verifies that $\IDer(L,k_\lambda) = k \la$ and that a linear map $\delta : L \to k_\lambda$ is a derivation if and only if $\delta(\frD_\lambda)=0$.
\lv{
First suppose $\delta : L \to k_\lambda$ is a derivation.  Then for $l \in L$, $u \in \frK_\lambda$, we have
\[
  \delta(\lambda(l)u - [l,u]) = \lambda(l) \delta(u) - \lambda(l)\delta(u) + \lambda(u) \delta(l) = 0.
\]
Hence $\delta(\frD_\lambda)=0$.

Now suppose $\delta(\frD_\lambda)=0$.  Then, as above, we have
\[
  \delta([l,u]) = \lambda(l)\delta(u) = \lambda(l) \delta(u) - \lambda(u) \delta(l) \text{ for } l \in L,\ u \in \frK_\lambda.
\]
By antisymmetry and linearity, it remains to verify that
\[
  \delta([l,l]) = \lambda(l)\delta(l) - \lambda(l) \delta(l)
\]
for some $l \in L$ with $\lambda(l)=1$.  But this is clearly satisfied.  Hence $\delta$ is a derivation, completing the proof of our claim.
}

Now fix $z\in L$ with $\la(z)=1$.  Then any $\de \in \Der(L,k_\lambda)$ can be written in the form $\de= t \la + \de_0$ with $t\in k$ and $\de_0(z) = 0$, and we can identify $\de_0$ with the restriction of $\delta$ to $\frK_\la$.  Equation~\eqref{eq:cohom-one-dim-rep} now follows from \eqref{eq:H-Der}.

For the proof of \eqref{lem-item:coroneb}, note that any $l \in L$ can be written in the form $l=tz + y$ with $t\in k$ and $y\in \frK_\la$.  Then, for $u \in \frK_\lambda$, we have $\la(l)u -[l,u] = (tu-[z,tu]) - [y,u]$, and so $\frD_\la $ has the form claimed in \eqref{eq:corone1}.  It is an ideal because, for $l,l'\in L$ and $u \in \frK_\lambda$, we have
\[
  [l',\la(l) u -[l,u]] = (\la(l)[l',u] -[l,[l',u]]) -[[l',l],u]\in \frD_\la,
\]
since $\frK_\la \ideal L$, and therefore $[l',u]\in \frK_\la$, and since $[[l',l],u] \in \frK_\la' \subseteq \frD_\la$. 
\lv{Proof using Proposition ~\ref{lem:codimone} below:\\
 Let $K = \Ker \lambda$. Then $K$ is an ideal of
codimension one and $(k_\lambda)^L = 0$. Hence, by
Proposition~\ref{lem:codimone}, it suffices to show $\rmH^1(K,
k_\la)^{L/K} \cong (K/L')^*$. As a $K$-module, $k_\la$ is
trivial. By \eqref{eq:triv}, we therefore have $\rmH^1(K,
k_\la)^{L/K} \cong \Hom_k(K/K', k_\la)^{L/K}$. The action of
$\bar l\in L/K$ on $\vphi \in \Hom_k(K/K', k_\la)^{L/K}$ is given
by $(\bar l \cdot \vphi)(\bar u) = \la(l)\bar u -
\vphi(\overline{[l,u]})$ for $l \in L$, $u \in K$. Hence $\bar l
\cdot \vphi = 0 $ for all $l\in L \iff \vphi(\frD_\la)=0$. } \qed
\end{proof}

With regards to extensions of one-dimensional modules by one-dimensional modules,
\[
  0 \to k_\mu  \to V \to k_\la \to 0,
\]
Lemma~\ref{lem:corone} says
\begin{equation} \label{eq:1dim-exts0}
  \Ext_L^1(k_\la, k_\mu) \cong (\frK_{\mu-\lambda}/\frD_{\mu - \lambda})^*.
\end{equation}

\begin{cor} \label{cor:abelian-1d-exts}
If $L$ is an abelian Lie algebra and $k_\lambda, k_\mu$ are two one-dimensional representations, then
\[
  \Ext^1_L (k_\lambda,k_\mu) =
  \begin{cases}
    0 & \text{if } \lambda \ne \mu, \\
    L^* & \text{if } \lambda = \mu.
  \end{cases}
\]
\end{cor}

\begin{proof}
If $\lambda \ne \mu$, it follows easily from \eqref{eq:corone1} that $\frD_{\mu-\lambda} = \frK_{\mu-\lambda}$ and the result is a consequence of~\eqref{eq:1dim-exts0}.  If $\lambda = \mu$, the result is simply Lemma~\ref{lem:corone}\eqref{lem-item:coronea}. \qed
\end{proof}

To calculate some other cohomology groups of interest here, we will
use the exact sequence of low-degree terms of the Hochschild-Serre
spectral sequence (\cite[Th.~6]{hoch-serre}, see also
\cite[p.~233]{wei}):
\begin{equation} \label{hose}
  0 \to \rmH^1(L/K, V^K) \xrightarrow{\inf} \rmH^1(L,V)
  \xrightarrow{\res} \rmH^1(K,V)^{L/K}
  \xrightarrow{t} \rmH^2(L/K, V^K) \xrightarrow{\inf} \rmH^2(L,V)
\end{equation}
whose ingredients we now explain. In this exact sequence $K
\ideal L$ is an ideal of $L$, $V$ is an $L$-module and $V^K $
is considered as $L/K$-module with the induced action. The inflation
map $\inf$ is induced by mapping a derivation $\pa : L/K \to V^K$ to
$\iota \circ \pa \circ \pi$, where $\pi : L \twoheadrightarrow L/K$ is the canonical quotient
map and $\iota : V^K \hookrightarrow V$ is the injection.  The map $\res$ is given by restriction, and the transgression
map $t$ is induced by the differential defining cohomology. The Lie
algebra $L$ acts on $\Der(K, V)$ in the obvious way, such that
$\IDer(K,V)$ is an $L$-submodule.\lv{Namely, $(l \cdot d)(x) = l
\cdot d(x) - d( [l, x])$ for $l\in L, d\in \Der(K,V)$ and $x\in
K$. The subspace $\IDer(K,V)$ is an $L$-submodule since $l\cdot
\pa_v = \pa_{l \cdot v}$ for $\pa_v\in \IDer(K,V)$.} Hence $L$ acts
on the quotient $\Der(K,V)/ \IDer(K,V)= \rmH^1(K,V)$. The action
of $K$ on $\rmH^1(K,V)$ is trivial, so that the action of $L$
factors through $L/K$. \lv{The formula $x \cdot d = \pa_{d(x)}$ for
$x\in K$ and $d\in \Der(K, V)$ shows that $K$ acts trivially on
$\rmH^1(K,V)$. Hence $L/K$ acts on $\rmH^1(K,V)$.}

\begin{prop}\label{prop:feiginprop}
Let $\rho : L \to \End_k V$ be a finite-dimensional representation of
$L$, and let $K \subseteq \Ker \rho$ be an ideal such that $\frl=L/K$
is finite-dimensional reductive.
\begin{enumerate}
  \item \label{prop-item:feigenprop:completely-reducible} Suppose that either $\frl$ is semisimple, or $\rho$ is completely
    reducible with $\rho(z)$ invertible for some $z\in \frl_{\rm ab}$.
    Then
    \begin{equation} \label{feigin1}
     \rmH^1(L, V) \cong \Hom_\frl(K_\ab, V),
    \end{equation}
    the isomorphism being induced by restricting a derivation $\pa : L
    \to V$ to $K$.

  \item \label{prop-item:feigenprop:gab-zero-action} If $\frl_{\rm ab}\cdot V=0$, we have
  \begin{equation} \label{eq:redH1}
    \rmH^1(\frl, V) \cong \Hom_k(\frl_{\rm ab}, V^{\frl_\rss}),
  \end{equation}
  induced by restriction, and an exact sequence
  \begin{equation} \label{eq:redH2}
    0 \to \Hom_k(\frl_{\rm ab}, V^{\frl_\rss}) \xrightarrow{\inf} \rmH^1(L,V)
    \xrightarrow{\res} \Hom_\frl(K_\ab,V) \xrightarrow{t}
    \rmH^2(\frl, V) \to \cdots
  \end{equation}
\end{enumerate}
\end{prop}

\begin{proof}
\eqref{prop-item:feigenprop:completely-reducible}  Since $V^K=V$, hence $\rmH^1(K,V) \cong \Hom_k(K_\ab, V)$ by \eqref{eq:triv}, the claim is immediate from the Hochschild-Serre spectral sequence \eqref{hose}
  \lv{$$ \xymatrix@C=20pt{
  0 \ar[r] & \rmH^1(\frl, V) \ar[r]^>>>>{\inf} & \rmH^1(L,V)
  \ar[r]^>>>>{\res} &
  \rmH^1(K,V)^{\frl} \ar[r]^d & \rmH^2(\frl, V) \ar[r] &\cdots }
  $$}
as soon as $\rmH^1(\frl, V) = 0 = \rmH^2(\frl, V)$. That these last two equations hold if $\frl$ semisimple, is the assertion of Whitehead's Lemmas (see for example \cite[Cor.~7.8.10 and Cor.~7.8.12]{wei}). But it is also known that they hold in the case that $\rho$ is completely reducible and $\rho(z)$ is invertible for some $z \in \frl_\ab$ (\cite[Th.~10]{hoch-serre}, or see \cite[\S3, Exercise~12j]{Bou:lie}). \lv{We apply the quoted exercise, second part, with $\frl=\frl$ and $\frh=k z$. Since $\rmH^1(\frl,V)$ is exact in $V$, we can assume that $V$ is irreducible. }

\eqref{prop-item:feigenprop:gab-zero-action} To prove \eqref{eq:redH1} we use $\rmH^1(\frl, V) \cong \Der(\frl, V) /\IDer(\frl, V)$ by \eqref{eq:H-Der}.  For a linear map $\pa: \frl \to V$, let $\pa_\rss$ and $\pa_\ab$ denote the restriction to $\frl_\rss$ and $\frl_\ab$ respectively.  Since $\frl_\ab \cdot V =0$, $\pa$ is a derivation if and only if $\pa_\rss$ is a derivation and $\pa_\ab \in \Hom_k (\frl_\ab, V^{\frl_\rss})$.  Hence, $\pa \mapsto \pa_\ab$ is a well-defined linear map $\Der(\frl, V) \to \Hom_k(\frl_\ab, V^{\frl_\rss})$. It is surjective since any linear map $f: \frl_\ab \to V^{\frl_\rss}$ extends to a derivation $\pa : \frl \to V$ with $\pa_\rss = 0$. Its kernel is $\IDer(\frl,V)$ because for $\pa \in \Der(\frl,V)$ we have $\pa_\ab =0 \iff \pa = \pa_\rss \in \Der(\frl_\rss, V) = \IDer(\frl_\rss, V)$ and because the map $\IDer(\frl,V) \to \IDer(\frl_\rss,V)$, $\pa \mapsto \pa_\rss$, is an isomorphism. The exact sequence \eqref{eq:redH2} is the Hochschild-Serre sequence \eqref{hose} for $(L,K)$, using $K \cdot V= 0$, the isomorphism \eqref{eq:redH1}, and $\rmH^1(K,V)^\frl \cong \Hom_\frl(K_\ab,V)$ by \eqref{eq:triv}. \qed
\end{proof}

We conclude this section with a discussion of the extensions for Lie algebras that can be decomposed as direct sums.

\begin{prop} \label{prop:ext-direct-sums}
Let $L_1$, $L_2$ be Lie algebras. We denote by $\scF_1, \scF_2$ and $\scF$ the category of finite-dimensional representations of $L_1$, $L_2$ and $L= L_1 \boxplus L_2$ respectively.
\begin{enumerate}
  \item \label{prop-item:direct-rep} Every module $V$ in $\scF$ is a tensor product $V= V_1 \ot V_2$, where $V_i$, $i=1,2$, are modules in $\scF_i$, uniquely determined up to isomorphism. The module $V$ is irreducible if and only if the $V_i$ are irreducible.

  \item \label{prop-item:kuenneth} Let $U_i, V_i$ be modules in $\scF_i$ for $i=1,2$. Then
    \begin{align*}
      \Ext^1_L (U_1 \ot U_2, V_1 \ot V_2) &\cong \big( (U_1^* \ot V_1)^{L_1} \ot \Ext_{L_2}(U_2, V_2) \big) \\ &\qquad \oplus \big( \Ext^1_{L_1} (U_1, V_1) \ot (U_2^* \ot V_2)^{L_2}\big).
    \end{align*}
    In particular, if $U_i, V_i$, $i=1,2$, are irreducible, then
    \begin{align*}
      \Ext^1_L (U_1 &\otimes U_2, V_1 \otimes V_2) \\
      &\cong
      \begin{cases}
        0 & \text{if } U_1 \not \cong V_1,\ U_2 \not \cong V_2, \\
        \Ext^1_{L_2}(U_2,V_2) & \text{if } U_1 \cong V_1,\ U_2 \not \cong V_2, \\
        \Ext^1_{L_1}(U_1,V_1) & \text{if } U_1 \not \cong V_1,\ U_2 \cong V_2, \\
        \Ext^1_{L_1}(U_1,V_1) \oplus \Ext^1_{L_2}(U_2,V_2) & \text{if }
              U_1 \cong V_1,\ U_2 \cong V_2.
      \end{cases}
    \end{align*}
\end{enumerate}
\end{prop}

We note that the formula in \eqref{prop-item:kuenneth} is mentioned
in \cite[Prop.~2]{CFK} for finite-dimensional Lie algebras and attributed to S.~Kumar.

\begin{proof}
Part \eqref{prop-item:direct-rep} is well known,  see for example \cite[Prop.~1.1]{NSS}. For the proof of \eqref{prop-item:kuenneth} we use
\eqref{eq:ext-homology} to rewrite the left hand side as $\rmH^1(L,
M_1 \ot M_2)$ for $M_1 = U_1^*\ot V_1$ and $M_2= U_2^* \ot V_2$.
Assuming for a moment the formula
\begin{equation} \label{eq:dirprodH1}
  \rmH^1(L, M_1 \ot M_2) \cong \big( M_1^{L_1} \ot \rmH^1(L_2, M_2) \big) \oplus \big( \rmH^1(L_1 , M_1) \ot M_2^{L_2}\big),
\end{equation}
we obtain the first formula in \eqref{prop-item:kuenneth} by another
application of \eqref{eq:ext-homology}. The second then follows from
Lemma~\ref{lem:schur-tensor}.

We give a proof of \eqref{eq:dirprodH1} due to S.~Kumar. First recall that $M_i^{L_i} = \rmH^0(L_i, M_i)$ by definition. Also, one knows \cite[3.1.9, 3.1.13]{kumar} that $\rmH^p(K,V) \cong \rmH^p(K, V^{**}) \cong (\rmH_p(K, V^*))^*$ for any finite-dimensional module of a Lie algebra $K$, relating the cohomology groups $\rmH^p$ with the homology groups $\rmH_p$. The advantage of homology is that it satisfies the K\"unneth formula
\[ \textstyle
  \bigoplus_{p+q=r} \rmH_p(L_1, M_1) \ot \rmH_q(L_2, M_2) \cong \rmH_r(L_1 \boxplus L_2, M_1 \ot M_2) \quad \text{for} \quad r \in \NN.
\]
Finally, we recall $(V_1 \ot V_1)^* \cong V_1^* \ot V_2^*$ if one of the vector spaces $V_i$ is finite-dimensional. With these tools at hand, we can now prove \eqref{eq:dirprodH1}:
\begin{align*}
  \rmH^1(&L, M_1 \ot M_2) \cong \rmH^1(L, (M_1^* \ot M_2^*)^*) \cong \big( \rmH_1(L, M_1^* \ot M_2^*)\big)^* \\
  &\cong \Big( \big(\rmH_0(L_1, M_1^*) \ot \rmH_1(L_2, M_2^*) \big) \oplus \big(\rmH_1(L_1, M_1^*) \ot \rmH_0(L_2, M_2^*) \big)\Big)^* \\
  &\cong \Big( \big(\rmH_0(L_1, M_1^*)\big)^* \ot \big(\rmH_1(L_2, M_2^*)\big)^*  \Big) \oplus \Big( \big(\rmH_1(L_1, M_1^*)\big)^* \ot \big(\rmH_0(L_2, M_2^*)\big)^* \Big) \\
  &\cong \big(\rmH^0(L_1, M_1) \ot \rmH^1(L_2, M_2) \big)  \oplus \big( \rmH^1(L_1, M_1) \ot \rmH^0(L_2, M_2) \big). \qquad \qquad \qed
\end{align*}
\end{proof}

%
\section{Extensions for equivariant map algebras} \label{sec:EMA-ext}
%

In this section we describe extensions of irreducible finite-dimensional modules of an equivariant map algebra $\frM=M(X,\g)^\Gamma$.  The reader is reminded that $X$ is an affine scheme of finite type (equivalently, $A$ is a finitely generated algebra) and $\g$ is a finite-dimensional reductive Lie algebra.

Let $R = k^n$ (as $k$-algebras) and let $\varepsilon_i$ be the element $(0,\dots,0,1,0,\dots,0)$, where the 1 appears in the $i$-th position.  Any $R$-module $M$ is
canonically a direct sum of $n$ uniquely determined submodules,
\[
  M= M_1 \oplus \cdots \oplus M_n, \quad M_i= \eps_i M,
\]
satisfying $\eps_i M_j = 0$ for $i\ne j$. Thus every $M_i$ is also a
$k$-vector space by identifying $k$ with the $i$th coordinate
subalgebra of $R$. Conversely, any direct sum $M=M_1 \oplus \cdots
\oplus M_n$ of $k$-vector spaces $M_i$ gives rise to an $R$-module
structure on $M$ by defining the action of the scalars in $R$ in the
obvious way.

The description of $R$-modules immediately extends to the category
of $R$-algebras. An $R$-algebra $\frl$ is naturally a direct product of
ideals, say $\frl = \boxplus_{i=1}^n \frl_i$, where each $\frl_i=\eps_i\frl$
is a $k$-algebra. Conversely, any direct product $\frl=\boxplus_{i=1}^n\frl_i$ of $k$-algebras $\frl_i$ can canonically be considered as an $R$-algebra.

Recall that a module of a Lie $R$-algebra $L$ is an $R$-module $M$
together with an $R$-bilinear map $L \times M \to M$, $(l,m) \mapsto
l \cdot m$ satisfying the usual rule for a Lie algebra action,
namely $[l_1, l_2] \cdot m = l_1 \cdot (l_2 \cdot m) - l_2 \cdot
(l_1 \cdot m)$ for $l_i \in L$ and $m\in M$. The following lemma is
immediate from the above.

\begin{lem} \label{lem:RLiealg}
Let $\frl=\frl_1\boxplus \cdots \boxplus \frl_n$ be a direct product of Lie $k$-algebras
$\frl_i$, $1\le i \le n$. As explained above we can view the $k$-algebra $\frl$ as an $R$-algebra for $R =
k^n$.

Every module of the Lie $R$-algebra $\frl$ is a direct sum of
uniquely determined $\frl$-submodules $M_i=\eps_i M$ such that $\frl_i \cdot M_j = 0$ for $i\ne j$. Conversely, given
$\frl_i$-modules $M_i$, $1\le i \le n$, the direct sum $M=M_1 \oplus
\cdots \oplus M_n$ becomes a module of the $R$-algebra $\frl$ with
respect to the obvious operations. \end{lem}

\begin{lem} \label{lem:noeth}
The fixed point subalgebra $A^\Ga$ is a finitely generated, hence Noetherian, $k$-algebra. Similarly, $A$ and $\frM$ are finitely generated, hence Noetherian, $A^\Ga$-modules.
\end{lem}

\begin{proof}
Since $A$ is a finitely generated $k$-algebra, so is $A^\Ga$ (\cite[V, \S1.9, Th.~2]{bou:ACb}).  Hence $A^\Ga$ is a Noetherian $k$-algebra.  Moreover, the same reference also shows that $A$ is a finitely generated $A^\Ga$-module, and hence Noetherian. Thus $\g\ot A$ is a finitely generated, hence Noetherian, $A^\Ga$-module. \lv{Here we use that a finitely generated module over a Noetherian ring is Noetherian (see [Eisenbud, I, Prop. 1.4]).}  But every submodule of a Noetherian module is again Noetherian. \qed
\end{proof}

Since $\frK_\bx$ acts trivially on any evaluation representation supported on $\Gamma \cdot \bx$, any extension between two evaluation representations supported on $\Gamma \cdot \bx$ will factor through $\frM/\frK_\bx'$.  It is therefore helpful to know the structure of this quotient.

\begin{prop}[{Structure of $\frM/\frK_\bx'$}] \label{prop:mdec}
For $\bx \in X_*$, define
\[
  I_\bx = \{a \in A : a\big( \ts \bigcup_{x\in \bx} \Ga \cdot x\big) = 0 \},\
  I_\bx^\Ga = I_\bx \cap A^\Ga, \
  \frN_\bx = \{ \nu \in \frK_\bx : I_\bx^\Ga \nu \subseteq \frK_\bx'\}.
\]
Then
\[
   \frK_\bx'  \ideal \frN_\bx \ideal \frK_\bx \ideal \frM
\]
is a sequence of ideals of the $A^\Ga$-algebra $\frM$.  The
quotient Lie algebra $\frM/ \frK_\bx'$ has the following structure.
\begin{enumerate}
  \item $\frK_{\bx,\ab} = \frK_\bx/\frK_\bx'$ is an abelian ideal of $\frM/\frK_\bx'$ and the quotient
      \[
        (\frM/\frK_\bx')\big/ (\frK_{\bx,\ab}) \cong \frM/\frK_\bx \cong \g^\bx.
      \]

  \item The adjoint representation of $\frM$ induces
  $\frg^\bx$-module structures on the quotients $\frM/\frK_\bx$ and
  $\frK_\bx/\frN_\bx$, and on the ideal $\frN_\bx/\frK_\bx'$ of $\frM/\frK_\bx'$.  In particular,
  \begin{enumerate}
    \item $\g^\bx$ acts on $\frM/\frK_\bx\cong \g^\bx$ by the adjoint representation
    and on $\frK_\bx/\frN_\bx$ by zero, and

    \item  $\frN_\bx/\frK_\bx'= \bigoplus_{x\in \bx} M^x$, where each $M^x$ is a
    finite-dimensional $\g^x$-module and $\g^x \cdot M^y=0$ for $x\ne y$.
  \end{enumerate}
\end{enumerate}
\end{prop}

\begin{proof}
The set $I_\bx$ is a $\Ga$-invariant ideal of $A$. Hence $I_\bx^\Ga \ideal A^\Ga$. The
algebra $A^\Ga$ acts naturally on $\frM$, and both $\frK_\bx$ and
$\frK_\bx'$ are clearly $A^\Ga$-submodules as well as ideals of $\frM$.
Moreover, the same is true for $\frN_\bx$: If $\al \in \frM$ and $\nu
\in \frN_\bx$ then $I_\bx^\Ga[\al, \nu] = [ \al, I_\bx^\Ga \nu] \subseteq [\frM,
\frK_\bx'] \subseteq \frK_\bx'$.  We thus have a chain $\frK_\bx'\ideal \frN_\bx
\ideal \frK_\bx$ of $A^\Ga$-invariant ideals of $\frM$, and consequently
an exact sequence
\begin{equation}
  0 \to \frN_\bx/\frK_\bx' \to \frK_{\bx,\ab} \to \frK_\bx/\frN_\bx \to 0
\end{equation}
of $\frM$-modules, each annihilated by $\frK_\bx$, i.e., an exact
sequence of $\frM/\frK_\bx \cong \g^\bx$-modules. We will analyze this
sequence further. First, since $I_\bx^\Ga\frM \subseteq \frK_\bx$, we have
$I_\bx^\Ga[\frM, \frK_\bx] = [I_\bx^\Ga \frM, \frK_\bx] \subseteq \frK_\bx'$, i.e.,
$[\frM, \frK_\bx] \subseteq \frN_\bx$. But this says that $\g^\bx \cong
\frM/\frK_\bx$ acts trivially on $\frK_\bx/\frN_\bx$.

By construction $\frN_\bx/\frK_\bx'$ and
$\frM/\frK_\bx$ are annihilated by $I_\bx^\Ga$, thus $\frN_\bx/\frK_\bx'$ is a
module of the $A^\Ga/I_\bx^\Ga$-algebra $\frM/\frK_\bx$. Since $A^\Ga/I_\bx^\Ga
\cong k^{|\bx|}$ (direct product of algebras), Lemma~\ref{lem:RLiealg}
implies that $\frN_\bx/\frK_\bx'$ is a direct sum $\frN_\bx/\frK_\bx' =
\bigoplus_{x\in \bx} M^x$, where each $M^x$ is a $\g^x$-module and
$\g^x\cdot M^y = 0$ for $x \ne y$.  Also, since $\frN_\bx$ is an
$A^\Ga$-submodule of the Noetherian $A^\Ga$-module $\frM$
(Lemma~\ref{lem:noeth}), $\frN_\bx$ is a finitely generated
$A^\Ga$-module. Hence $\frN_\bx/\frK_\bx'$ is a finitely generated
$k^{|\bx|}$-module, i.e., the $\g^x$-modules $M^x$ are all
finite-dimensional over $k$. \qed
\end{proof}

\begin{example}\label{ex:J}
Consider $\frM$ as in Example~\ref{eg:order-two}.  Let $I=I_\bx = I_0 \oplus I_1$ where $I_j = I \cap A_j$ for $j = 0,1$.  Then
\begin{align*}
  \frK_\bx &= (\g_{0,\rss} \ot I_0) \oplus (\g_{0,\ab} \ot I_0 ) \oplus (\g_1 \ot I_1), \\
  \frK'_\bx &= (\g_{0,\rss} \ot (I_0^2 + I_1^2)) \oplus (\g_{0,\ab} \ot I_1^2) \oplus (\g_1 \ot I_0 I_1), \quad \text{and} \\
  \frN_\bx & = (\g_{0,\rss} \ot I_0) \oplus (\g_{0,\ab} \ot J) \oplus (\g_1 \ot I_1), \quad \text{where} \\
  J &= \{ a\in I_0: aI_0 \subseteq I_1^2\}.
\end{align*}
Using the above, one can easily construct examples showing that $\frK'_\bx \subsetneq \frN_\bx \subsetneq \frK_\bx$  in general, even examples where $\frK_{\bx,\ab}$ is infinite-dimensional.  We will use the precise structure of $\frK_{\bx,\ab}$ in the proof of Proposition~\ref{prop:disjoint-support}.
\end{example}

The following well-known lemma describes the irreducible finite-dimensional representations of reductive Lie algebras.

\begin{lem}\label{lem:redirrep}
Any irreducible finite-dimensional representation of a finite-dimen\-sional reductive Lie
algebra $\g$ is a tensor product $V_\rss \ot V_\ab$ where
$V_\rss$ is an irreducible representation of the semisimple part
$\g_\rss$ of $\g$ and where $V_\ab$ is an
irreducible, hence one-dimensional, representation of the centre
$\g_\ab$ of $\g$. Equivalently, an irreducible representation of
$\g$ is an irreducible representation of $\g_\rss$, on which $\g_\ab$
acts by some linear form.
\end{lem}

\begin{prop}[Evaluation representations with disjoint support] \label{prop:disjoint-support}
If $V_1$ and $V_2$ are nontrivial irreducible evaluation representations with disjoint support, then $\Ext^1(V_1,V_2)=0$.
\end{prop}

In the case $\Ga=\{1\}$ (so $\frM = \g \ot A$) and $\g$ is semisimple, this result is proven in \cite[Lem.~3.3]{kodera} using the theory of Weyl modules, a technique which is not currently available for arbitrary equivariant map algebras.

\begin{proof}
Choose $\bx_i \in X_*$, $i=1,2$, containing one point in each $\Gamma$-orbit of the support of $V_i$ and set $\bx = \bx_1 \cup \bx_2$.  As in Proposition~\ref{prop:mdec}, let $\frN=\frN_\bx$, $\frK=\frK_\bx$, and $V = \Hom_k (V_1,V_2) \cong V_1^* \otimes V_2$.  Since $\Ext^1_\frM (V_1,V_2) \cong \rmH^1(\frM,V)$, it suffices to show that $\rmH^1(\frM,V)$ is zero.  We know from Section~\ref{sec:EMAs} that $V$ can be viewed as a nontrivial irreducible $\g^\bx$-module that is nontrivial as a $\g^{\bx_i}$-module, $i=1,2$ (where we view $V_i$ as a trivial $\g^{\bx_j}$-module for $i \ne j$).

If $V$ is nontrivial as a $\g^\bx_\ab$-module, then Proposition~\ref{prop:feiginprop}\eqref{prop-item:feigenprop:completely-reducible} implies $\rmH^1(\frM,V) \cong \Hom_{\g^\bx}(\frK_\ab,V)$.  On the other hand, if $V$ is trivial as a $\g^\bx_\ab$-module (hence nontrivial as a $\g^\bx_\rss$-module), then Proposition~\ref{prop:feiginprop}\eqref{prop-item:feigenprop:gab-zero-action} implies that the map $\res : \rmH^1(\frM,V) \to \Hom_{\g^\bx} (\frK_\ab,V)$ is injective.  Hence, in either case, it suffices to show $\Hom_{\g^\bx}(\frK_\ab,V) = 0$.

We now use the structure of the $\g^\bx$-module $\frK_\ab$ as described in Proposition~\ref{prop:mdec}. Suppose $f \in \Hom_{\g^\bx}(\frK_\ab,V)$.  If $f(\frN/\frK') = 0$, then $f$ descends to a $\g^\bx$-module map $\frK/\frN \to V$.  This map must be zero since $\frK/\frN$ is a trivial $\g^\bx$-module and $V$ is a nontrivial irreducible $\g^\bx$-module.

On the other hand, if $f$ does not vanish on $\frN/\frK'$, it maps this space onto $V$, since $V$ is an irreducible $\g^\bx$-module.  It follows that $\frN/\frK'$ contains a $\g^\bx$-module $M$ isomorphic to $V$.  Then we must have $M\subseteq M^x$ for some $x \in \bx$.  But this contradicts the fact that $V$ is a nontrivial $\g^{\bx_i}$-module, $i=1,2$. \qed
\end{proof}

\lv{Proof: As $L$-module, $M \ot N^* \cong \Hom_k(N,M)$, see
Lemma~\ref{lem:geniso},  hence $(M \ot N^*)^L \cong \Hom_L(N,M)$, and
the latter space is at most one-dimensional by Schur and our
assumption $k=\bar k$, and is one-dimensional $\iff N \cong M$. }

\begin{theo}[Two evaluation representations]
\label{theo:two-eval}
Suppose $V$ and $V'$ are irreducible evaluation representations corresponding to $\psi, \psi' \in \mathcal{E}$ respectively.  Let $V = \bigotimes_{x \in \bx} V_x$ and $V' = \bigotimes_{x \in \bx} V_x'$ for some $\bx \in X_*$, where $V_x, V_x'$ are (possibly trivial) irreducible evaluation modules at the point $x \in \bx$.  Then the following are true.
\begin{enumerate}
  \item \label{theo-item:two-eval:multiple-orbit-diff}
    If $\psi$ and $\psi'$ differ on more than one $\Gamma$-orbit, then $\Ext^1_\frM(V,V')=0$.

  \item \label{theo-item:two-eval:one-orbit-diff}
    If $\psi$ and $\psi'$ differ on exactly one orbit $\Gamma \cdot x_0$, $x_0 \in \bx$, then
    \[
      \Ext^1_\frM(V,V') \cong \Ext^1_\frM(V_{x_0},V'_{x_0}).
    \]

  \item \label{theo-item:two-eval:isom} If $\psi = \psi'$ (so $V \cong V')$, then
    \[ \textstyle
      (\frM_\ab^*)^{|\bx| -1 } \oplus \Ext^1_\frM (V, V) \cong \bigoplus_{x\in \bx} \Ext^1_\frM (V_x, V_x).
    \]
\end{enumerate}
\end{theo}

\begin{proof}
Suppose $\psi(x_0) \ne \psi'(x_0)$ for some $x_0 \in \bx$.  Then
\begin{align*}
  \Ext^1_\frM(V,V') &= \textstyle \Ext^1_\frM \left( \bigotimes_{x \in \bx} V_x,
  \bigotimes_{x \in \bx} V'_x \right) \\
  &\cong  \textstyle \Ext^1_\frM \left( \left( \bigotimes_{x \in \bx,\ x \ne x_0} V_x \right) \otimes
  \left( \bigotimes_{x \in \bx,\ x \ne x_0} (V'_x)^* \right), V'_{x_0}
  \otimes V_{x_0}^* \right).
\end{align*}
Recall that a tensor product of finite-dimensional completely reducible (e.g.\ irreducible) modules is again completely reducible. In particular, $V'_{x_0} \ot V^*_{x_0}$ is a direct sum of nontrivial irreducible $\g^{x_0}$-modules (hence $\frM$-modules) by Lemma~\ref{lem:schur-tensor}, since $V'_{x_0}$ and $V_{x_0}$ are not isomorphic. Using that $\Ext$ commutes with finite direct sums (\cite[Prop.~3.3.4]{wei}), it then follows
from Proposition~\ref{prop:disjoint-support} that $\Ext^1_\frM(V,V')=0$ unless $\left( \bigotimes_{x \in \bx,\ x \ne x_0} V_x \right) \otimes \left( \bigotimes_{x \in \bx,\ x \ne x_0} (V'_x)^* \right)$ contains a copy of the trivial module.  By Lemma~\ref{lem:schur-tensor}, this occurs if and only if $\bigotimes_{x \in \bx,\ x \ne x_0} V_x \cong \bigotimes_{x \in \bx,\ x \ne x_0} (V'_x)^*$, in which case it contains exactly one copy of the trivial module.  By \cite[Prop.~4.14]{NSS}, this is true if and only if $V_x \cong V_x'$ for all $x \in \bx$, $x \ne x_0$.  If this condition is satisfied, we have
\begin{gather*} \textstyle
  \Ext^1_\frM \left( \left( \bigotimes_{x \in \bx,\ x \ne x_0} V_x \right) \otimes
  \left( \bigotimes_{x \in \bx,\ x \ne x_0} (V'_x)^* \right), V'_{x_0}
  \otimes V_{x_0}^* \right) \\
  \cong \Ext^1_\frM (k_0, V_{x_0}' \otimes V_{x_0}^*)
  \cong \Ext^1_\frM (V_{x_0}, V_{x_0}').
\end{gather*}
This concludes the proof of parts~\eqref{theo-item:two-eval:multiple-orbit-diff} and~\eqref{theo-item:two-eval:one-orbit-diff}.

Now suppose $\psi = \psi'$.  Then for each $x \in \bx$, by Lemma~\ref{lem:schur-tensor} we have
\[ \ts
  V_x \otimes V_x^* \cong  k_0 \oplus \left( \bigoplus_{i \in J_x} V_x^i \right),
\]
where the $V_x^i$ are nontrivial irreducible $\g^\bx$-modules.  Then
\begin{align*}
  \Ext^1_\frM (V,V') &= \ts \Ext^1_\frM \left( \bigotimes_{x \in \bx} V_x, \bigotimes_{x \in \bx} V_x \right) \\
  &\cong \ts \Ext^1_\frM \left( \bigotimes_{x \in \bx} (V_x \otimes V_x^*), k_0 \right) \\
  &= \ts \Ext^1_\frM \left( \bigotimes_{x \in \bx}  \left( k_0 \oplus \left( \bigoplus_{i \in J_x} V_x^i \right) \right), k_0 \right) \\
  &= \ts \Ext^1_\frM \left( k_0 \oplus \left( \bigoplus_{x \in \bx,\ i \in J_x} V_x^i \right), k_0 \right) \\
  &= \ts \Ext^1_\frM \left( k_0, k_0 \right) \oplus \bigoplus_{x \in \bx,\ i \in J_x} \Ext^1_\frM \left( V_x^i, k_0 \right),
\end{align*}
where, in the second-to-last equality, we used part~\eqref{theo-item:two-eval:multiple-orbit-diff} to conclude that $\Ext_\frM^1(V^i_x \otimes V^j_y, k_0) = 0$ for $x \ne y$, $i \in J_x$, $j \in J_y$.  On the other hand, we have
\begin{align*} \ts
  \bigoplus_{x \in \bx} \Ext^1_\frM (V_x,V_x') &\cong \ts \bigoplus_{x \in \bx} \Ext^1_\frM \left( V_x \otimes V_x^*, k_0 \right) \\
  &= \ts \bigoplus_{x \in \bx} \Ext^1_\frM \left( k_0 \oplus \bigoplus_{i \in J_x} V_x^i, k_0 \right) \\
  &= \ts \bigoplus_{x \in \bx} \left(  \Ext^1_\frM(k_0,k_0) \oplus \bigoplus_{i \in J_x} \Ext^1_\frM (V_x^i, k_0 ) \right).
\end{align*}
Comparing these expressions, and using \eqref{eq:ext-homology} and Lemma~\ref{lem:corone}\eqref{lem-item:coronea} to replace the $\Ext_\frM^1(k_0,k_0)$ by $\frM_\ab^*$, yields part~\eqref{theo-item:two-eval:isom}. \qed
\end{proof}

\begin{rem}
The special case of Theorem~\ref{theo:two-eval} where $\Gamma$ is trivial and $\g$ is semisimple was proved by Kodera (\cite[Th.~3.6]{kodera}).  In this case, the term $(\frM_\ab^*)^{|\bx|-1}$ does not appear, since $\frM_\ab=0$ in the setting of \cite{kodera}.
\end{rem}

Theorem~\ref{theo:two-eval} reduces the determination of extensions between evaluation modules to the computation of extensions between single orbit evaluation representations supported on the same orbit.  It is thus important to have an explicit formula for these.

\begin{theo}[Evaluation representations supported on the same orbit]
\label{theo:ext-eval-point} Let $V$ and $V'$ be two irreducible finite-dimensional single orbit evaluation representations supported on the same orbit $\Gamma \cdot x$ for some $x \in X_\rat$.  Suppose that ${\g^x_\ab}$ acts on $V$ and $V'$ by linear forms $\la$ and $\la'$ respectively. Let
\[
   \frZ_x := \ev^{-1}_x({\g^x_\ab}) = \{ \al \in \frM : [\al, \frM] \subseteq \frK_x \}.
\]
Then
\begin{equation} \label{eq:evap1}
  \Ext_\frM^1(V, V') \cong
  \begin{cases}
    \Hom_{\g^x} (\frK_{x,\ab}, V^* \ot V') & \hbox{if $\la \ne \la'$}, \\
    \Hom_{\g^x_\rss}\,( \frZ_{x,\ab}, V^*\ot V') & \hbox{if $\la= \la'$}.
  \end{cases}
\end{equation}
In particular, if $\g^x$ is semisimple, then $\g^x=\g^x_\rss$, $\la=\la'=0$,
$\frK_x=\frZ_x$, and
\begin{equation} \label{eq:evap2}
  \Ext_\frM^1(V, V') \cong
  \Hom_{\g^x} (\frK_{x,\ab}, V^* \ot V').
\end{equation}
\end{theo}

\begin{proof} Let $\al \in \frM$. Then $[\al, \frM] \subseteq \frK_x
\iff [\al(x), \be(x)]=0$ for all $\be \in \frM \iff \al(x) \in
{\g^x_\ab}$. This proves the characterization of $\frZ_x$.

We know from \eqref{eq:ext-homology} that $\Ext^1_\frM(V, V') \cong
\rmH^1(\frM, W)$ for the $\frM$-module $W= V^* \ot V'$. Note that ${\g^x_\ab}$ acts on $W$ by $\la' - \la$. By the definition of an evaluation module, the representation of $\frM$ on
$W$ factors through $\frM/\frK_x \cong \g^x$. Hence, in the case $\lambda \ne \lambda'$,
the isomorphism \eqref{eq:evap1} is a special case of
\eqref{feigin1}. In the case $\lambda = \lambda'$, ${\g^x_\ab}$ acts trivially on $W$ and
the representation of $\frM$ on $W$ factors through $\frM/\frZ_x$. Since $\frM/\frZ_x \cong \g^x_\rss$, the isomorphism
\eqref{eq:evap1} is also a consequence of \eqref{feigin1}. \qed
\end{proof}

\begin{rem}
  In fact, the proof of Theorem~\ref{theo:ext-eval-point} carries through to the case that $V$ and $V'$ are supported on multiple orbits and~\eqref{eq:evap1} remains true in this generality (with $x$ replaced by $\bx \in X_*$).  We choose to present the result in the single orbit case since Theorem~\ref{theo:two-eval} tells us that the extensions will be zero if $V$ and $V'$ differ on more than one orbit.
\end{rem}

Since the action of $\Ga$ leaves $\g_\rss$ and $\g_\ab$ invariant and hence also $\g_\rss \ot A$ and $\g_\ab \ot A$, we have a decomposition
\begin{equation} \label{eq:algebra-decomp}
  M(X,\g)^\Ga  =  M(X,\g_\rss)^\Ga \boxplus M(X,\g_\ab)^\Ga.
\end{equation}
The following proposition allows us to reduce to the case where $\g$ is semisimple.

\begin{prop} \label{prop:M-sum-extensions}
Suppose $V, V'$ are irreducible finite-dimensional representations of $\frM$.  Write $V=V_\rss \otimes k_\lambda$ and $V' = V'_\rss \otimes k_{\lambda'}$ for $V_\rss, V'_\rss$ irreducible representations of $M(X,\g_\rss)^\Gamma$ and $k_\lambda, k_{\lambda'}$ one-dimensional representations of $M(X,\g_\ab)^\Gamma$.
\begin{enumerate}
  \item \label{prop-item:M-sum-la-ne} If $\lambda \ne \lambda'$, then $\Ext^1_\frM(V,V') = 0$.

  \item \label{prop-item:M-sum:la-eq} If $\lambda = \lambda'$, then
    \[
      \Ext^1_\frM(V,V') \cong
      \begin{cases}
        \Ext^1_{M(X,\g_\rss)^\Gamma}(V_\rss,V'_\rss), & V_\rss \not \cong V'_\rss, \\
        \Ext^1_{M(X,\g_\rss)^\Gamma}(V_\rss,V'_\rss) \oplus (M(X,\g_\ab)^\Gamma)^*, & V_\rss \cong V'_\rss,
      \end{cases}
    \]
    where $\Ext^1_{M(X,\g_\rss)^\Gamma}(V_\rss,V'_\rss)$ is described in Theorem~\ref{theo:two-eval}.
\end{enumerate}
\end{prop}

\begin{proof}
This follows from Proposition~\ref{prop:ext-direct-sums}\eqref{prop-item:kuenneth} and Corollary~\ref{cor:abelian-1d-exts}. \qed
\end{proof}

We conclude this section with a discussion of extensions in the case of irreducible finite-dimensional representations that are not evaluation representations.  Since Corollary~\ref{cor:ext-moving-arguments} implies that
\[
  \Ext^1_\frM(V,V') \cong \Ext^1_\frM(k_0,V^* \otimes V'),
\]
our previous results apply if $V^* \otimes V'$ is an evaluation representation.  Thus it suffices to describe extensions between the trivial representation $k_0$ and irreducible representations which are not evaluation representations.

Fix $\bx \in X_*$, an evaluation representation $V_\rss$ supported on $\Gamma \cdot \bx$ with $\g^\bx_\ab$ acting trivially, and a one-dimensional representation $k_\lambda$ which is not an evaluation representation.  Let
\[
  V = V_\rss \otimes k_\lambda,\quad \frK = \frK_\bx \cap \frK_\lambda, \quad \frl = \frM/\frK.
\]
The representations $\rho_{V_\rss}$ and $\rho_V$ of $\frM$ on $V_\rss$ and $V$ respectively, as well as $\la$, factor through the canonical map $p : \frM \twoheadrightarrow \frl$, giving rise to representations $\bar \la \in \frl^*$, $\bar \rho_{V_\rss} : \frl \to \gl(V_\rss)$ and $\bar \rho_V : \frl \to \gl(V)$, defined by
\[
  \la = \bar \la \circ p, \quad \rho_{V_\rss} = \bar \rho_{V_\rss} \circ p, \quad\hbox{and} \quad \rho_V = \bar \rho_V \circ p.
\]
Since $\la(\frK_\bx) \ne 0$ (otherwise $\lambda$ factors through $\ev_\bx$ and so is an evaluation representation), there exists $z\in \frK_\bx \setminus \frK$ such that $\la(z)=1$. The canonical image $\bar z \in \frl$ of $z$ then has the property that $\bar \la (\bar z) = 1$. Since $z \in \frK_\bx$, we have $\ev_\bx(z)=0$.  Thus $z$ acts by zero on $V_\rss$, and so
\[
  \rho_V(z) = \Id.
\]

\begin{prop}\label{prop:reductive-homl}
  The Lie algebra $\frl$ is finite-dimensional reductive.  Moreover,
  \[
    \Ext^1_\frM(k_0,V) \cong \rmH^1(\frM, V) \cong \Hom_\frl(\frK_\ab, V).
  \]
\end{prop}

\begin{proof}
  The first isomorphism is simply~\eqref{eq:ext-homology}.  It is obvious from the description of $\frl$ above that $\dim \frl < \infty$, that is, $\frK$ has finite codimension in $\frM$. To show that $\frl$ is reductive, it is equivalent to prove that $\frl'$ is semisimple. From the exact sequence
  \[
    0 \to \frK_\bx / \frK \to \frM/\frK \to \frM/\frK_\bx \to 0
  \]
  and $\la(z)=1$, it follows that we have an exact sequence
  \[
    0 \to k\bar z \to \frl \to \g^\bx \to 0.
  \]
  The epimorphism $\frl \to \g^\bx$ maps $\frl'$ onto the semisimple Lie algebra $(\g^\bx)'$. Since the kernel of the map $\frl \to \g^\bx$ is $k\bar z$, it is therefore enough to show that $\bar z \not \in \frl'$. We know that $\la(\frM')=0$. Since the epimorphism $\frM \to \frl$ maps $\frM'$ onto $\frl'$, $\bar \la(\frl')=0$ follows. But then $\bar \la(\bar z)= 1$ implies $\bar z \not \in \frl'$. Thus $\frl'\cong (\g^\bx)' = \g^\bx_\rss$ is semisimple.  The formula for $\rmH^1(\frM,V)$ is then an application of
  Proposition~\ref{prop:feiginprop}\eqref{prop-item:feigenprop:completely-reducible}. \qed
\end{proof}

\begin{rem}
  The above result should be compared to the $\lambda \ne \lambda'$ case of~\eqref{eq:evap1}.  Loosely speaking, Proposition~\ref{prop:reductive-homl} says that \eqref{eq:evap1} continues to hold in the case that $\lambda - \lambda'$ is not an evaluation representation.
\end{rem}

%
\section{Abelian group actions} \label{sec:abelian}
%

In this section, we focus on the case where the group $\Gamma$ is abelian.  In this context, we are able to give a more explicit description of the extensions between evaluation representations at a single point $x \in X$, where $\g^x$ is semisimple.

We know already from \eqref{eq:evap2} that $\Ext^1_\frM (V_1, V_2) \cong \Hom_{\g^x} (\frK_{x,\ab}, V_1^* \ot V_2)$ for evaluation representations $V_1$ and $V_2$ at $x$.  It is therefore crucial to understand the $\g^x$-module structure of $\frK_{x,\ab}$. It turns out that rather than dealing with $\frK_{x,\ab}=\frK_x/\frK_x'$ itself, a certain quotient $\frK_x/\frQ_x$ of $\frK_{x,\ab}$ is more accessible as a $\g^x$-module and will still provide us with some useful information about $\Ext^1_\frM(V_1, V_2)$.

Let $\Xi$ be the character group of $\Ga$. This is an abelian
group, whose group operation we will write additively. Hence, $0$ is the
character of the trivial one-dimensional representation, and if an
irreducible representation affords the character $\xi$, then $-\xi$
is the character of the dual representation.

If $\Ga$ acts on an algebra $B$ by automorphisms, it is well-known that
$B=\bigoplus_{\xi \in \Xi} B_\xi$ is a $\Xi$-grading, where $B_\xi$
is the isotypic component of type $\xi$. It follows that $\frM$ can be written as
\begin{equation} \label{eq:xrat0}
  \frM = \ts \bigoplus_{\xi \in \Xi} \, \g_\xi \ot A_{-\xi},
\end{equation}
since $\g = \bigoplus_\xi \g_\xi$ and $A=\bigoplus_\xi A_\xi$
are $\Xi$-graded and $(\g_\xi \ot A_{\xi'})^\Ga = 0$ if $\xi'\ne
-\xi$. The decomposition \eqref{eq:xrat0} is an algebra $\Xi$-grading.

If $V$ is any $\Xi$-graded vector space and $H \subseteq \Xi$ is any subset, we define
\[
  V_H = \ts \bigoplus_{\xi \in H} V_\xi.
\]
If $B$ is a $\Xi$-graded algebra and $V$ is a $\Xi$-graded $B$-module, i.e., $B_\tau \cdot V_\xi \subseteq V_{\xi + \tau}$ for all $\tau,\xi \in \Xi$, then it is clear that if $H$ is a subgroup of $\Xi$, then $B_H$ is a subalgebra of $B$ and
\[
  V = \ts \bigoplus_{\omega \in \Xi/H} V_\omega
\]
is a decomposition of $B_H$-modules.

\begin{lem}\label{lem:graded-tensor-products}
Suppose an abelian group $\Delta$ acts on a set $S$ and let $k[\Delta] = \bigoplus_{\delta \in \Delta} \, k e_\delta$ be the group algebra of $\Delta$, with multiplication $e_\delta e_\mu = e_{\delta + \mu}$, $\delta,\mu \in \Delta$.  Furthermore, suppose that
\begin{enumerate}
  \item $\frl = \bigoplus_{\delta \in \Delta} \frl_\delta$ is a $\Delta$-graded Lie algebra,

  \item $U= \bigoplus_{s \in S} U_s$ is an $\frl$-module with $\frl_\delta \cdot U_s \subseteq U_{\delta \cdot s}$,

  \item $V= \bigoplus_{s\in S} V_s$ is a $k[\Delta]$-module with $e_\delta \cdot V_s \subseteq V_{\delta \cdot s}$.
\end{enumerate}
Then $\frl$ acts on $W=\bigoplus_{s\in S} U_s \ot V_s$ by
\begin{equation} \label{eq:l-action-on-W}
  l_\delta \cdot (u_s \ot v_s) = (l_\delta \cdot u_s) \ot (e_\delta \cdot v_s)
\end{equation}
for $l_\delta \in \frl_\delta$, $u_s\in U_s$, and $v_s \in V_s$.
For every $\Delta$-orbit $O \subseteq S$ the subspace $W_O=\bigoplus_{s \in O}
U_s \ot V_s$ is $\frl$-invariant. If $\Delta$ acts freely on $O=\Delta \cdot s_0$ then as $\frl$-modules
\begin{equation}
  W_O \cong \ts \big( \bigoplus_{s\in O} U_s\big) \ot  V_{s_0}
\end{equation}
where on the right hand side $\frl$ acts only on the first factor.
\end{lem}

\begin{proof}
That \eqref{eq:l-action-on-W} defines an action of $\frl$ is easily checked and that $W_O$ is an $\frl$-submodule is obvious. For the proof of the last claim we can use the $\frl$-module isomorphism
$\psi : W_O \to \bigoplus_{s\in O} U_s \ot V_{s_0}$ given by $\psi(u_s \ot v_s) = u_s \ot e_{-\la} \cdot v_s$ for $s=\la \cdot s_0$. \lv{$\psi$ is a vector space isomorphism since the action of $e_\la $ is invertible. For $s=\la \cdot s_0$ and $l_\mu \in \frl_\mu$ we have
\begin{align*}
  \psi \big( l_\mu \cdot (u_s \ot v_s)\big) &= \psi \big( (l_\mu \cdot u_s) \ot (e_\mu \cdot v_s)\big)
  = (l_\mu \cdot u_s) \ot \big( e_{-\la -\mu} \cdot (e_\mu \cdot v_s) \big) \\
  &= (l_\mu \cdot u_s) \ot e_{-\la}\cdot v_s = l_\mu \cdot \psi(u_s \ot v_s).
\end{align*} } \qed
\end{proof}

For $x\in X_\rat$, let
\[
  I=\{a\in A : a(y)= 0 \  \forall\ y\in \Gamma \cdot x\} = \textstyle \bigcap_{y\in \Gamma \cdot x} \fm_y,
\]
where we recall that $\fm_y$ is the maximal ideal corresponding to the rational point $y\in X$.  Clearly $I$ is a $\Ga$-invariant ideal of $A$.  Also define
\[
  \Xi_x = \{ \xi \in \Xi : \xi|_{\Ga_x} = 0 \} \cong \Xi(\Ga/\Ga_x).
\]

\begin{lem} \label{lem:A-mod-I-isom}
We have an isomorphism of algebras
\begin{equation} \label{eq:AI-isom-group-alg} \textstyle
  A/I = \bigoplus_{\xi \in \Xi_x} (A/I)_\xi \cong k[\Xi_x].
\end{equation}
In particular
\begin{equation} \label{eq:function-ideal-condition}
  \xi \not\in \Xi_x  \iff A_\xi = I_\xi,
\end{equation}
and so
\begin{equation} \label{eq:free-action-ideal-condition}
  \Gamma_x = \{1\} \iff A_\xi \ne I_\xi \ \forall\ \xi \in \Xi.
\end{equation}
\end{lem}

\begin{proof}
It is easy to see that $\xi \not\in \Xi_x  \implies A_\xi = I_\xi$.  This implies the first equality in \eqref{eq:AI-isom-group-alg}.   Since $A/I$ is the coordinate ring of the finite set of points $\Gamma \cdot x$ on which $\Gamma/\Gamma_x$ acts simply transitively, it follows that for each $\xi \in \Xi_x$ there is a unique function (more precisely, coset of functions) in $(A/I)_\xi$ taking the value one at $x$.  The isomorphism in \eqref{eq:AI-isom-group-alg} is then given by identifying this function with $e_\xi$.  From this isomorphism, \eqref{eq:function-ideal-condition} follows, which in turn implies~\eqref{eq:free-action-ideal-condition}. \qed
\end{proof}

\begin{rem} \label{rem:A-noetherian-not-needed}
Lemma~\ref{lem:A-mod-I-isom} continues to hold for $x \in \maxSpec A$, without the assumption that $A$ is finitely generated.  One merely replaces $k$ by $A/\fm_x$ everywhere in the proof.
\end{rem}

We say that $\Gamma$ acts \emph{freely} on an affine scheme $X=\Spec A$ if it acts freely on $\maxSpec A$.  This is the case, for instance, for the multiloop algebras (Example~\ref{eg:multiloop}).

\begin{lem} \label{lem:free-actions}
Suppose a finite abelian group $\Gamma$ acts on a unital associative commutative $k$-algebra $A$ (and hence on $X = \Spec A$) by automorphisms.   Let $A = \bigoplus_{\xi \in \Xi} A_\xi$ be the associated grading on $A$, where $\Xi$ is the character group of $\Gamma$.  Then the following conditions are equivalent:
\begin{enumerate}
  \item \label{lem-item:free-actions:Gamma-free} $\Gamma$ acts freely on $X$,
  \item \label{lem-item:free-actions:A0-cond} $A_\xi A_{-\xi} = A_0$ for all $\xi \in \Xi$,
  \item \label{lem-item:free-actions:A-strong-grading} $A_\tau A_\xi = A_{\tau + \xi}$ for all $\tau,\xi \in \Xi$ (i.e. the grading on $A$ is \emph{strong}),
  \item \label{lem-item:free-actions:ideal-strong-grading} $\prod_{i=1}^n I_{\xi_i} = (I^n)_{\sum_{i=1}^n \xi_i}$ for all $n \ge 1$, $\xi_1,\dots,\xi_n \in \Xi$, and any $\Gamma$-invariant ideal $I$ of $A$.
\end{enumerate}
\end{lem}

\begin{proof}
We will use Lemma~\ref{lem:A-mod-I-isom} without the assumption that $A$ is finitely generated (see Remark~\ref{rem:A-noetherian-not-needed}).

\eqref{lem-item:free-actions:Gamma-free} $\Rightarrow$ \eqref{lem-item:free-actions:A0-cond}:  Assume $\Gamma$ acts freely on $X$.  Towards a contradiction, suppose $A_\xi A_{-\xi} \ne A_0$ for some $\xi \in \Xi$.  Let $J = A A_\xi$ be the ideal generated by $A_\xi$.  Since $A_\xi$ is $\Gamma$-invariant, $J$ is a $\Gamma$-invariant ideal.  Note that $J_0 = A_\xi A_{-\xi} \ne A_0$ and so $J \ne A$.  Thus $J$ is contained in some maximal ideal $\fm_x$.  Since $J$ is $\Gamma$-invariant, we have that $J$ is contained in $I = \bigcap_{y \in \Gamma \cdot x} \fm_y$.  So $A_\xi = J_\xi \subseteq I_\xi$.  Thus $I_\xi = A_\xi$.  By \eqref{eq:free-action-ideal-condition} we have $\Gamma_x \ne \{1\}$, which contradicts the fact that $\Gamma$ acts freely on $X$.

\eqref{lem-item:free-actions:A0-cond} $\Rightarrow$ \eqref{lem-item:free-actions:Gamma-free}: Assume $\Gamma$ does not act freely on $X$.  Then there exists a point $x \in X_\rat$ such that $\Gamma_x \ne \{1\}$.  Let $I = \bigcap_{y \in \Gamma \cdot x} \fm_y$.  By \eqref{eq:function-ideal-condition}, we have $A_\xi = I_\xi$ for all $\xi \in \Xi \setminus \Xi_x$.   Choose some $\xi \in \Xi \setminus \Xi_x$ (which is possible since $\Gamma_x \ne \{1\}$).  Then $A_\xi A_{-\xi} = I_\xi A_{-\xi} \subseteq I_0 \ne A_0$.

\eqref{lem-item:free-actions:A0-cond} $\Rightarrow$ \eqref{lem-item:free-actions:A-strong-grading}: Assume~\eqref{lem-item:free-actions:A0-cond} is true.  Fix $\tau,\xi \in \Xi$.  It is clear that $A_\tau A_\xi \subseteq A_{\tau + \xi}$ for all $\tau,\xi$.  By~\eqref{lem-item:free-actions:A0-cond}, we can write
\[ \ts
  1 = \sum_{i=1}^n f_i g_i,\quad f_i \in A_{-\xi},\ g_i \in A_\xi.
\]
Then, for all $p \in A_{\tau + \xi}$, we have
\[  \ts
  p = p \left( \sum_{i=1}^n f_ig_i \right) = \sum_{i=1}^n (pf_i)g_i \in A_\tau A_\xi.
\]

\eqref{lem-item:free-actions:A-strong-grading} $\Rightarrow$ \eqref{lem-item:free-actions:ideal-strong-grading}:  Suppose \eqref{lem-item:free-actions:A-strong-grading} holds.  Let $I$ be a $\Gamma$-invariant ideal of $A$ and $\xi_1,\dots,\xi_n \in \Xi$.  Set $\xi = \sum \xi_i$.  It is clear that $\prod I_{\xi_i} \subseteq (I^n)_\xi$ and so it suffices to prove the reverse inclusion.  Since $(I^n)_{\xi}$ is the sum of all $\prod_{i=1}^n I_{\tau_i}$ for which $\sum \tau_i = \xi$, it is enough to show that $\prod_{i=1}^n I_{\tau_i} \subseteq \prod_{i=1}^n I_{\xi_i}$ for all $\tau_1,\dots,\tau_n \in \Xi$ satisfying $\sum \tau_i = \xi$.  It follows from \eqref{lem-item:free-actions:A-strong-grading} that $\prod_{i=1}^n A_{\xi_i - \tau_i} = A_0$.  Thus we can write
\[ \ts
  1 = \sum_{i=1}^m f_{i,1} \cdots f_{i,n},\quad f_{i,j} \in A_{\xi_j-\tau_j},\ i=1,\dots,m,\
   j = 1,\dots,n.
\]
Then for any $p_i \in I_{\tau_i}$, $1 \le i \le n$, we have
\[ \ts
  \prod_{i=1}^n p_i = \left(\sum_{i=1}^m f_{i,1} \cdots f_{i,n}\right) \prod_{i=1}^n p_i = \sum_{i=1}^m (p_1 f_{i,1}) \cdots (p_n f_{i,n}) \in \prod_{i=1}^n I_{\xi_i}.
\]

It is obvious that \eqref{lem-item:free-actions:ideal-strong-grading} $\Rightarrow$ \eqref{lem-item:free-actions:A-strong-grading} and \eqref{lem-item:free-actions:A-strong-grading} $\Rightarrow$ \eqref{lem-item:free-actions:A0-cond} and so the proof is complete. \qed
\end{proof}

Now,
\begin{align}\nonumber
  \frK_x &= \{ \al \in \g \ot A : \al(x) = 0 \}^\Ga = (\g \ot I)^\Ga
  = \ts \bigoplus_{\xi \in \Xi} \, \g_\xi \ot I_{-\xi}, \\  \label{eq:xrat2}
  \frM/\frK_x &\cong \ts \bigoplus_{\xi \in \Xi} \left(\g_\xi \ot A_{-\xi}\big/ I_{-\xi} \right) \cong \ts\bigoplus_{\xi \in \Xi_x} \g_\xi \ot
  (A/I)_{-\xi},\\ \nonumber
  \frK_x' &= \ts \bigoplus_{\xi \in \Xi}\, (\frK_x')_\xi,
  \quad \text{where} \quad (\frK_x')_\xi =
  \sum_{\ta \in \Xi} [\g_\ta, \g_{\xi-\ta}] \ot I_{-\ta} I_{\ta - \xi}. \nonumber
\end{align}
The ideal $I^2$ is $\Gamma$-invariant since $I$ is.  We define
\[
 \frQ_x := (\g \ot I^2)^\Ga =  \ts\bigoplus_{\xi \in \Xi} \frQ_\xi,
 \quad \text{where } \frQ_\xi = \g_\xi \ot (I^2)_{-\xi},
\]
which is a $\Xi$-graded ideal of $\frM$ containing $\frK_x'$.
Thus
\[
  \frK_x'  \ideal \frQ_x \ideal \frK_x \ideal \frM.
\]

Now,
\begin{equation}
  \g^x = \textstyle \bigoplus_{\xi \in \Xi_x} \, \g^x_\xi
\end{equation}
is a $\Xi_x$-graded Lie algebra.  We know that $\g^x \cong \frM/\frK_x$ acts on $\frK_{x,\ab}$, an
action induced by the adjoint action of  $\frM$ on the ideal $\frK_x$.
Since $\frQ_x$ is an ideal of $\frM$, the action of
$\frM $ on $\frK_{x,\ab}$ leaves $\frQ_x/\frK_x'$
invariant, so that $\frM/\frK_x$ also acts on the quotient
\begin{equation}
   \frK_{x,\ab} \Big/ \frQ_x \big/ \frK_x' \cong  \frK_x / \frQ_x \cong \ts \bigoplus_{\xi \in \Xi} \, \g_\xi \ot (I/I^2)_{-\xi}.
\end{equation}
The action of $\g^x$ on $\frK_x/\frQ_x$ is given in terms of the elements $e_\tau$ used in the isomorphism \eqref{eq:AI-isom-group-alg} as follows:
\begin{equation}
  l_\tau \cdot (u_\xi \ot v_{-\xi}) = [l_\tau, u_\xi] \ot e_{-\tau} \cdot
  v_{-\xi} \in \g_{\tau + \xi} \ot (I/I^2)_{-(\tau +  \xi)},
\end{equation}
for $l_\tau \in \g^x_\tau$, $u_\xi \in \g_\xi$ and $v_{-\xi} \in (I/I^2)_{-\xi}$.

\begin{lem}[Decomposition of the $\g^x$-module $\frK_x/\frQ_x$, $\Gamma$ abelian] \label{lem:deco}
We have a decomposition of $\g^x$-modules
\[ \textstyle
  \frK_x/\frQ_x = \bigoplus_{\omega \in \Xi/\Xi_x} (\frK_x/\frQ_x)_\omega,
\]
and
\[
  (\frK_x/\frQ_x)_\omega \cong \g_\omega \otimes (I/I^2)_{-a_\omega}
\]
as $\g^x$-modules, where $a_\omega$ is any representative of the coset $\omega \in \Xi/\Xi_x$, and where $\g^x$ acts trivially on each $(I/I^2)_{-a_\omega}$.  In particular
\[
  (\frK_x/\frQ_x)_{\Xi_x} \cong \g^x \otimes (I/I^2)_0
\]
as $\g^x$-modules.  For every $\g^x$-module $V$ we have
\begin{equation} \label{eq:KQ-homs}
    \Hom_{\g^x} (\frK_x/\frQ_x, V) \cong \ts \bigoplus_{\om \in \Xi/\Xi_x} \,
    \Hom_{\g^x}( \g_\omega, V) \ot (I/I^2)^*_{-a_\omega}.
\end{equation}
\end{lem}

\begin{proof}
The first part is a direct consequence of Lemma~\ref{lem:graded-tensor-products}.  Then~\eqref{eq:KQ-homs} follows immediately. \qed
\end{proof}

\begin{rem}
Since the space $I/I^2$ is finite-dimensional, one could replace $(I/I^2)_{-a_\omega}^*$ by $(I/I^2)_{-a_\omega}$ in~\eqref{eq:KQ-homs}.  We choose to keep the dual because of the geometric interpretation as the tangent space (as opposed to the cotangent space).
\end{rem}

The exact sequence
\[
  0 \to \frQ_x/\frK_x' \to \frK_{x,\ab} \to \frK_x/\frQ_x \to 0
\]
of $\g^x$-modules gives rise to the exact sequence of $\g^x$-modules
\begin{equation} \label{eq:long-exact-gx-module-sequence}
  \begin{split}
  0 \to \Hom_{\g^x}(\frK_x/\frQ_x, V) \to \Hom_{\g^x}(\frK_{x,\ab}, V) \to
  \Hom_{\g^x}(\frQ_x/\frK_x', V) \\
  \to \Ext^1_{\g^x}( \frK_x/\frQ_x,V)  \to \cdots
  \end{split}
\end{equation}
for any $\g^x$-module $V$. \lv{Reference: Assem, page 250}

\begin{prop}[$\Gamma$ abelian] \label{prop:abelian-QK-ext-formula}
Let $V_1$ and $V_2$ be irreducible finite-dimensional evaluation representations supported on the orbit $\Gamma \cdot x$, $x\in X_\rat$, with $\g^x$ semisimple. Then, with the above notation,
\begin{equation} \label{eq:exab2}
  \begin{split}
  \Ext_\frM^1(V_1, V_2) \cong& \Hom_{\g^x}(\frQ_x/\frK_x', V_1^* \ot V_2) \\
  &\oplus \ts \big( \bigoplus_{\om \in \Xi/\Xi_x} \, \Hom_{\g^x} (\g_{\om}, V_1^* \ot V_2) \ot (I/I^2)^*_{-a_\om} \big),
  \end{split}
\end{equation}
where $a_\om$ is any representative of the coset $\om \in
\Xi/\Xi_x$.
\end{prop}

\begin{proof} We abbreviate $V=V_1^*\ot V_2$. Recall from Theorem~\ref{theo:ext-eval-point}
that $\Ext^1_\frM(V_1, V_2) \cong \Hom_{\g^x} (\frK_{x,\ab}, V)$. The
claim therefore follows from \eqref{eq:KQ-homs} and the exact sequence \eqref{eq:long-exact-gx-module-sequence} as soon as we show that $\Ext^1_{\g^x}(\frK_x/\frQ_x, V)=0$. To prove this, we
use the $\g^x$-module decomposition of $\frK_x/\frQ_x$ established in
Lemma~\ref{lem:deco} and the fact that $\Ext$ commutes with finite sums (\cite[Prop.~3.3.4]{wei}):
\begin{equation*}
 \Ext^1_{\g^x}(\frK_x/\frQ_x, V) \cong  \ts \bigoplus_{\om \in \Xi/\Xi_x}
     \Ext^1_{\g^x}( \g_\om \ot (I/I^2)_{-a_\om}, V).
 \end{equation*}
The $\g^x$-module $\g_\om \ot (I/I^2)_{-a_\om}$ is a direct sum of
submodules $M_\be\cong \g_\om$ for $\be$ in some set $B$.
Hence
\[
 \Ext_{\g^x}^1(\g_\om \ot (I/I^2)_{-a_\om}, V) \cong \ts \bigoplus_{\be \in B} \Ext^1_{\g^x}(\g_\om, V)= 0
\]
since $\Ext^1_{\g^x}(\g_\om, V)= 0$ by semisimplicity of $\g^x$. \qed
\end{proof}
\lv{The proposition also holds for $\g^x$ reductive if one
imposes the following conditions: $\Ext^1_{\g^x}(\g_\om, V) = 0$ and
$\g^x_{\ab}$ acts on $V_i$ by $\la_i$, $i=1,2$ with $\la_1\ne
\la_2$. }

\begin{cor}[$\Gamma$ abelian, $\Ga_x$ trivial] \label{cor:ab-trivial-stablizer}
Suppose $\g$ is semisimple and $x \in X_\rat$ is such that $\Ga_x$ is trivial. Then for any two evaluation modules $V_1$, $V_2$ with support contained in $\Gamma \cdot x$ we have
\[
  \Ext^1_\frM(V_1, V_2) \cong \Hom_\g(\frQ_x/\frK_x', V_1^* \ot V_2) \oplus \left( \Hom_\g( \g, V_1^* \ot V_2) \ot (I/I^2)_0^* \right).
\]
\end{cor}

\begin{proof}
This is immediate from \eqref{eq:exab2} since $\g^x = \g$ and $\Xi_x = \Xi$. \qed
\end{proof}

The following result generalizes \cite[Prop.~3.1]{kodera}, where the case of a trivial group $\Gamma$ is considered.

\begin{prop}[$\Gamma$ abelian and acting freely on $X$] \label{prop:free-abelian}
Suppose $\Gamma$ acts freely on $X$ and $\g$ is semisimple.  Then for any two evaluation modules $V_1$, $V_2$ at $x$ we have
\[
  \Ext^1_\frM(V_1, V_2) \cong \Hom_\g( \g, V_1^* \ot V_2) \ot (\fm_x/\fm_x^2)^*.
\]
\end{prop}

\begin{proof}
Let $I = \bigcap_{y \in \Gamma \cdot x} \fm_y$.  By Lemma~\ref{lem:free-actions}, we have $I_\tau I_\xi = (I^2)_{\tau + \xi}$ for all $\tau,\xi \in \Xi$.  Then
\[ \textstyle
  (\frK_x')_\xi = \left(\sum_\la [\g_\la, \g_{\xi-\la}] \right)\ot (I^2)_{-\xi} = \g_\xi \ot (I^2)_{-\xi} = \frQ_\xi,\quad \text{for all } \xi \in \Xi.
\]
By Corollary~\ref{cor:ab-trivial-stablizer}, we then have
\[
  \Ext^1_\frM(V_1, V_2) \cong \Hom_\g( \g, V_1^* \ot V_2) \ot (I/I^2)_0^*.
\]
Since the group $\Gamma$ acts freely on $X$, we have
\[ \textstyle
  (I/I^2)_0 = \left(\bigcap_{y \in \Gamma \cdot x} \fm_y / \bigcap_{y \in \Gamma \cdot x} \fm_y^2 \right)_0 \cong \left( \bigoplus_{y \in \Gamma_\cdot x} \fm_y/\fm_y^2 \right)_0 \cong \fm_x/\fm_x^2
\]
and the result follows.
\lv{I think it follows by our standard Bourbaki reference (\cite[Ch.~II, \S1.2, Prop.~5 and Prop.~6]{Bou:ACb1}) that the diagonal map $A \to A^{|\Gamma|}$ induces an isomorphism
\begin{equation} \label{eq:isom}
  A/\left( \bigcap_{y \in \Gamma \cdot x} \fm_y^2 \right) \cong \bigoplus_{y \in \Gamma \cdot x} A/\fm_y^2.
\end{equation}
By restriction, we have a map
\[
  \left( \bigcap_{y \in \Gamma \cdot x} \fm_y \right) / \left( \bigcap_{y \in \Gamma \cdot x} \fm_y^2 \right) \to \bigoplus_{y \in \Gamma \cdot x} \fm_y/\fm_y^2.
\]
(It is clear that the image of the restriction lies in the space on the right hand side.)
Since this map is the restriction of an isomorphism, it is injective.  So it suffices to show it is surjective and for this, it suffices to show that each summand is contained in the image (since this is a linear map).  Choose $y \in \Gamma \cdot x$.  By \eqref{eq:isom} we can find an element $f \in A$ such that
\[
  f \equiv 1 \mod \fm_y^2,\quad f \in \fm_z^2 \ \forall\ z \in \Gamma \cdot x,\ z \ne y.
\]
Then, for an arbitrary $g \in \fm_y$, we have
\[
  fg \equiv g \mod \fm_y^2,\ fg \in \fm_y \cap \bigcap_{\substack{z \in \Gamma \cdot x \\ z \ne y}} \fm_z^2.
\]}\qed
\end{proof}

\begin{rem} \label{rem:g-strong-grading}
The proof of Proposition~\ref{prop:free-abelian} shows that $\frK_x' = \frQ_x$ when $\g$ is semisimple and $I_\tau I_\xi = (I^2)_{\tau + \xi}$ for all $\tau,\xi \in \Xi$, where $I = \bigcap_{y \in \Gamma \cdot x} \fm_y$.  Another condition ensuring that $\frK_x'=\frQ_x$ is that the grading on $\g$ be strong, that is, $[\g_\tau,\g_\xi]=\g_{\tau + \xi}$ for all $\tau,\xi \in \Xi$.  Indeed, if this is the case, then
\[ \textstyle
  (\frK_x')_\xi = \g_\xi \ot \left( \sum_\mu I_{-\mu} I_{\mu - \xi} \right) = \g_\xi \ot (I^2)_{-\xi} = \frQ_\xi \quad \text{for all } \xi \in \Xi.
\]
Note that, since this condition is independent of the point $x$, it implies that $\frK'_x = \frQ_x$ for all $x \in X_\rat$.
\end{rem}

\begin{example} \label{ex:OnotKprime}
To show that in general $\frK_x'\subsetneq \frQ_x$ we use
Example~\ref{ex:J}. For any point $x\in X_\rat$ we have
\begin{align*}
   \frK_x' &= \big(\g_{0,\rss} \ot (I_0^2 + I_1^2) \big)
       \oplus (\g_{0,\ab} \ot I_1^2) \oplus  (\g_1 \ot I_0I_1), \\
   \frQ_x &= \big(\g_{0,\rss} \ot (I_0^2 + I_1^2) \big)
   \oplus \big( \g_{0,\ab} \ot (I_0^2 + I_1^2) \big)
   \oplus  (\g_1 \ot I_0I_1), \hbox{ and so} \\
   \frQ_x/\frK_x' &= \g_{0,\ab} \ot (I_0^2 + I_1^2) / I_1^2.
\end{align*}
Hence
\[
  \frQ_x = \frK_x' \iff \g_{0,\ab} = 0 \hbox{ or } I_0^2 \subseteq I_1^2 \iff \frN_x = \frK_x,
\]
where the last equivalence follows from Example~\ref{ex:J}. We note
that $\g_{0,\ab} \ne 0$ and $I_0^2 \not\subseteq I_1^2$ in case $\frM$ is the Onsager
algebra and $x \ne \pm 1$.  Indeed, in the notation of Section~\ref{sec:app-onsager}, set $a= x + x^{-1}$.  Then $I_0 = (z-a)A_0$ and $I_1 = (z-a)yA_0$.  Hence
\[
  I_1^2 = (z-a)^2y^2A_0 = (z-a)^2(z-2)(z+2)A_0 \subsetneq (z-a)^2 A_0 = I_0^2.
\]
\end{example}

%
\section{Block decompositions} \label{sec:block-decomps}
%

In this section we investigate the block decomposition of the category of finite-dimensional representations of an equivariant map algebra.  We first recall some basic facts about block decompositions in general.

Let $\mathcal{C}$ be an abelian category in which every object has finite length (for instance, the category $\mathcal{F}$ of finite-dimensional representations of an equivariant map algebra is such a category).  Then it is well known that every object can be written uniquely (up to isomorphism) as a direct sum of indecomposable objects.

\begin{defin}[Linked]
Two indecomposable objects $V_1$ and $V_2$ are \emph{linked} if there is no decomposition $\mathcal{C} = \mathcal{C}_1 \oplus \mathcal{C}_2$ as a sum of two abelian subcategories, such that $V_1 \in \mathcal{C}_1$, $V_2 \in \mathcal{C}_2$.
\end{defin}

It is easy to see that linkage is an equivalence relation.

\begin{prop} \label{prop:general-block-decomp}
Let $\cB$ be the set of equivalence classes of linked indecomposable objects.  For $\alpha \in \cB$, let $\mathcal{C}_\alpha$ be the full subcategory of $\mathcal{C}$ consisting of direct sums of objects from $\alpha$.  Then $\mathcal{C} = \bigoplus_{\alpha \in \cB} \mathcal{C}_\alpha$ and this is the unique decomposition of $\mathcal{C}$ into a sum of indecomposable abelian subcategories.
\end{prop}

\begin{defin}[Block decomposition]
In the setting of Proposition~\ref{prop:general-block-decomp}, the subcategories $\mathcal{C}_\alpha$ are called the \emph{blocks} of $\mathcal{C}$ and the decomposition $\mathcal{C}=\bigoplus_\alpha \mathcal{C}_\alpha$ is called the \emph{block decomposition} of $\mathcal{C}$.
\end{defin}

By the Jordan-H\"older Theorem, one can uniquely specify the irreducible objects (with multiplicity) which occur as constituents of any $X \in \mathcal{C}$.

\begin{defin}[Ext-blocks]
On the set of irreducible objects of $\mathcal{C}$, consider the smallest equivalence relation such that two irreducible objects $V,V'$ are equivalent whenever they are isomorphic or $\Ext^1_\mathcal{C}(V,V') \ne 0$.  We call the equivalence classes for this equivalence relation \emph{ext-blocks} and let $\cB_\mathrm{ext}(\cC)$ denote the set of ext-blocks.  For $b \in \cB_\mathrm{ext}(\cC)$, let $\mathcal{C}_b$ denote the full subcategory of $\cC$ whose objects are precisely those objects in $\cC$ whose constituents all lie in $b$.
\end{defin}

For any object $M$ in $\cC$ and ext-block $b$, let $M_b$ denote the sum of all submodules of $M$ contained in $\cC_b$.  Note that $M_b$ is the largest submodule with this property.

\begin{lem}
For any objects $M,M'$ in $\cC$, we have
\begin{enumerate}
  \item $M = \bigoplus_{b \in \cB_\mathrm{ext}(\cC)} M_b$, and
  \item $\Hom_\frM (M,M') = \bigoplus_{b \in \cB_\mathrm{ext}(\cC)} \Hom_\frM (M_b,M_b')$.
\end{enumerate}
\end{lem}

\begin{proof}
This is proven in \cite[II.7.1]{Jan} in the setting of representations of algebraic groups. The proof there immediately translates to the current setting. \qed
\end{proof}

\lv{
\begin{proof}
Let $N = \sum_{b \in \cB_\mathrm{ext}(\cC)} M_b \subseteq M$.  We first prove
that this sum is direct.  If it were not direct, there would exist
some $b \in \cB_\mathrm{ext}(\cC)$ and some irreducible object $P$ such that
\[
  P \subseteq M_b \cap \left( \sum_{b' \ne b} M_{b'} \right).
\]
But then there would exist composition series of $M_b$ and of
$\sum_{b' \ne b} M_{b'}$ containing $P$, which is a contradiction.

Next we prove that $N = M$.  Suppose, on the contrary, that $N
\subsetneq M$.  Then there exists a subobject $N'$ of $M$,
containing $N$, such that $V=N'/N$ is irreducible.  Let $b$ be the
ext-block of $V$ and
\[
  N'' = \bigoplus_{b' \ne b} M_{b'} \subseteq N \subseteq N'.
\]
Note that the constituents of $N'/N''$ belong to $b$ since they are
either $V$ or constituents of $N/N'' \cong M_b$.  Therefore, the
exact sequence
\[
  0 \to N'' \to N' \to N'/N'' \to 0
\]
splits.  Indeed, since $\Ext^1$ commutes with direct sums, it is
sufficient to prove $\Ext^1_\mathcal{C}(M_b,M{b'}) = 0$ for $b' \ne
b$. This can be done by a double induction on the length of $M$ and
$M'$, using the long exact Ext-sequence.  Thus there exists an
subobject $E$ of $N'$ such that $N' = N'' \oplus E$.  Then $E \cong
N'/N''$ and so $E \subseteq M_b$.  But then $N' \subseteq N$, a
contradiction.

Finally, we prove $\Hom_\cC (M_b,M_b') = 0$ for $b \ne b'$. Suppose,
on the contrary, that $\Hom_\cC(M_b,M_b') \ne 0$ for some $b \ne
b'$.  Then, by passing to subquotients if necessary, we obtain a
nonzero morphism from an irreducible object in the ext-block $b$ to an
irreducible object in the ext-block $b'$.  But this contradicts that fact
that morphisms between non-isomorphic irreducible objects are always
zero. \qed
\end{proof}}

\begin{cor} \label{cor:block-decomp-from-ext-blocks}
The $\cC_b$, $b \in \cB_\mathrm{ext}(\cC)$, are the blocks of $\cC$.
\end{cor}

\begin{prop} \label{prop:block-direct-sum}
Let $L_1$, $L_2$ be Lie algebras. We denote by $\scF_1, \scF_2$ and $\scF$ the category of finite-dimensional representations of $L_1$, $L_2$ and $L= L_1 \boxplus L_2$ respectively. Let $\cB_i$, $i=1,2$, and $\cB$ be the blocks of the categories $\scF_i$ and $\scF$. The map, which assigns to irreducible $L_i$-modules $V_i$ in $\scF_i$ the block of $V_1 \ot V_2$ in $\scF$, induces a bijection between $\cB_1 \times \cB_2$ and $\cB$.
\end{prop}

\begin{proof}
To describe $\cB$ it suffices by Corollary~\ref{cor:block-decomp-from-ext-blocks} to describe the
ext-blocks of $\scF$. That they are given as stated is immediate from Proposition~\ref{prop:ext-direct-sums}. \qed
\end{proof}

\begin{example}\label{eg:ss-ab}
We can apply Proposition~\ref{prop:block-direct-sum} to an equivariant map algebra $\frM = M(X,\g)^\Ga$.  Recall the decomposition~\eqref{eq:algebra-decomp}. The finite-dimensional irreducible representations of the abelian Lie algebra $M(X,\g_\ab)^\Ga$ are all one-dimensional. Corollary~\ref{cor:abelian-1d-exts} then shows that the blocks are naturally enumerated by $(M(X,\g_\ab)^\Ga)^*$, and so by Proposition~\ref{prop:block-direct-sum} there is a natural bijection
\begin{equation} \label{eq:block-decomp}
  \cB\big( M(X,\g)^\Ga\big) \cong \cB\big( M(X,\g_\rss)^\Ga\big) \times (M(X,\g_\ab)^\Ga)^*,
\end{equation}
where $\cB\big( M(X,\g)^\Ga\big)$ and $\cB\big(M(X,\g_\rss)^\Ga\big)$ denote the blocks of the categories of finite-dimensional $M(X,\g)^\Gamma$-modules and $M(X,\g_\rss)^\Gamma$-modules respectively.  The decomposition \eqref{eq:algebra-decomp} is also helpful in deciding if $\frM$ is
extension-local, as defined below.
\end{example}

\begin{defin}[Category $\cF_\eval$ and spectral characters] \label{def:ss-spec-char}
Let $\cF_\eval$ be the full subcategory of $\cF$ consisting of modules whose constituents are evaluation modules.  For $x \in X_\rat$, we define $\cF_x$ to be the full subcategory of $\cF_\eval$  whose objects are those modules whose constituents are (finite-dimensional) evaluation modules with support contained in $\Gamma \cdot x$.

Let $\cB_x$ be the set of blocks of the category $\cF_x$.  For $\gamma \in \Gamma$, the categories $\cF_x$ and $\cF_{\gamma \cdot x}$ are the same and so $\cB_x = \cB_{\gamma \cdot x}$.  We can thus define an action of $\Gamma$ on $\cB_\eval := \bigsqcup_{x \in X_\rat} \cB_x$ by letting $\gamma : \cB_x \to \cB_{\gamma \cdot x}$, $\gamma \in \Gamma$, be the identification.  If $\chi$ is a map from $X_\rat$ to $\cB_\eval$, mapping $x \in X_\rat$ to an element of $\cB_x$, we define the \emph{support} of $\chi$ to be
\[
  \Supp \chi = \{x \in X_\rat : \chi(x) \ne 0\},
\]
where here $0$ denotes the block of the trivial module.  Let $\frB_\eval$ be the set of finitely supported equivariant maps from $X_\rat$ to $\cB_\eval$, mapping $x \in X_\rat$ to an element of $\cB_x$.
Adopting terminology from \cite{CM} and \cite{kodera} for the special case where $\frM= \g \ot A$ and $\g$ is semisimple, we call elements of $\frB_\eval$ \emph{spectral characters}.

For $\psi \in \mathcal{E}$, define $\chi_\psi \in \frB_\eval$ by letting $\chi_\psi(x)$, $x \in X_\rat$, be the block containing the isomorphism class $\psi(x)$.  If $V$ is an object of $\cF_\eval$ such that there exists $\chi \in \frB_\eval$ with the property that $\chi_\psi = \chi$ for every (isomorphism class of) irreducible constituent $\ev_\psi$ of $V$, then we say $V$ has \emph{spectral character} $\chi$.  For $\chi \in \frB_\eval$, let $\cF^\chi_\eval$ be the full subcategory of $\mathcal{F}_\eval$ containing precisely the objects with spectral character $\chi$.
\end{defin}

\begin{lem} \label{lem:eval-ext-block-spec-char}
Two irreducible evaluation modules in $\cF_\eval$ are in the same ext-block if and only if they have the same spectral character.
\end{lem}

\begin{proof}
We first prove that any two irreducible evaluation modules with the same spectral character lie in the same ext-block.  Let $\psi, \psi' \in \mathcal{E}$ be such that $\chi_\psi = \chi_{\psi'}$, and let $V,V'$ be evaluation representations corresponding to $\psi,\psi'$ respectively.  Write $V = \bigotimes_{x \in \bx} V_x$ and $V' = \bigotimes_{x \in \bx} V'_x$ for some $\bx \in X_*$ (allowing $V_x$ or $V_x'$ to be trivial if necessary).  We prove the result by induction on the number $n$ of points $x \in \Supp \psi \cup \Supp \psi'$ where $\psi(x) \ne \psi'(x)$.  If $n=0$, then $V \cong V'$ and the result is clear.  Suppose $n \ge 1$ and choose a point $y \in \Supp \psi \cup \Supp \psi'$ such that $\psi(y) \ne \psi'(y)$.  Thus $V_y \not \cong V'_y$.  Since $\chi_\psi = \chi_{\psi'}$, we know that $V_y$ and $V'_y$ lie in the same ext-block of $\cF_y$.  Thus there exists a sequence
\[
  V_y = V_y^0, V_y^1, \dots, V_y^\ell = V'_y
\]
of irreducible objects of $\cF_y$ such that $\Ext_\frM^1(V_y^i,V_y^{i+1}) \ne 0$ or\linebreak $\Ext_\frM^1(V_y^{i+1},V_y^i) \ne 0$ for all $i=0,\dots,\ell-1$.  For $i = 0,\dots,\ell$, define $V^i = V_y^i \otimes \bigotimes_{x \in \bx \setminus \{y\}} V_x$.  By Theorem~\ref{theo:two-eval}, we have $\Ext_\frM^1(V^i,V^{i+1}) \ne 0$ or $\Ext_\frM^1(V^{i+1},V^i) \ne 0$ for all $i = 0,\dots,\ell-1$.  Thus $V=V^0$ and $V^l$ lie in the same block.  If $\psi''$ is the element of $\mathcal{E}$ corresponding to the evaluation representation $V^l$, we have that $\psi''$ and $\psi'$ differ at $n-1$ points.  By the inductive hypothesis, $V^l$ and $V'$ lie in the same block.  Hence $V$ and $V'$ lie in the same block, completing the inductive step.

Next we prove that any two irreducible evaluation modules in the same ext-block have the same spectral character.  Let $V,V'$ be irreducible modules in the same ext-block corresponding to $\psi, \psi' \in \mathcal{E}$, respectively.  It suffices to consider the case where $V \not \cong V'$ and $\Ext^1_\frM(V,V') \ne 0$.  By Theorem~\ref{theo:two-eval}, $\psi$ and $\psi'$ differ on exactly one orbit $\Gamma \cdot x_0$ and
\[
  \Ext^1_\frM (V_{x_0},V'_{x_0}) \cong \Ext^1_\frM(V,V') \ne 0.
\]
Thus $V_{x_0}$ and $V'_{x_0}$ lie in the same block of $\cF_{x_0}$ and $\chi_\psi = \chi_{\psi'}$. \qed
\end{proof}

\begin{prop}[Block decomposition of $\cF_\eval$] \label{prop:eval-blocks}
The $\cF^\chi_\eval$, $\chi \in \frB_\eval$, are the blocks of $\cF_\eval$.  Thus $\cF_\eval = \bigoplus_{\chi \in \frB_\eval} \cF^\chi_\eval$ is the block decomposition of $\cF_\eval$.
\end{prop}

\begin{proof}
This follows immediately from Corollary~\ref{cor:block-decomp-from-ext-blocks} and Lemma~\ref{lem:eval-ext-block-spec-char}. \qed
\end{proof}

\begin{rem} \label{rem:action-on-weight-lattice}
Suppose $\g^x$ is semisimple for some $x \in X_\rat$ and fix a triangular decomposition and a set of simple roots of $\g^x$.  Then the irreducible finite-dimensional modules of $\g^x$ are parameterized (according to their highest weight) by the set $P^+_x$ of dominant weights of $\g^x$.  Thus $\cB_x$ is always isomorphic to some quotient $\bar P^+_x$ of $P^+_x$ with respect to the equivalence relation defining ext-blocks.  In many specific examples of equivariant map algebras, we can give a precise description of this quotient (see Section~\ref{sec:applications}).

Even though $\cB_x = \cB_{\gamma \cdot x}$ for all $x \in X_\rat$ and $\gamma \in \Gamma$, the isomorphism $\cB_x \cong \bar P_x^+$ depends on $x$.  It is well known that for a semisimple Lie algebra $\frs$, $\Aut \frs \cong \Int \frs \rtimes \Out \frs$, where $\Int \frs$ is the group of inner automorphisms of $\frs$ and $\Out \frs$ is the group of diagram automorphisms of $\frs$.  The diagram automorphisms act naturally on $P^+$, the set of dominant weights of $\frs$.  If $\rho$ is an irreducible representation of $\frs$ of highest weight $\lambda \in P$ and $\gamma$ is an automorphism of $\frs$, then $\rho \circ \gamma^{-1}$ is the irreducible representation of $\frs$ of highest weight $\gamma_{\mathrm{Out}} \cdot \lambda$, where $\gamma_{\mathrm{Out}}$ is the outer part of the automorphism $\gamma$ (see \cite[VIII, \S7.2, Rem.~1]{Bou75}).  So the group $\Gamma$ acts naturally on each $P^+$ via the quotient $\Aut \frs \twoheadrightarrow \Out \frs$.  In the case that $\Gamma$ acts freely on $X$ (so that $\g^x = \g$ for all $x \in X_\rat$), $\g$ is semisimple, and $\bar P^+ := \bar P^+_x = \bar P^+_y$ for all $x,y \in X_\rat$, the set $\frB_\eval$ can then be identified with the set of finitely-supported equivariant maps from $X$ to $\bar P^+$.  For example, we will see that this is the case when $\Gamma$ is abelian (and acts freely on $X$), in which case $\bar P^+ \cong P/Q$.  In particular, this holds for untwisted map algebras and multiloop algebras (see Section~\ref{sec:applications}).
\end{rem}

\begin{defin}[Extension-local]
We say an equivariant map algebra $\frM$ is \emph{extension-local}
if $\Ext_\frM^1(V,k_\lambda)=0$ whenever $V$ is an irreducible
finite-dimensional evaluation representation and $k_\lambda$ is any
one-dimensional representation that is not an evaluation
representation.  Equivalently, by~\eqref{eq:ext-homology}, $\frM$ is
extension-local if $\rmH^1(\frM,V \otimes k_\lambda) = 0$ for $V$
and $k_\lambda$ as above.
\end{defin}

\begin{rem} \label{rem:ext-local}
If all irreducible finite-dimensional representations of $\frM$ are evaluation representations, then $\frM$ is extension-local.  This is the case, for example, if $\frM$ is perfect and in all of the main examples of equivariant map algebras, including untwisted map algebras, multiloop algebras, and generalized Onsager algebras (see \cite{NSS} and Section~\ref{sec:applications}). It is also immediate from Corollary~\ref{cor:abelian-1d-exts} that an abelian equivariant map algebra $\frM$ is extension-local.  In fact, the authors are not aware of any equivariant map algebras that are not extension-local.
\end{rem}

\begin{lem} \label{lem:ext-local-if-Mgss-perfect}
Let $\frM = \frM_1 \boxplus \frM_2$ be a direct product of equivariant map algebras, either of the form $\frM_i= M(X_i,\g)^\Ga$ where $X=X_1\sqcup X_2$ is a disjoint union of $\Ga$-invariant affine schemes, or of the form $\frM_i = M(X,\g_i)^\Ga$ where $\g=\g_1 \boxplus \g_2$ is a direct product of $\Ga$-invariant ideals $\g_i$.  Then $\frM$ is extension-local if and only if both $\frM_1$ and $\frM_2$ are so. In particular, an equivariant map algebra $\frM$ is extension-local if and only if $M(X,\g_\rss)^\Ga$ is extension-local.
\end{lem}

\begin{proof}
Suppose that $\frM_1$ and $\frM_2$ are extension-local, and let $V$ be an irreducible finite-dimensional evaluation representation and $k_\lambda$ any one-dimensional representation of $\frM$. Obviously, $k_\la \cong k_{\la_1} \ot k_{\la_2}$ for $\la_i = \la |_{\frM_i}$. Similarly, it follows from our assumptions on the $\frM_i$ that $V\cong V_1\ot V_2$ where $V_i$, $i=1,2$, are evaluation representations of $\frM_i$ respectively.  Indeed, in the first case we decompose $\bx = \bx_1 \sqcup \bx_2$ with $\bx_i = \bx\cap X_i$ and get $\g^\bx = \g^{\bx_1} \boxplus \g^{\bx_2}$. In the second case, we get $\g^\bx = \g_1^\bx \boxplus \g_2^\bx$.  Hence, by
Proposition~\ref{prop:ext-direct-sums}\eqref{prop-item:kuenneth},
\begin{equation} \label{eq:direct-sum-proof-formula}
 \Ext^1_\frM(V, k_\la) \cong
    \big( (V_1^* \ot k_{\la_1})^{\frM_1} \ot \Ext^1_{\frM_2}(V_2, k_{\la_2})\big)
  \oplus \big( \Ext^1_{\frM_1} (V_1, k_{\la_1}) \ot (V_2^* \ot k_{\la_2})^{\frM_2}
  \big).
\end{equation}
If $\Ext^1_\frM(V, k_\la) \ne 0$, say $(V_1^* \ot k_{\la_1})^{\frM_1} \ot \Ext^1_{\frM_2}(V_2, k_{\la_2}) \ne 0$, then $V_1 \cong k_{\la_1}$ is an evaluation representation by Lemma~\ref{lem:schur-tensor}. Since $\Ext^1_{\frM_2}(V_2, k_{\la_2})\ne 0$, $k_{\la_2}$ is also an evaluation representation since $\frM_2$ is extension-local.  But then so is $k_\la$.

Conversely, assume that $\frM$ is extension-local. By symmetry it is enough to prove that $\frM_1$ is also extension-local.  Let $V_1$ be an evaluation representation and $k_{\la_1}$ a one-dimensional representation of $\frM_1$ that is not an evaluation representation. Put $V_2=k_0$ and $\la_2=0$. Then $k_\la \cong k_{\la_1} \ot k_{\la_2}$ is not an evaluation representation of
$\frM$, while $V= V_1 \ot V_2$ is so.  Hence $\Ext^1_\frM(V,k_\la) = 0$ since $\frM$ is extension-local.  But then $\Ext^1_{\frM_1}(V_1, k_{\la_1}) = 0$ follows from~\eqref{eq:direct-sum-proof-formula} since $(V_2^* \ot k_{\la_2})^{\frM_2} \cong k_0^{\frM_2} = k_0\ne 0$.

The remaining assertion is immediate from the fact that the abelian algebra $\frM(X,\g_\ab)^\Gamma$ is extension-local. \qed
\end{proof}

Note that under the identification of $\frM_\ab^*$ with one-dimensional representations of $\frM$, vector addition in $\frM_\ab^*$ corresponds to the tensor product of representations.  It follows that the space of one-dimensional evaluation representations is a vector subspace of $\frM_\ab^*$, which we will denote by $\frM_{\ab,\eval}^*$.  We fix a vector space complement $\frM_{\ab,\noneval}^*$ so that $\frM_\ab^* = \frM_{\ab,\eval}^* \oplus \frM_{\ab,\noneval}^*$.

\begin{rem}
If the set $\tilde X = \{x \in X_\rat : [\g^x,\g^x] \ne \g^x\}$ is finite, there is a canonical choice of complement.  Namely, fix a set $\bx$ of points of $X_\rat$ containing exactly one point from each $\Gamma$-orbit in $\tilde X$.  Then, by \cite[(5.8)]{NSS}, we have $\frM_\ab^* \cong \frM_{\ab,\eval}^* \oplus \frM_{\ab,\noneval}^*$, where $\frM_{\ab,\eval}^* = \left( \bigoplus_{x \in \bx} \g^x/[\g^x,\g^x] \right)^*$ and $\frM_{\ab,\noneval}^*$ is the dual of the kernel of the canonical map $\frM/[\frM,\frM] \twoheadrightarrow \bigoplus_{x \in \bx} \g^x/[\g^x,\g^x]$ induced by evaluation at $\bx$.
\end{rem}

\begin{defin}[Spectral characters] \label{def:general-spectral-char}
Let $\mathfrak{B} = \mathfrak{B}_\eval \times \frM_{\ab,\noneval}^*$.  Using Remark~\ref{rem:irred-classification}\eqref{rem-item:unique-decomp}, every irreducible finite-dimensional representation of $\frM$ can be written as $V_\eval \otimes k_\lambda$ for $V_\eval \in \cF_\eval$ (unique up to isomorphism and corresponding to some $\psi \in \mathcal{E}$) and unique $\lambda \in \frM_{\ab,\noneval}^*$.  Note that this factorization depends on the choice of the complement $\frM^*_{\ab,\noneval}$.  For such a representation, we define $\chi_{\psi,\lambda} = (\chi_\psi, \lambda) \in \mathfrak{B}$.
If $V$ is an object of $\cF$ such that there exists $\chi \in \mathfrak{B}$ with the property that $\chi_{\psi,\lambda}=\chi$ for every (isomorphism class of) irreducible constituent $\ev_\psi \otimes \lambda$ of $V$, we say $V$ has \emph{spectral character} $\chi$.  Under the natural embedding of $\frB_\eval$ into $\frB$ via $\chi \mapsto (\chi, 0)$, this definition of spectral character restricts to the previous one (Definition~\ref{def:ss-spec-char}).  For $\chi \in \mathfrak{B}$, let $\cF^\chi$ be the full subcategory of $\cF$ containing precisely the objects with spectral character $\chi$.
\end{defin}

We note that the decomposition $\mathfrak{B} = \mathfrak{B}_\eval \times \frM_{\ab,\noneval}^*$ of Definition~\ref{def:general-spectral-char} is different from the one in \eqref{eq:block-decomp}.

\begin{lem} \label{lem:general-ext-block-spec-char}
If $\frM$ is extension-local, then two irreducible modules in $\cF$ are in the same ext-block if and only if they have the same spectral character.
\end{lem}

\begin{proof}
We first prove that any two irreducible modules with the same spectral character lie in the same ext-block.  Suppose $V \otimes k_\lambda$ and $V' \otimes k_{\lambda'}$ have the same spectral character for some $V,V' \in \cF_\eval$ and $\lambda,\lambda' \in \frM_{\ab,\noneval}^*$.  It follows from Definition~\ref{def:general-spectral-char} that $V$ and $V'$ also have the same spectral character and that $\lambda = \lambda'$.  Thus, by Lemma~\ref{lem:eval-ext-block-spec-char}, $V$ and $V'$ are in the same ext-block of $\cF_\eval$.  Therefore, there exists a sequence $V = V^0, V^1, \dots, V^\ell = V'$ such that $\Ext^1_\frM(V^i, V^{i+1}) \ne 0$ or $\Ext^1_\frM(V^{i+1},V^i) \ne 0$ for $0 \le i < \ell$.  Then, for $0 \le i < \ell$,
\begin{gather*}
  \Ext^1_\frM(V^i \otimes k_\lambda, V^{i+1} \otimes k_{\lambda'}) = \Ext^1_\frM(V^i, V^{i+1}) \ne 0, \quad \text{or} \\
  \Ext^1_\frM(V^{i+1} \otimes k_{\lambda'}, V^i \otimes k_\lambda) = \Ext^1_\frM(V^{i+1}, V^i) \ne 0.
\end{gather*}
Thus $V \otimes k_\lambda$ and $V' \otimes k_\lambda$ lie in the same ext-block.

Next we prove that any two irreducible modules in the same ext-block have the same spectral character.  Let $V \otimes k_\lambda$ and $V' \otimes k_{\lambda'}$ be two irreducible modules in the same ext-block with $V,V' \in \cF_\eval$ and $\lambda,\lambda' \in \frM_{\ab,\noneval}^*$.  Then there exist sequences $V = V^0, V^1, \dots, V^\ell = V'$ in $\cF_\eval$ and $\lambda = \lambda^0, \lambda^1, \dots, \lambda^\ell = \lambda'$ in $\frM_{\ab,\noneval}^*$ such that $\Ext^1_\frM(V^i \otimes k_{\lambda^i}, V^{i+1} \otimes k_{\lambda^{i+1}}) \ne 0$ or $\Ext^1_\frM(V^{i+1} \otimes k_{\lambda^{i+1}}, V^i \otimes k_{\lambda^i}) \ne 0$ for $0 \le i < \ell$.  Thus $\Ext^1_\frM(V^i \otimes (V^{i+1})^*, k_{\lambda^{i+1}-\lambda^i}) \ne 0$ or $\Ext^1_\frM(V^{i+1} \otimes (V^i)^*, k_{\lambda^i-\lambda^{i+1}}) \ne 0$ for $0 \le i < \ell$.  Since $\frM$ is extension-local, this implies that $k_{\lambda^{i+1}-\lambda^i}$ is an evaluation module for each $0 \le i < \ell$.  But then $\lambda^{i+1}-\lambda^i \in \frM_{\ab,\eval}^*$ and so $\lambda^{i+1} - \lambda^i=0$.  Therefore $\lambda = \lambda^0 = \lambda^1 = \dots = \lambda^\ell = \lambda'$.  This in turn implies that $\Ext^1_\frM(V^i, V^{i+1}) \ne 0$ or $\Ext^1_\frM(V^{i+1},V^i) \ne 0$ for $0 \le i < \ell$.  Therefore $V$ and $V'$ are in the same ext-block of $\cF_\eval$ and so have the same spectral character by Lemma~\ref{lem:eval-ext-block-spec-char}.  It follows that $V \otimes k_\lambda$ and $V' \otimes k_{\lambda'}$ have the same spectral character. \qed
\end{proof}

\begin{theo}[Block decomposition of $\cF$] \label{theo:blocks}
For an extension-local equivariant map algebra $\frM$, the $\cF^\chi$, $\chi \in \frB_X$, are the blocks of $\cF$.  Thus $\cF = \bigoplus_{\chi \in \frB_X} \cF^\chi$ is the block decomposition of $\cF$.
\end{theo}

\begin{proof}
This follows immediately from Corollary~\ref{cor:block-decomp-from-ext-blocks} and Lemma~\ref{lem:general-ext-block-spec-char}. \qed
\end{proof}

%
\section{Applications} \label{sec:applications}
%

In this section, we apply our results on extensions and block decompositions to various specific examples of equivariant map algebras which have a prominent place in the literature.

\subsection{Free abelian group actions and multiloop algebras} \label{sec:app-multiloop}

Suppose that the group $\Gamma$ is abelian and acts freely on $X$.  As noted in Section~\ref{sec:abelian}, we have a decomposition
\[ \textstyle
  \frM = \bigoplus_{\xi \in \Xi} \g_\xi \otimes A_{-\xi},
\]
where $\Xi$ is the character group of $\Gamma$.  Since the action of $\Gamma$ must preserve $\g_\ab$ and $\g_\rss$, we also have decompositions $\g_\ab = \bigoplus_{\xi \in \Xi} \g_{\ab,\xi}$ and $\g_\rss = \bigoplus_{\xi \in \Xi} \g_{\rss,\xi}$.  Using Lemma~\ref{lem:free-actions}, we have
\begin{gather*} \textstyle
  \frM' = \bigoplus_{\xi,\tau \in \Xi} [\g_\xi,\g_\tau] \otimes A_{-\xi - \tau} = \bigoplus_{\xi \in \Xi} \left( \bigoplus_{\tau \in \Xi} [\g_\tau, \g_{\xi-\tau}] \right) \otimes A_{-\xi} \\
  \textstyle = \bigoplus_{\xi \in \Xi} \g_{\rss,\xi} \otimes A_{-\xi} = (\g_\rss \otimes A)^\Gamma = M(X,\g_\rss)^\Gamma.
\end{gather*}
Therefore
\[ \textstyle
  \frM_\ab \cong \bigoplus_{\xi \in \Xi} \g_{\ab,\xi} \otimes A_{-\xi} = (\g_\ab \otimes A)^\Gamma = M(X,\g_\ab)^\Gamma.
\]
We thus have
\[
  \frM \cong \frM' \boxplus \frM_\ab,
\]
where $\frM'$ is perfect (i.e.\ $\frM'=\frM$), as is easily seen by computations similar to the above.

\begin{lem} \label{lem:free-action-perfect}
If $\frM$ is an equivariant map algebra with $\g$ semisimple and $\Gamma$ abelian and acting freely on $X$, then $\frM$ is perfect.  Thus $\frM$ has no nontrivial one-dimensional representations and all irreducible finite-dimensional representations of $\frM$ are evaluation representations.
\end{lem}

\begin{proof}
The first statement follows immediately from the above discussion.  The second is then a result of Proposition~\ref{prop:irred-classification}. \qed
\end{proof}

\begin{cor} \label{cor:free-action-ext-local}
If $\Gamma$ is abelian and acts freely on $X$, then $\frM$ is extension-local.
\end{cor}

\begin{proof}
This follows easily from Lemmas~\ref{lem:ext-local-if-Mgss-perfect} and Lemma~\ref{lem:free-action-perfect}. \qed
\end{proof}

By the above discussion and Proposition~\ref{prop:M-sum-extensions}, to describe the extensions between irreducibles, it suffices to consider the case where $\g$ is semisimple, and hence $\frM$ is perfect.

\begin{prop}[Extensions for $\Gamma$ abelian and acting freely on $X$] \label{prop:free-action-extensions}
\ Suppose $\g$ is semisimple and $V, V'$ are irreducible finite-dimensional representations of $\frM$.  Write $V = \bigotimes_{x \in \bx} V_x$ and $V' = \bigotimes_{x \in \bx} V'_x$ for some $\bx \in X_*$ and evaluation representations $V_x, V'_x$ at $x \in \bx$. Then we have the following description of the extensions of $V$ by $V'$.
\begin{enumerate}
  \item \label{prop-item:free-action-ext-b} If $V_x \not \cong V'_x$ for more than one $x \in \bx$, then $\Ext^1_\frM(V,V')=0$.

  \item \label{prop-item:free-action-ext-c} If $V_{x_0} \not \cong V'_{x_0}$ for some $x_0 \in \bx$, and $V_x \cong V'_x$ for all $x \in \bx \setminus \{x_0\}$, then
      \[
        \Ext^1_\frM(V,V') = \Hom_{\g} (\g, V_{x_0}^* \otimes V'_{x_0}) \otimes (\fm_{x_0}/\fm_{x_0}^2)^*.
      \]

  \item \label{prop-item:free-action-ext-d} If $V \cong V'$, then
      \[ \textstyle
        \Ext^1_\frM(V,V') = \bigoplus_{x \in \bx} \Hom_\g (\g, V_x^* \otimes V'_x) \otimes (\fm_x/\fm_x^2)^*.
      \]
\end{enumerate}
\end{prop}

\begin{proof}
Parts~\eqref{prop-item:free-action-ext-b} and~\eqref{prop-item:free-action-ext-c} follow from Theorem~\ref{theo:two-eval} and Proposition~\ref{prop:free-abelian}. Part~\eqref{prop-item:free-action-ext-d} is a consequence of Proposition~\ref{prop:free-abelian} and Theorem~\ref{theo:two-eval}\eqref{theo-item:two-eval:isom}, where we note that $\frM_\ab=0$. \qed
\end{proof}

Because of their prominence in the literature, we state for reference the special case where $\frM$ is a multiloop algebra.

\begin{cor}[Extensions for multiloop algebras]
Suppose
\[
  \frM = (\g,\sigma_1,\dots,\sigma_n,m_1,\dots,m_n)
\]
is a multiloop algebra as in Example~\ref{eg:multiloop}, and $V,V'$ are irreducible finite-dimensional representations of $\frM$.  Write $V = \bigotimes_{x \in \bx} V_x$, $V'= \bigotimes_{x \in \bx} V'_x$ for some $\bx \in X_*$ and (possibly trivial) evaluation representations $V_x, V'_x$ at $x \in \bx$.
\begin{enumerate}
  \item If $V_x \not \cong V'_x$ for more than one $x \in \bx$, then $\Ext^1_\frM(V,V')=0$.

  \item If $V_{x_0} \not \cong V'_{x_0}$ for some $x_0 \in \bx$, and $V_x \cong V'_x$ for all $x \in \bx \setminus \{x_0\}$, then
      \[
        \Ext^1_\frM(V,V') \cong \Hom_\g(\g,V_{x_0}^* \otimes V'_{x_0}) \otimes k^n.
      \]

  \item If $V \cong V'$, then
  \[ \textstyle
    \Ext^1_\frM(V,V) \cong \bigoplus_{x \in \bx} \Hom_\g(\g,V_x^* \otimes V_x) \otimes k^n.
  \]
\end{enumerate}
\end{cor}

\begin{proof}
In the case of the multiloop algebra, we have $(\fm_x/\fm_x^2)^* \cong k^n$ for all $x \in X_\rat$.  The result then follows from Proposition~\ref{prop:free-action-extensions}. \qed
\end{proof}

\begin{rem}[Extensions for untwisted map algebras]
Note that Proposition~\ref{prop:free-action-extensions} also describes the extensions for untwisted map algebras since for these the group $\Gamma = \{1\}$ clearly acts freely on $X$.  In this case, Proposition 6.3 specializes to \cite[Th.~3.6]{kodera} (see also \cite[\S 3.8]{CG05}).
\end{rem}

\begin{prop}[Block decompositions for $\Gamma$ abelian and acting freely on $X$] \label{prop:free-group-block-decomp}
Suppose that for all $x \in X_\rat$, the tangent space $(\fm_x/\fm_x^2)^* \ne 0$.  For example, assume that $X$ is an irreducible algebraic variety of positive dimension.  Then the blocks of $\frM$ are naturally enumerated by $\frB^\rss \times \left( M(X,\g_\ab)^\Gamma \right)^*$, where $\frB^\rss$ is the set of spectral characters for $M(X,\g_\rss)^\Gamma$, which can be identified with the set of finitely-supported equivariant maps $X_\rat \to P/Q$ (see Remark~\ref{rem:action-on-weight-lattice} for a description of the action of $\Gamma$ on $P/Q$).
\end{prop}

\begin{proof}
By Example~\ref{eg:ss-ab}, it suffices to consider the case where $\g$ is semisimple.  Then Proposition~\ref{prop:free-abelian} implies that for $V,V'$ irreducible evaluation modules at $x \in X_\rat$, we have
\[
   \Ext^1_\frM(V,V') \ne 0 \iff \Hom_\g(\g\ot V,V') \ne 0.
\]
Thus, the conditions for a non-vanishing Ext-group are the same as for the
untwisted map algebra.  The fact that $\cB_x \cong P/Q$ then follows from \cite[Prop.~1.2]{CM} (or from Corollary~\ref{cor:kumar} with $U=\g$, hence $\Span_\Z \wt U=Q$).  The remainder of the statement is an immediate consequence of Theorem~\ref{theo:blocks} (or Proposition~\ref{prop:eval-blocks}) and the fact that $\frM$ is perfect and hence has no nontrivial one-dimensional representations. \qed
\end{proof}

\begin{cor}[Blocks for multiloop algebras] \label{cor:multi}
The blocks of the category $\cF$ of finite-dimensional representations of the multiloop algebra
\[
  M(\g,\sigma_1,\dots,\sigma_n,m_1,\dots,m_n),
\]
$\g$ semisimple, are naturally enumerated by finitely-supported equivariant maps from $X$ to $P/Q$.
\end{cor}

\begin{proof}
This follows immediately from Proposition~\ref{prop:free-group-block-decomp} since $\g$ is semisimple. \qed
\end{proof}

\begin{rem} \label{rem:blocks-loops}
A special case of multiloop algebras are the untwisted and twisted loop algebras.  For them, block decompositions were described in \cite{CM} and \cite{S} respectively.
\end{rem}

\begin{cor}[Blocks for untwisted map algebras]
Suppose that, for all\linebreak $x \in X_\rat$, the tangent space $(\fm_x/\fm_x^2)^* \ne 0$.  For example, assume that $X$ is an irreducible algebraic variety of positive dimension.  Then the blocks of the category $\cF$ of finite-dimensional representations of an untwisted map algebra $\g \otimes A$ are naturally enumerated by $\frB^\rss \times (\g_\ab \otimes A)^*$ where $\frB^\rss$ is the set of finitely-supported maps from $X$ to $P/Q$.
\end{cor}

\begin{proof}
This is an immediate consequence of Proposition~\ref{prop:free-group-block-decomp}. \qed
\end{proof}

\subsection{Order two groups} \label{sec:app-order-2}

With an aim towards describing extensions and block decompositions for the generalized Onsager algebras, we consider in this subsection the case where the group $\Gamma$ is of order two (see Example~\ref{eg:order-two}).

Let $\frM = M(X,\g)^\Gamma$ be an equivariant map algebra with $\Gamma = \{1,\sigma\}$ a group of order two, and $\g$ reductive.  By Proposition~\ref{prop:M-sum-extensions}, the extensions between irreducible finite-dimensional representations of $\frM$ are determined by extensions between representations of $M(X,\g_\rss)^\Gamma$.  By Example~\ref{eg:ss-ab}, the same is true for the blocks of $\cF$.

A simple ideal $\frs$ of $\g$ is either invariant under the action of $\sigma$ or is mapped onto another simple ideal.  In the latter case, $(\frs \oplus \sigma(\frs), \sigma) \cong (\frs \boxplus \frs, \ex)$ as algebras with involutions, where $\ex$ is the exchange involution of $\frs \boxplus \frs$ defined by $\ex(u,v) = (v,u)$.  Therefore, $\frM$ is a direct product of equivariant map algebras of type $M(X,\frl)^\Gamma$ with $\frl$ simple or with $\frl = \frs \boxplus \frs$ and $\sigma$ acting by $\ex$.  In view of Propositions~\ref{prop:ext-direct-sums} and~\ref{prop:block-direct-sum} it is therefore enough to consider these two cases separately.

\begin{lem} \label{lem:exchange-extensions}
Suppose $\frM = M(X,\g)^\Gamma$ where $\g = \frs \boxplus \frs$ for a simple Lie algebra $\frs$ and $\sigma$ acting on $\g$ by the exchange involution.  Then $\frM$ is perfect, and for two evaluation representations $V,V'$ with support in $\Gamma \cdot x$, we have
\[
  \Ext^1_\frM(V,V') \cong
  \begin{cases}
    \Hom_\g(\g,V^* \otimes V') \otimes (I/I^2)^*_0, & x \notin X_\rat^\Gamma, \\
    \Hom_\frs(\frs,V^* \otimes V') \otimes (I/I^2)^*, & x \in X_\rat^\Gamma,
  \end{cases}
\]
where $I = \{a \in A : a(\Gamma \cdot x)=0\}$ and $X_\rat^\Gamma = \{x \in X_\rat : \sigma \cdot x = x\}$ is the set of $\Gamma$-fixed points of $X_\rat$.
\end{lem}

\begin{proof}
It is easy to see that $\g = \g_0 \oplus \g_1$ is a strong grading.  Thus $\frM$ is perfect and, by Remark~\ref{rem:g-strong-grading}, $\frK_x' = \frQ_x$ for all $x \in X_\rat$.  The formula for the extensions then follows from Proposition~\ref{prop:abelian-QK-ext-formula}. \qed
\end{proof}

\begin{lem} \label{lem:local-blocks-exchange}
Suppose $\frM = M(X,\g)^\Gamma$ where $\g = \frs \boxplus \frs$ for a simple Lie algebra $\frs$ and $\sigma$ acting on $\g$ by the exchange involution.  Furthermore, suppose that for all $x \in X_\rat$, the tangent space $(\fm_x/\fm_x^2)^* \ne 0$.   Then the set of blocks $\cB_x$ of the category $\cF_x$ is
\[
  \cB_x \cong
  \begin{cases}
    P/Q, & x \notin X_\rat^\Gamma, \\
    P_0/Q_0, & x \in X_\rat^\Gamma,
  \end{cases}
\]
where $P_0$ and $Q_0$ are the weight and root lattices of $\frs$ respectively.  Furthermore, the blocks of $\cF$ are naturally enumerated by the set of finitely supported equivariant maps
\[
  X_\rat \to (P/Q) \sqcup (P_0/Q_0)
\]
such that $x$ is mapped to $P/Q$ if $x \notin X_\rat^\Gamma$ and to $(P_0/Q_0)$ if $x \in X_\rat^\Gamma$.
\end{lem}

\begin{proof}
Let $I = \{a \in A : a(\Gamma \cdot x) = 0\}$.  If $x \in X_\rat^\Gamma$, then $I = \fm_x$ and $I/I^2 = \fm_x/\fm_x^2 \ne 0$, hence $(I/I^2)^* \ne 0$.  On the other hand, if $x \notin X_\rat^\Gamma$, then $I = \fm_x \cap \fm_{\sigma \cdot x}$ and $I/I^2 \cong (\fm_x/\fm_x^2) \oplus (\fm_{\sigma \cdot x}/\fm_{\sigma \cdot x}^2)$, with $\sigma$ acting by interchanging the two summands.  Thus $(I/I^2)^*_0 \ne 0$.  The description of $\cB_x$ then follows from Lemma~\ref{lem:exchange-extensions} and \cite[Prop.~1.2]{CM} (or Corollary~\ref{cor:kumar}). Since $\frM$ is perfect (hence extension-local) by Lemma~\ref{lem:exchange-extensions}, the description of the blocks is a consequence of Theorem~\ref{theo:blocks} (or Proposition~\ref{prop:eval-blocks}). \qed
\end{proof}

We now turn our attention to the remaining case where $\g$ is simple.  Note that if $\g_0$ is semisimple, one easily sees that $\frM$ is perfect.  For a point $x \in X_\rat$, let $\fm_{\bar x} = \fm_x \cap A_0$ denote the maximal ideal of $A_0$ corresponding to the image $\bar x$ of $x$ in the quotient $X \gq \Gamma = \Spec A_0 = \Spec A^\Gamma$.

\begin{lem} \label{lem:g-simple-order-2-ext-local}
If $\g$ is simple and $\Gamma$ is of order two, then $\frM$ is extension-local.
\end{lem}

\begin{proof}
Suppose $\lambda$ is a one-dimensional representation that is not an evaluation representation.  We have
\[
  \frM = (\g_0 \otimes A_0) \oplus (\g_1 \otimes A_1),\quad \frM' = (\g_{0,\rss} \otimes A_0) \oplus (\g_{0,\ab} \otimes A_1^2) \oplus (\g_1 \otimes A_1).
\]
Therefore, since $\lambda$ is nontrivial, we must have $\g_{0,\ab} \ne 0$ (so $\dim \g_{0,\ab} = 1$) and $\frK_\lambda$ must be of the form
\[
  \frK_\lambda = (\g_{0,\rss} \otimes A_0) \oplus (\g_{0,\ab} \otimes U) \oplus (\g_1 \otimes A_1)
\]
for some subspace $U$ of codimension one in $A_0$ such that $A_1^2 \subseteq U \subseteq A_0$. Let $V$ be an evaluation representation, with support contained in $\Gamma \cdot \bx$ for some $\bx \in X_*$, and let $I=I_\bx$.  Then
\[
  \frK_\bx = (\g_0 \otimes I_0) \oplus (\g_1 \otimes I_1), \quad \frK = \frK_\lambda \cap \frK_\bx = (\g_{0,\rss} \otimes I_0) \oplus (\g_{0,\ab} \otimes U \cap I_0) \oplus (\g_1 \otimes I_1).
\]
Thus
\begin{align*}
  \frK' &= (\g_{0,\rss} \otimes (I_0^2 + I_1^2)) \oplus (\g_{0,\ab} \otimes I_1^2) \oplus ([\g_{0,\rss},\g_1] \otimes I_0I_1 + \g_1 \otimes (U \cap I_0) I_1) \\
  &\supseteq (\g_{0,\rss} \otimes (I_0^2 + I_1^2)) \oplus (\g_{0,\ab} \otimes I_1^2) \oplus (\g_1 \otimes JI_1),
\end{align*}
where $J = A_0(U \cap I_0)$ is the ideal of $A_0$ generated by $U \cap I_0$ and we have used the fact that $[\g_{0,\ab},\g_1]=\g_1$ (see Example~\ref{eg:order-two}\eqref{eg-item:order-two:c}).

Since $k_\lambda$ is not an evaluation representation, we have $\frK_\bx \not \subseteq \frK_\lambda$.  Thus we can choose $z \in \g_{0,\ab} \otimes I_0$ such that $\lambda(z)=1$.  Then $z$ acts as the identity on $V \otimes k_\lambda$.  From the above we see that, for $m \ge 1$,
\[
  z^m \cdot \frK \subseteq z^m \cdot \frK_\bx \subseteq \g_1 \otimes I_0^m I_1,
\]
where $z^m \cdot \alpha$ denotes $z \cdot (z \cdot \dots (z \cdot \alpha) \dots )$, with $z$ acting $m$ times.

For a subset $B \subseteq A_0$, let $Z(B) = \{\bar x \in \maxSpec A_0 : f(\bar x)=0 \ \forall\ f \in B\}$ denote the zero set of $B$.  So $Z(I_0) = \bar \bx$, where $\bar \bx = \{\bar x : x \in \bx\}$.  We claim that $Z(U \cap I_0) = \bar \bx$.  It is clear that $Z(U \cap I_0) \supseteq Z(I_0)$, and so it remains to show the reverse inclusion.  Note first that $I_0 = \bigcap_{\bar x \in \bar \bx} \fm_{\bar x}$.  Suppose there exists $\bar y \in Z(U \cap I_0)$ such that $\bar y \notin \bar \bx$. Since the quotient map $X \to X
\gq \Ga$ is open and surjective, $\bar y$ is the image in $X \gq \Gamma$ of some point $y \in X_\rat$.  Thus
\[ \ts
  \frK_{\bx \cup \{y\}} = (\g_0 \otimes I') \oplus (\g_1 \otimes ( I_1 \cap \fm_y)), \quad \text{where} \quad I' = \fm_{\bar y} \prod_{\bar x \in \bar \bx} \fm_{\bar x}.
\]
Because $k_\lambda$ is not an evaluation representation, $\frK_{\bx \cup \{y\}} \not \subseteq \frK_\lambda$ and so $I' \not \subseteq U$.  Fix a nonzero $p \in I' \setminus U$.  Since $U$ has
codimension one in $A_0$, we have $U \oplus kp = A_0$.  Choose $f \in A_0$ such that $f(x) = 0$ for all $x \in \bx$ and $f(y)=1$. Then $f=f_U + ap$ for some $f_U \in U$ and $a \in k$. It follows that $f_U(x)=0$ for all $x \in \bx$ and $f_U(y)=1$.  Hence $f_U \in U \cap I_0$ and $y \notin Z(U \cap I_0)$.  This contradiction proves our claim.

Since $I_0$ is a radical ideal, it follows from \cite[Prop.~7.14]{AM69} that $J$ contains some power $I_0^m$ of $I_0$.  Thus $z^m \cdot \frK \subseteq \frK'$.  Hence $z$ acts nilpotently on $\frK_\ab$ and so
\[
  \rmH^1(\frM,V \otimes k_\lambda) = \Hom_\frl(\frK_\ab,V \otimes k_\lambda) = 0,
\]
where the first equality holds by Proposition~\ref{prop:reductive-homl}. \qed
\end{proof}

\begin{prop} \label{prop:order-2-ext-local}
An arbitrary equivariant map algebra $M(X,\g)^\Gamma$ with $\g$ reductive and $\Gamma$ of order two is extension-local.
\end{prop}

\begin{proof}
By Lemma~\ref{lem:ext-local-if-Mgss-perfect}, it suffices to show $M(X,\g_\rss)^\Gamma$ is extension-local.  The same lemma, together with the discussion at the beginning of this subsection, shows that it suffices to consider the cases where $\g$ is simple or $(\g,\sigma) \cong (\frs \boxplus \frs, \ex)$ with $\frs$ simple.  Thus the result follows from Lemmas~\ref{lem:exchange-extensions} (since perfect Lie algebras are extension-local) and~\ref{lem:g-simple-order-2-ext-local}. \qed
\end{proof}

The following lemma will allow us to give an explicit description of the block decompositions for generalized Onsager algebras in Section~\ref{sec:app-onsager}.

\begin{lem} \label{lem:order-two-coh}
Let $\frM$ be an equivariant map algebra as in Example~{\rm \ref{eg:order-two}}, i.e., $\g$ is simple and $\Ga$ has order $2$ acting nontrivially on $\g$. Assume further that $k_\la$ is a nontrivial one-dimensional representation. Then one of the following holds:

\begin{enumerate}
  \item $\rmH^1(\frM, k_\la) = 0$, or

  \item $\g=\lsl_2(k)$, $k_\la$ is an evaluation representation at some $x\in X^\Ga_\rat$, uniquely determined by $\la$, and $\la = \rho_x \circ \ev_x$, where $\rho_x \in \g_0^*$ is one of the two irreducible subrepresentations of the $\g_0$-module $\g_1$. In this case,
      \[
        \rmH^1(\frM, k_\la) \cong (A_1 /\fm_{\bar x} A_1)^*.
      \]
      In particular, if $A$ is a domain and $\Gamma$ acts nontrivially on $A$, then\linebreak $\rmH^1(\frM,k_\lambda) \ne 0$.  In the case of the (usual) Onsager algebra, $x=\pm 1$ and $\rmH^1(\frM, k_\la)$ is one-dimensional.
\end{enumerate}
\end{lem}

\begin{proof}
Since $\la \ne 0$, we have $\frM_\ab \cong \g_{0,\ab} \ot
(A_0/A_1^2) \ne 0$, and we can choose $z\in \g_{0,\ab} \ot A_0$
satisfying $\la(z) = 1$. Because $\g_{0,\ab}$ is one-dimensional,
there exists a codimension one subspace $B_\la \subseteq A_0$ such
that
\begin{align*}
  \frK_\la &= (\g_{0,\rss} \ot A_0) \oplus (\g_{0,\ab} \ot B_\la) \oplus (\g_1 \ot A_1), \\
  \frK_\la ' &= (\g_{0,\rss} \ot A_0) \oplus (\g_{0,\ab} \ot A_1^2) \oplus ([\g_{0,\rss}, \g_1] \ot A_1 + [\g_{0,\ab}, \g_1] \ot B_\la A_1).
\end{align*}
Suppose $\g \ne \lsl_2(k)$. It then follows from
Example~\ref{eg:order-two}\eqref{eg-item:order-two:c} that $\g_1 =
[\g_{0,\rss}, \g_1]$. Hence $\g_1\ot A_1 \subseteq \frK_\la'
\subseteq \frD_\la$. Since $z$ is central in $\g_0 \ot A_0$,
formula~\eqref{eq:corone1} shows $\g_{0,\ab} \ot B_\la \subseteq
\frD_\la$, whence $\frD_\la = \frK_\la$, i.e., $\rmH^1(\frM, k_\la)
= 0$.

Let now $\g= \lsl_2(k)$. From
Example~\ref{eg:order-two}\eqref{eg-item:order-two:sl2} we know that
then $\g_0 = \g_{0,\ab}$, $\g_1 = V_1 \oplus V_{-1}$, and there
exists $0 \ne \rh \in \g_0^*$ such that $\g_0$ acts on $V_{\pm 1}$
by $\pm \rh$. Hence
\[
  \frK_\la =  (\g_0  \ot B_\la) \oplus (\g_1 \ot A_1), \quad \frK_\la ' = (\g_0  \ot A_1^2) \oplus (\g_1 \ot B_\la A_1).
\]
We can write $z$ in the form $z=z_0 \ot z_A$ with $z_A\in A$ and
$z_0 \in \g_0$ satisfying $\rh(z_0) = 1$. Since $[z,\g_0 \ot B_\la]
= 0$, we get from \eqref{eq:corone1} that
\[
  \frD_\la = (\g_0 \ot B_\la) \oplus \big( V_1 \ot (B_\la A_1 + (1-z_A)A_1)\big) \oplus \big( V_{-1} \ot (B_\la A_1 + (1+z_A)A_1)\big).
\]
If $B_\la$ is not an ideal of $A_0$, we obtain $B_\la A_1 = B_\la
A_0 A_1 = A_0 A_1 = A_1$, so that $\frD_\la = \frK_\la$ and then
$\rmH^1(\frM, k_\la) = 0$ follows.

We are therefore left with the case that $B_\la$ is an ideal of
$A_0$. Being of codimension one, there exists a unique $\bar x \in
\maxSpec(A_0)$ such that $B_\la = \fm_{\bar x}$. From $\g_0 \ot
A_1^2 \subseteq \frM' \subseteq \frK_\lambda$ it follows that $A_1^2
\subseteq B_\la$, so $B_\la \oplus A_1 = \fm_x$ for a unique $x\in
X^\Ga_\rat$ by \eqref{eq:order-two:1}. Hence $\frK_\la = \Ker \ev_x$
and thus $\la = \rh_x \circ \ev_x$ for some $\rh_x \in \g_0^*$. If
$1-z_A \in \fm_{\bar x}$, i.e., $\rh_x = \rh$, \lv{Note $\la(z_0 \ot
(1-z_A)) = \la(z_0 \ot 1) - 1 = \rh_x(z_0) - 1$} we get $1+z_A
\not\in \fm_{\bar x}$ and $\frK_\la / \frD_\la \cong V_1 \ot (A_1 /
\fm_{\bar x} A_1) \cong A_1 / \fm_{\bar x} A_1$ follows. Similarly,
in case $1+z_A \in \fm_{\bar x}$, we have $\rh_x = - \rh$ and
$\frK_\la / \frD_\la \cong V_{-1} \ot (A_1 / \fm_{\bar x} A_1) \cong
A_1 / \fm_{\bar x} A_1$. Finally, if $1 \pm z_A \not \in \fm_{\bar
x}$, then $\frK_\la / \frD_\la = 0$.

By Lemma~\ref{lem:noeth}, $A$ is a Noetherian $A_0$-module and so $A_1$ is a finitely generated $A_0$-module. If $A_1 = \fm_{\bar x} A_1$, then by Nakayama's Lemma there exists $a_0\not \in \fm_{\bar x}$ such that $a_0 A_1 = 0$.  If $\Gamma$ acts nontrivially on $A$, then $A_1 \ne 0$.  Thus $A$ is not a
domain.

For the Onsager algebra one knows (see, for example, the proof of
\cite[Prop.~6.2]{NSS}) that $A_1$ is a free $A_0$ module of rank
$1$, whence $A_1/\fm_{\bar x} A_1 \cong A_0/\fm_{\bar x}$ is
one-dimensional. \qed
\end{proof}

One could continue to work in the generality of equivariant map algebras associated to groups of order two and deduce the extensions and block decompositions in the case where $\g$ is simple.  However, in the interest of making the exposition easier to follow and of obtaining explicit formulas, we will instead now focus on the case of the generalized Onsager algebras, which we treat in the next subsection.

\subsection{Generalized Onsager algebras} \label{sec:app-onsager}

We now apply our results to generalized Onsager algebras (see Example~\ref{eg:Onsager}).  By Theorem~\ref{theo:two-eval}, to describe arbitrary extensions between irreducible finite-dimensional representations, it suffices to give explicit formulas for the extensions between single orbit evaluation representations supported on the same orbit, which are described in Theorem~\ref{theo:ext-eval-point}.

Note that $A_0$ is a polynomial algebra (in the variable $z=t+t^{-1}$) and $A_1=yA_0$, where $y=t-t^{-1}$.   We have
\[
  y^2 = (t-t^{-1})^2 = t^2 - 2 + t^{-2} = z^2-4 = (z-2)(z+2)
\]
where the points $z=\pm 2$ correspond to the images in $X \gq \Gamma$ of the points $\pm 1 \in X$.  Thus
\[
  A_1^2 = y^2A_0 = \fm_{\bar 1} \fm_{\overline{-1}}.
\]

Suppose $x \in X$.  Let
\[
  I = \{f \in A\ |\ f|_{\Gamma \cdot x} = 0\},\quad I_0 = I \cap A_0,\quad I_1 = I \cap A_1.
\]
Then, as in Example~\ref{ex:J},
\begin{align*}
  \frK_x &= (\g_0 \otimes I_0) \oplus (\g_1 \otimes I_1), \\
  \frK_x' &= \left(\g_{0,\rss} \otimes (I_0^2 + I_1^2) \right) \oplus \left(\g_{0,\ab}
  \otimes I_1^2 \right) \oplus \left(\g_1 \otimes I_0I_1\right).
\end{align*}
Now, suppose first that $x \ne \pm 1$, so that $\Gamma_x = \{1\}$ and $\g^x=\g$.  Then
\[
  I_0 = \fm_{\bar x},\quad I_1=y\fm_{\bar x}.
\]
Thus
\[
  I_1^2 = y^2\fm_{\bar x} = \fm_{\bar 1} \fm_{\overline{-1}} \fm_{\bar x}^2,
\]
and so
\[
  \frK_x' = \left(\g_{0,\rss} \otimes \fm_{\bar x}^2\right) \oplus \left( \g_{0,\ab} \otimes \fm_{\bar 1} \fm_{\overline{-1}} \fm_{\bar x}^2 \right) \oplus \left( \g_1 \otimes y \fm_{\bar x}^2 \right).
\]
Therefore
\begin{align*}
  \frK_{x,\ab} &= \left( \g_{0,\rss} \otimes \fm_{\bar x}/\fm_{\bar x}^2 \right) \oplus \left( \g_{0,\ab} \otimes \fm_{\bar x}/\fm_{\bar 1}\fm_{\overline{-1}}\fm_{\bar x}^2 \right) \oplus \left( \g_1 \otimes y\fm_{\bar x}/y\fm_{\bar x}^2 \right) \\
  &\cong \left( \g_0 \otimes \fm_{\bar x}/\fm_{\bar x}^2 \right) \oplus \left( \g_{0,\ab} \otimes A_0/\fm_{\bar 1} \right) \oplus \left( \g_{0,\ab} \otimes A_0/\fm_{\overline{-1}} \right) \oplus \left( \g_1 \otimes y\fm_{\bar x}/y\fm_{\bar x}^2 \right).
\end{align*}
Indeed, by \cite[Ch.~II, \S1.2, Prop.~5 and Prop.~6]{Bou:ACb1}, $\fm_{\bar x}/\fm_{\bar 1}\fm_{\overline{-1}}\fm_{\bar x}^2 \cong  (\fm_{\bar x} / \fm_{\bar 1} \fm_{\bar x}) \boxplus
(\fm_{\bar x} / \fm_{\overline{-1}} \fm_{\bar x}) \boxplus (\fm_{\bar x} / \fm_{\bar x}^2)$
and $\fm_{\bar x} / \fm_{\overline{\pm 1}} \fm_{\bar x} \cong A_0 / \fm_{\overline{\pm 1}}$ since the canonical map $\fm_{\bar x} \to A_0/\fm_{\overline{\pm 1}}$ is surjective.  Thus, as $\g \cong
\frM/\frK_x$-modules, we have
\begin{equation} \label{eq:gen-Onsager-Kxab-free-point}
  \frK_{x,\ab} \cong (\g \otimes \fm_{\bar x}/\fm_{\bar x}^2) \oplus k_0 \oplus k_0, \quad x \ne \pm 1.
\end{equation}

Now suppose $x = \pm 1$, so that $\Gamma_x = \Gamma$ and $\g^x = \g_0$.  Then
\[
  I_0 = \fm_{\bar x},\quad I_1 = A_1 = yA_0.
\]
Thus
\[
  I_1^2 = \fm_{\bar 1} \fm_{\overline{-1}},
\]
and so
\[
  \frK_x' = \left(\g_{0,\rss} \otimes \fm_{\bar x}\right) \oplus \left(\g_{0,\ab} \otimes \fm_{\bar 1} \fm_{\overline{-1}} \right) \oplus \left( \g_1 \otimes y \fm_{\bar x} \right),
\]
where we have used the fact that $\fm_{\bar x}^2 + \fm_{\bar 1} \fm_{\overline{-1}} = \fm_{\bar x}$ for $x = \pm 1$.  Therefore,
\begin{align*}
  \frK_{x,\ab} &= \left( \g_{0,\ab} \otimes \fm_{\bar x}/\fm_{\bar 1}\fm_{\overline{-1}} \right) \oplus \left( \g_1 \otimes yA_0/y\fm_{\bar x} \right).
\end{align*}
We then have the following isomorphism of $\g_0$-modules:
\begin{equation} \label{eq:gen-Onsager-Kxab-nonfree-point}
  \frK_{\pm 1,\ab} \cong
  \begin{cases}
    \g_1 \oplus k_0 & \text{if } \dim \g_{0,\ab}=1,\\
    \g_1 & \text{if } \dim \g_{0,\ab}=0.
  \end{cases}
\end{equation}
Now,
\[
  \frZ_x = \ev_x^{-1}(\g^x_\ab) = \ev_x^{-1}(\g_{0,\ab}) = (\g_{0,\rss} \otimes \fm_{\bar x}) \oplus (\g_{0,\ab} \otimes A_0) \oplus (\g_1 \otimes yA_0).
\]
Thus
\begin{align*}
  \frZ_x' &= (\g_{0,\rss} \ot \fm_{\bar x}^2 + \g_0 \ot \fm_{\bar 1} \fm_{\overline{-1}}) \oplus ([\g_{0,\ab} , \g_1] \ot y A_0 + [\g_{0,\rss}, \g_1] \ot y \fm_{\bar x}) \\
  &\cong (\g_{0,\rss} \otimes \fm_{\bar x}) \oplus (\g_{0,\ab} \otimes \fm_{\bar 1} \fm_{\overline{-1}}) \oplus (\g_1 \otimes y\fm_{\bar x} + [\g_{0,\ab},\g_1] \otimes y A_0),
\end{align*}
where we have again used the fact that $\fm_{\bar x}^2 + \fm_{\bar 1} \fm_{\overline{-1}} = \fm_{\bar x}$ for $x = \pm 1$ and some results from Example~\ref{eg:order-two}.  Therefore
\[
  \frZ_{x,\ab} \cong (\g_{0,\ab} \otimes A_0/\fm_{\bar 1}) \oplus (\g_{0,\ab} \otimes A_0/\fm_{\overline{-1}})
  \oplus (\g_1 \ot yA_0 ) / ( \g_1 \ot y \fm_{\bar x} + [\g_{0, \ab}, \g_1] \ot y A_0 ).
\]
If $\g_{0,\ab} = 0$ then $\frZ_{x,\ab} \cong (\g_1 \ot yA_0) / (\g_1 \ot y \fm_{\bar x}) \cong \g_1$. On the other hand, if $\g_{0,\ab} \ne 0$, the first two terms are isomorphic to $k_0$ as $\g_0$-modules. Moreover, in this case $\g_1 = [\g_{0,\ab}, \g_1]$ by Example~\ref{eg:order-two}\eqref{eg-item:order-two:c}.
\lv{
\begin{proof}
Let $\frn = \{ u \in \g_1 : [\g_{0,\ab}, u]=0\}$. We claim that
\[
  I:= [\frn, \frn] \oplus \frn
\]
is an ideal of $\g$. Indeed, it is clear that $[\frg_0, \frn] \subseteq \frn$, hence $[\frg_0, I]\subseteq I$. To check $[\g_1, I] \subseteq I$, we use that $\g_{0,\ab}$ acts completely reducibly on $\g_1$, hence $\g_1 = [\g_{0,\ab}, \g_1] \oplus \frn$. This implies
\[
  [\g_1, \frn] = [[\g_{0,\ab}, \g_1], \frn] + [\frn,\frn] = [\frn,\frn]
\]
since $[[\g_{0,\ab}, \g_1], \frn] = [[\g_{0,\ab}, \frn],\g_1] + [\g_{0,\ab}, [\g_1, \frn]]$ and the last term vanishes since $\g_{0,\ab}$ is the centre of $\g_0$. Finally, $[\g_1, [\frn,\frn]] = [[\g_1, \frn], \frn] = [[\frn,\frn],\frn]$ and this last space is obviously killed by $\g_{0,\ab}$, i.e., it lies in $\frn$.

Having now shown that $I$ is an ideal of $\g$, assume that $I=\g$. Since then $[\g_{0,\ab}, \g] = [\g_{0,\ab}, I] = 0$, we get $\g_{0,\ab}=0$. In the other case $\frn=0$ follows. \qed
\end{proof} }
Therefore the last term in the description of $\frZ_{x,\ab}$ vanishes. To summarize, we have the following isomorphism of $\g_0$-modules.
\begin{equation} \label{eq:gen-Onsager-Zxab-nonfree-point}
  \frZ_{\pm 1,\ab} \cong
  \begin{cases}
    k_0 \oplus k_0  &\text{if } \dim \g_{0,\ab} = 1, \\
    \g_1 & \text{if } \dim \g_{0,\ab}=0.
  \end{cases}
\end{equation}

\begin{prop}[Extensions for generalized Onsager algebras] \label{prop:onsager-exts}
Suppose $V,V'$ are irreducible finite-dimensional evaluation representations at the same point $x \in X_\rat$.  If $x \ne \pm 1$, then
\[
  \Ext_\frM^1 (V,V') \cong
  \begin{cases}
    \Hom_\g (\g,V^* \otimes V') & \text{if } V \not \cong V', \\
    \Hom_\g (\g,V^* \otimes V') \oplus k^2 & \text{if } V \cong V'.
  \end{cases}
\]
If $x = \pm 1$, then
\[
  \Ext_\frM^1 (V,V') \cong
  \begin{cases}
    \Hom_{\g_0} (\g_1,V^* \otimes V') & \text{if } V \not \cong V', \\
    \Hom_{\g_0} (\g_1,V^* \otimes V') \oplus k^{\oplus 2 \dim \g_{0,\ab}}& \text{if } V \cong V'.
  \end{cases}
\]
\end{prop}

\begin{proof}
Suppose that $\g^x_\ab$ acts on $V,V'$ by $\lambda, \lambda' \in (\g^x_\ab)^*$ respectively.  First consider the case $x \ne \pm 1$.  Then $\g^x=\g$ is simple and so $\lambda=\lambda'=0$.  By Theorem~\ref{theo:ext-eval-point} and \eqref{eq:gen-Onsager-Kxab-free-point}, we have
\[
  \Ext_\frM^1(V,V') \cong \Hom_\g((\g \otimes \fm_{\bar x}/\fm_{\bar x}^2) \oplus k_0 \oplus k_0, V^* \otimes V').
\]
Since $\dim X \gq \Gamma = 1$ and $\bar x$ is a smooth point, $\fm_{\bar x}/\fm_{\bar x}^2 \cong k$.  The result then follows from Lemma~\ref{lem:schur-tensor}.

Now suppose $x = \pm 1$.  If $\lambda \ne \lambda'$, then $\dim \g_{0,\ab} =1$ and by Theorem~\ref{theo:ext-eval-point} and \eqref{eq:gen-Onsager-Kxab-nonfree-point}, we have
\[
  \Ext_\frM^1(V,V') \cong \Hom_{\g_0} (\g_1 \oplus k_0, V^* \otimes V') \cong \Hom_{\g_0}(\g_1, V^* \otimes V'),
\]
where the second isomorphism holds by Lemma~\ref{lem:schur-tensor}.

If $\lambda = \lambda'$ and $V \not \cong V'$, by Theorem~\ref{theo:ext-eval-point}, \eqref{eq:gen-Onsager-Zxab-nonfree-point}, and Lemma~\ref{lem:schur-tensor}, we have
\begin{gather*}
  \begin{split}
    \Ext^1_\frM(V,V') = \Hom_{\g_{0,\rss}} (k_0 \oplus k_0, &V^* \otimes V') \\
    &= 0 = \Hom_{\g_0} (\g_1, V^* \otimes V') \quad \text{if } \dim \g_{0,\ab} = 1,
  \end{split} \\
  \Ext^1_\frM(V,V') = \Hom_{\g_{0,\rss}} (\g_1, V^* \otimes V')=  \Hom_{\g_0} (\g_1, V^* \otimes V') \quad \text{if } \dim \g_{0,\ab} = 0,
\end{gather*}
where the last equality in the first line holds because $\g_{0,\ab}$ acts trivially on $V^* \otimes V'$, but nontrivially on $\g_1 = [\g_{0,\ab},\g_1]$ by Example~\ref{eg:order-two}\eqref{eg-item:order-two:c}.

Finally, if $V \cong V'$, then by Theorem~\ref{theo:ext-eval-point} and \eqref{eq:gen-Onsager-Zxab-nonfree-point} we have
\[
  \Ext_\frM^1(V,V') \cong \Hom_{\g_{0,\rss}} (\g_1^{\g_{0,\ab}} \oplus k_0^{\oplus 2\dim \g_{0,\ab}}, V^* \otimes V'),
\]
where we recall that $\g_1^{\g_{0,\ab}} = 0$ if $\g_{0,\ab} \ne 0$.  Since $\g_{0,\ab}$ acts trivially on $V^* \otimes V'$, we have
\[
  \Hom_{\g_{0,\rss}} (\g_1^{\g_{0,\ab}}, V^* \otimes V') \cong \Hom_{\g_0} (\g_1,V^* \otimes V'),
\]
and $\Hom_{\g_{0,\rss}} (k_0, V^* \otimes V') \cong k$, by Lemma~\ref{lem:schur-tensor}.  The result follows. \qed
\end{proof}

Specializing the proof above to the various cases leads to the following more explicit formula.

\begin{cor}[Extensions for generalized Onsager algebras] \label{cor:onsager-ext-explicit}
Suppose $V$ and $V'$ are irreducible finite-dimensional evaluation representations at the same point $x = \pm 1$, on which $\g_{0,\ab}$ acts by $\lambda,\lambda'$ respectively.  Then
\[
  \Ext^1_\frM(V,V') \cong
  \begin{cases}
    \Hom_{\g_0}(\g_1, V^* \ot V') & \text{if $\la \ne \la'$ or if } \la = \la',\ \g_{0,\ab}=0, \\
    0   & \text{if } \la = \la',\ \g_{0,\ab} \ne 0,\ V\not \cong V', \\
    k^2 & \text{if } \g_{0,\ab} \ne 0,\ V \cong V'.
  \end{cases}
\]
\end{cor}

\lv{
\begin{proof}
If $\lambda \ne \lambda'$, then $V \not \cong V'$ and the result is a consequence of Proposition~\ref{prop:onsager-exts}.  Similarly, if $\g_{0,\ab}=0$, the two expressions given in Proposition~\ref{prop:onsager-exts} are equal, and the result follows.

Now assume $\lambda = \lambda'$, $\g_{0,\ab} \ne 0$, and $V \not \cong V'$.  By Proposition~\ref{prop:onsager-exts}, we have $\Ext^1_\frM(V,V') \cong \Hom_{\g_0}(\g_1,V^* \otimes V')$.  Since $\g_{0,\ab}$ acts trivially on $V^* \otimes V'$, we have
\[
  \Hom_{\g_0}(\g_1,V^* \otimes V') = \Hom_{\g_0}(\g_1^{\g_{0,ab}},V^* \otimes V') = 0.
\]

Finally, assume $\g_{0,\ab} \ne 0$, $V \cong V'$.  Again, since $\g_{0,\ab}$ acts trivially on $V \cong V'$ we have, using Proposition~\ref{prop:onsager-exts},
\[
  \Ext^1_\frM(V,V') \cong \Hom_{\g_0}(\g_1,V^* \otimes V') \oplus k^2 = \Hom_{\g_0}(\g_1^{\g_{0,\ab}},V^* \otimes V') \oplus k^2 = k^2. \qed
\]
\end{proof}
}

We now turn our attention to giving an explicit description of the block decomposition of the category $\cF$ of irreducible finite-dimensional representations.  Since $\frM$ is extension-local by Lemma~\ref{lem:g-simple-order-2-ext-local}, we can apply the results of Section~\ref{sec:block-decomps}.  Because all irreducible finite-dimensional representations are evaluation representations, we have $\cF_\eval = \cF$, $\frB_\eval = \frB$, and Theorem~\ref{theo:blocks} (or Proposition~\ref{prop:eval-blocks}) tells us that the block decomposition is given by $\cF = \bigoplus_{\chi \in \frB} \cF^\chi$.  It remains to describe $\cB_x$ for $x \in X_\rat$.

If $x \ne \pm 1$, then by Proposition~\ref{prop:onsager-exts} and \cite[Prop.~1.2]{CM} (or Corollary~\ref{cor:kumar}), we have $\cB_x \cong P/Q$.  So in the following, we fix $x = \pm 1$.  Let $P_0$ and $Q_0$ be the weight and root lattices of $\g_{0,\rss}$ respectively.
For a finite-dimensional $\g_{0,\rss}$-module $W$, we let $\Span_\Z \wt(W) \subseteq P_0$ be the $\Z$-span of the weights of $W$.  If $\g_{x,\ab} = 0$, then Corollary~\ref{cor:kumar} implies (see Remark~\ref{rem:Kostant-Chari-Moura}) that $\mathcal{B}_x \cong P_0/\Span_\Z \wt(\g_1)$.  It remains to consider the case $\g_{x,\ab} \ne 0$, in which case we know that $\g_{x,\ab} \cong k$ (see Example~\ref{eg:order-two}\eqref{eg-item:order-two:b}).  The one-dimensional evaluation representations are thus of the form $k_a$ where $a \in \g_{x,\ab}^* \cong k$.

By Example~\ref{eg:order-two}\eqref{eg-item:order-two:c}, we have
\[
  \g_1 \cong (V \otimes k_1) \oplus (V^* \otimes k_{-1}) \quad \text{(as $\g_0$-modules)}
\]
for some irreducible $\g_{0,\rss}$-module $V$.  So $V \cong V(\nu)$, the irreducible $\g_{0,\rss}$-module of highest weight $\nu$ for some $\nu \in P_0^+$, the set of dominant integral weights of $\g_{0,\rss}$.  We have chosen the isomorphism $\g_{x,\ab}^* \cong k$ so that the one-dimensional representations appearing in the above decomposition are $k_{\pm 1}$.  The irreducible objects of $\mathcal{F}_x$ are of the form $V(\lambda) \otimes k_a$, for $a \in \g_{0,\ab}^* \cong k$.  They are thus enumerated by $P_0^+ \times k$.  We would like to find an explicit description of the equivalence relation on this set that describes the ext-blocks.

By Corollary~\ref{cor:onsager-ext-explicit}, we have
\[
  \Ext^1_\frM (V(\lambda) \otimes k_a,V(\mu) \otimes k_b) = 0 \quad \text{if} \quad a = b \text{ and } \lambda \ne \mu.
\]
Additionally,
\[
  \Hom_{\g_0} (\g_1, V(\lambda)^* \otimes V(\mu))=0
\]
since $\g_{0,\ab}$ acts on each irreducible summand of $\g_1$ nontrivially but on $V(\lambda)^* \otimes V(\mu)$ trivially.  Therefore, the relation on $P_0^+ \times k$ describing the ext-blocks is the equivalence relation generated by
\[
  (\lambda,a) \sim (\mu,b) \quad \text{if} \quad \Hom_{\g_0} (\g_1, V(\lambda)^* \otimes V(\mu) \otimes k_{b-a}) \ne 0.
\]
We denote this equivalence relation again by $\sim$.

\begin{lem}
The equivalence relation $\sim$ is the equivalence relation generated by
\begin{equation} \label{eq:equiv-relation-generator}
  (\lambda,a) \sim (\mu,b) \quad \text{if} \quad \Hom_{\g_0} (V \otimes k_1, V(\lambda)^* \otimes V(\mu) \otimes k_{b-a}) \ne 0.
\end{equation}
\end{lem}

\begin{proof}
Let $\approx$ be the equivalence relation generated by~\eqref{eq:equiv-relation-generator}.  Since $\g_1 \cong (V \otimes k_1) \oplus (V^* \otimes k_{-1})$, it is clear that $\approx$ is contained in $\sim$ (i.e.\ if two elements are equivalent with respect to $\approx$, then they are equivalent with respect to $\sim$).  Now fix $(\lambda,a)$ and $(\mu,b)$ with $\Hom_{\g_0} (\g_1, V(\lambda)^* \otimes V(\mu) \otimes k_{b-a}) \ne 0$.  Then we have
\begin{multline*}
  0 \ne \Hom_{\g_0} (\g_1, V(\lambda)^* \otimes V(\mu) \otimes k_{b-a}) \\
   \cong\Hom_{\g_0} (V \otimes k_1, V(\lambda)^* \otimes V(\mu) \otimes k_{b-a}) \oplus \Hom_{\g_0} (V^* \otimes k_{-1}, V(\lambda)^* \otimes V(\mu) \otimes k_{b-a}).
\end{multline*}
Therefore, one of the above summands must be nonzero.  If the first summand is nonzero, then $(\lambda,a) \approx (\mu,b)$.  If the second summand is nonzero, then
\[
  0 \ne \Hom_{\g_0} (V^* \otimes k_{-1}, V(\lambda)^* \otimes V(\mu) \otimes k_{b-a}) = \Hom_{\g_0} (V(\mu)^* \otimes V(\lambda) \otimes k_{a-b}, V \otimes k_1).
\]
Now, since both arguments are completely reducible $\g_0$-modules, the nonvanishing of the above Hom-space implies that there is an irreducible $\g_0$-module that is a summand of both arguments.  But this implies that
\[
  \Hom_{\g_0} (V \otimes k_1, V(\mu)^* \otimes V(\lambda) \otimes k_{a-b}) \ne 0,
\]
and so $(\mu,b) \approx (\lambda,a)$.  Thus $\sim$ is contained in $\approx$, completing the proof. \qed
\end{proof}

\begin{lem} \label{lem:UVV-dual=Q}
We have $\Span_\Z \wt(V \otimes V^*) = Q_0$.
\end{lem}

\begin{proof}
By Example~\ref{eg:order-two}\eqref{eg-item:order-two:a}, $\g_{0,\rss}$ acts faithfully on $V$.  It follows that $\g_{0,\rss}$ acts faithfully on $V \otimes V^*$ and so $Q_0 \subseteq \Span_\Z \wt(V \otimes V^*)$ (see, for example, \cite[Exercise~21.5]{Hum72}).
\lv{Details: It is clear
that $\Span_\Z \wt \subseteq P$. To see that $Q \subseteq \Span_\Z \wt$, let $\al$
be a root of $\frs$ and let $0 \ne x_\al \in \frs_\al$. Also, let $\rho$
be the representation affording $U$. Since $\rho(x_\al) \ne 0$ by
faithfulness, there exists a weight space $V_\mu$ of $V$ such that
$0 \ne \rho(x_\al) V_\mu \subset V_{\mu + \al}$. Hence $\mu$ and
$\mu + \al$ are weights and thus $\al = (\mu + \al) - \al \in
\Span_\Z \wt$.}
On the other hand, the weights of $V \ot V^*$ are of the form $\nu - w_0(\nu) - \omega$, where $w_0$ is the longest element of the Weyl group of $\g_{0,\rss}$ and $\omega \in Q_0$. Since $\nu - w_0(\nu) \in Q_0$ (\cite[VI, \S1.9, Prop. 27]{Bou81}), all the weights of $V \ot V^*$ lie in $Q_0$. \qed
\end{proof}

\begin{lem} \label{lem:Onsager-equiv-fixed-a}
For all $a \in \g_{0,\ab}^*$, we have $(\lambda,a) \sim (\mu,a)$ if and only if $\mu - \lambda \in Q_0$.
\end{lem}

\begin{proof}
First suppose that $(\lambda,a) \sim (\mu,a)$.  Then there exists a sequence
\[
  (\lambda,a) = (\lambda_0,a_0), (\lambda_1,a_1),\dots, (\lambda_n,a_n) = (\mu,b)
\]
such that for each $0 \le i < n$, we have
\begin{align}
  \Hom_{\g_0} (V \otimes k_1, V(\lambda_i)^* \otimes V(\lambda_{i+1}) \otimes k_{a_{i+1}-a_i}) &\ne 0, \label{eq:hom-i-i+1} \\
  \text{or} \quad \Hom_{\g_0} (V \otimes k_1, V(\lambda_{i+1})^* \otimes V(\lambda_i) \otimes k_{a_i-a_{i+1}}) &\ne 0. \label{eq:hom-i+1-i}
\end{align}
Now~\eqref{eq:hom-i-i+1} implies that $a_{i+1}=a_i+1$ and~\eqref{eq:hom-i+1-i} implies $a_{i+1}=a_i-1$.  Since $a_0 = a = a_n$, we must have that $n$ is even and we can partition the set $\{0,1,2,\dots,n-1\} = J_1 \sqcup J_2$, with $|J_1|=|J_2|$, such that~\eqref{eq:hom-i-i+1} holds for $i \in J_1$ and~\eqref{eq:hom-i+1-i} holds for $i \in J_2$.  This implies that
\begin{align*}
  \Hom_{\g_0} (V, V(\lambda_i)^* \otimes V(\lambda_{i+1})) \ne 0 \quad \text{for} \quad i \in J_1, \\
  \Hom_{\g_0} (V^*, V(\lambda_i)^* \otimes V(\lambda_{i+1})) \ne 0 \quad \text{for} \quad i \in J_2.
\end{align*}
By Lemma~\ref{lem:hom-space-root-lattice} (with $\frs=\g_{0,\rss}$), we have that
\begin{align*}
  \lambda_{i+1}-\lambda_i \in \wt(V) + Q_0 \quad \text{for} \quad i \in J_1, \\
  \lambda_{i+1}-\lambda_i \in \wt (V^*) + Q_0 \quad \text{for} \quad i \in J_2,
\end{align*}
where $\wt(W)$ denotes the set of weights of a $\g_{0,\rss}$-module $W$.  Thus
\begin{align*}
  \mu - \lambda &= (\lambda_n - \lambda_{n-1}) + (\lambda_{n-1} - \lambda_{n-2}) + \dots + (\lambda_1 - \lambda_0) \\
  &\quad \in \underbrace{\wt(V) + \dots + \wt(V)}_\text{$n/2$ terms} + \underbrace{\wt(V^*) + \dots + \wt(V^*)}_\text{$n/2$ terms} + Q_0 \\
  &\quad = \wt((V \otimes V^*)^{\otimes n/2}) + Q_0 = Q_0,
\end{align*}
by Lemma~\ref{lem:UVV-dual=Q}.

Now suppose $\mu - \lambda \in Q_0 = \Span_\Z \wt(V \otimes V^*)$.  Since $V \otimes V^*$ is a faithful $\g_{0,\rss}$-module, it follows from Corollary~\ref{cor:kumar} that there exists a sequence
\[
  \lambda = \lambda_0 , \lambda_1,\dots, \lambda_n = \mu
\]
such that
\[
  \Hom_{\g_0} (V \otimes V^* \otimes V(\lambda_i), V(\lambda_{i+1})) \ne 0,\quad 0 \le i < n.
\]
This implies that
\[
  \Hom_{\g_0} (V \otimes k_1, (V(\lambda_i) \otimes k_a)^* \otimes (V(\lambda_{i+1}) \otimes V \otimes k_{a+1})) \ne 0.
\]
Thus $(\lambda_i,a) \sim (\delta,a+1)$ for some irreducible summand $V(\delta)$ of $V(\lambda_{i+1}) \otimes V$.  Therefore
\[
  0 \ne \Hom_{\g_0} (V(\delta), V(\lambda_{i+1}) \otimes V) \cong \Hom_{\g_0} (V(\lambda_{i+1})^* \otimes V(\delta), V).
\]
This implies that $V$ is an irreducible summand of $V(\lambda_{i+1})^* \otimes V(\delta)$.  But then
\[
  0 \ne \Hom_{\g_0} (V, V(\lambda_{i+1})^* \otimes V(\delta)) \cong \Hom_{\g_0} (V \otimes k_1, (V(\lambda_{i+1}) \otimes k_a)^* \otimes (V(\delta) \otimes k_{a+1})),
\]
and so $(\lambda_{i+1},a) \sim (\delta,a+1)$.  Hence $(\lambda_i,a) \sim (\lambda_{i+1},a)$ for all $0 \le i < n$.  It follows that $(\lambda,a) \sim (\mu,a)$. \qed
\end{proof}

\begin{prop} \label{prop:Onsager-full-equiv-relation}
We have $(\lambda, a) \sim (\mu,b)$ if and only if there exists an $n \in \Z$ such that
\begin{equation} \label{eq:Onsager-equiv-relation-explicit}
  \mu + Q_0 = \lambda + n \nu + Q_0 \quad \text{and} \quad b=a+n.
\end{equation}
\end{prop}

\begin{proof}
The relation \eqref{eq:Onsager-equiv-relation-explicit} is the equivalence
relation generated by the relation  $\bowtie$ defined by
\[
  (\la, \mu) \bowtie (\mu, b) \quad \text{if} \quad \mu + Q_0 = \la + \nu + Q_0,\ b = a+1.
\]
To show that \eqref{eq:Onsager-equiv-relation-explicit} implies $(\lambda,a) \sim (\mu,b)$, it therefore suffices to show that $(\la, a) \bowtie (\mu, b)$
implies $(\la, a) \sim (\mu, b)$. Thus assume $\mu = \la + \nu +
\om$ for some $\om \in Q_0$ and $b=a+1$. Since $(\mu, b) \sim (\mu -
\om, b)$ by Lemma~\ref{lem:Onsager-equiv-fixed-a}, it is enough to
prove $(\la, a) \sim (\mu - \om, b)$. In other words, we can assume
$\mu = \la + \nu$, $b=a+1$. But then
\[
  \Hom_{\g_0}(V \otimes k_1, (V(\lambda) \otimes k_a)^* \otimes (V(\mu) \otimes k_b)) = \Hom_{\g_0} (V(\nu) \otimes V(\lambda), V(\mu)) \ne 0,
\]
since the tensor product $V(\nu) \otimes V(\lambda)$ has an
irreducible summand isomorphic to $V(\lambda+\nu) =V(\mu)$.  Thus
$(\lambda,a) \sim (\mu,b)$.

For the other direction, assume $(\lambda,a) \sim (\mu,b)$. It
suffices to consider the case
\[
  0 \ne \Hom_{\g_0} (V \otimes k_1, (V(\lambda) \otimes k_a)^* \otimes (V(\mu) \otimes k_b)) = \Hom_{\g_0} (V(\nu) \otimes V(\lambda), V(\mu)),
\]
where the equality follows from the fact that we must have $b=a+1$,
which is immediate by considering the action of
$\g_{0,\ab}$. Thus $V(\nu) \otimes V(\lambda)$ contains an
irreducible summand isomorphic to $V(\mu)$ and so $\mu = \lambda +
\nu - \omega$ for some $\omega \in Q_0$.  Hence $\lambda + \nu + Q_0
= \mu + Q_0$. \qed
\end{proof}

\begin{cor}
We have $\cB_x \cong (P_0/Q_0) \times (k/\Z)$, $x = \pm 1$.
\end{cor}

\begin{proof}
Fix a set of representatives for $k/\Z$. By
Proposition~\ref{prop:Onsager-full-equiv-relation}, we have a well-defined and injective map that associates to the equivalence class of $(\lambda,a)$ in $(P_0^+ \times k)/\sim$
the element $(\lambda+ n \nu + Q_0,a+n + \Z) \in (P_0/Q_0) \times
(k/\Z)$, where $n$ is the unique integer such that $a+n$ is one of
these chosen representatives for $k/\Z$. It is surjective since every class in $P_0/Q_0$ is
represented by some $\la \in P_0^+$. \qed
\end{proof}

\begin{prop}[Blocks for generalized Onsager algebras] \label{prop:gen-Onsager-block-decomp}
The blocks of the category of finite-dimensional representations of a generalized Onsager algebra are naturally enumerated by the set of finitely supported equivariant maps
\[
  X_\rat \to (P/Q) \sqcup \big( (P_0/Q_0) \times (k/\Z) \big)
\]
such that $x$ is mapped to $P/Q$ if $x \ne \pm 1$ and to $(P_0/Q_0) \times (k/\Z)$ if $x = \pm 1$.
\end{prop}

\begin{proof}
This follows from Theorem~\ref{theo:blocks} (or Proposition~\ref{prop:eval-blocks}) and the above computations of the $\cB_x$, $x \in X_\rat$. \qed
\end{proof}

\begin{cor}
If $\frM$ is the usual Onsager algebra, then the blocks of the category of finite-dimensional representations can be naturally identified with the set of finitely supported equivariant functions $\chi$ from $X_\rat$ to $(P/Q) \cup (\C/\Z)$ where $\chi(x) \in P/Q$ for $x \ne \pm 1$ and $\chi(x) \in \C/\Z$ for $x = \pm 1$.
\end{cor}

\begin{proof}
This follows immediately from Proposition~\ref{prop:gen-Onsager-block-decomp} and the fact that $k=\C$ and $\g_{0,\rss}=0$. \qed
\end{proof}

%
\appendix
\section{Extensions and the weight lattice}
%

We present here some useful results of independent interest that allow us, in certain cases, to give a simple explicit description of the block decomposition of the category $\mathcal{F}$ of finite-dimensional representations of an equivariant map algebra.

The following proposition and its proof were explained to us by S.~Kumar.

\begin{prop}\label{prop:kumar} Let $G$ be a semisimple
algebraic group over $k$ and let $\rho : G \to \SL(U)$ be a faithful
finite-dimensional rational $G$-module. Then, for any
irreducible finite-dimensional $G$-module $V$, there exists $m \in
\NN$ such that $V$ is isomorphic to an irreducible submodule of
$U^{\ot m}$.
\end{prop}

\begin{proof} Let $\rho : G \to \SL_k(U)$ be the corresponding
representation. Besides $\rho$ we will  use the trivial $G$-module
structure on the underlying vector space of $U$, denoted $U^{\rm
triv}$.  Observe that the canonical $G$-module structure on
$\End_k(U) = \Hom_k(U^{\rm triv}, U)$ satisfies $(g \cdot f)(u) = g
\cdot f(u)$ for $g \in G$, $f\in \Hom_k(U^{\rm triv}, U)$, and $u\in U$. Let
$\theta : G \to \Hom_k(U^{\rm triv}, U)$ be the composition of
$\rho$ and the canonical injection $\SL_k(U) \to \Hom_k (U^{\rm
triv}, U)$. It is easily seen that $\theta $ is an injective
$G$-module map, where  $G$ acts on itself by left multiplication.
\lv{Indeed, let $g,h\in G$ and $u \in U$. Then $\big(\theta (g \cdot
h)\big)(u) = \big(\theta (gh)\big)(u) = \big(\theta (g)
\theta(h)\big)(u) = g\cdot \big(\theta(h)(u)\big) \theta(h) = \big(
g \cdot \theta(h)\big)(u)$. Also, assume $\theta(g) = \Id$, i.e.,
for all $f\in \End_k(U)$ and $u\in U$ we have $f(u) = (g \cdot f)
(u)= g \cdot f(u)$, hence $\rho(g) = \Id$ and then $g=1$ by
assumption on $\rho$.} We therefore get a surjective $G$-module map
of the coordinate algebras of the corresponding algebraic varieties,
\[
  \theta^* : k[ \Hom(U^{\rm triv}, U)] \twoheadrightarrow k[G].
\]
As $G$-modules,
\[
  \Hom_k(U^{\rm triv}, U) \cong (U^{\rm triv})^* \ot U \cong
  U^{\rm triv} \ot U \cong U \oplus \cdots \oplus U,
\]
with the last direct sum having $\dim U$ summands. Therefore, again
as $G$-modules,
\[
  k[ \Hom(U^{\rm triv}, U)] \cong \Sym^\bullet(U^* \oplus \cdots \oplus U^*) \cong \Sym^\bullet(U^*) \ot \cdots \ot \Sym^\bullet (U^*),
\]
where $\Sym^\bullet(\cdot)$ is the symmetric algebra. We compose the
$G$-module map $\theta^*$ with the canonical $G$-module epimorphism
\[
  \T^\bullet(U^*) \ot \cdots \ot \T^\bullet(U^*) \twoheadrightarrow \Sym^\bullet(U^*) \ot \cdots \ot \Sym^\bullet (U^*)
\]
from the tensor product of the tensor algebras, to get a $G$-module epimorphism
\[
  \T^\bullet(U^*) \ot \cdots \ot \T^\bullet(U^*) \twoheadrightarrow k[G].
\]
It now follows that for any irreducible finite-dimensional
$G$-submodule $W \subseteq k[G]$, there exist natural numbers $n_1,
\ldots, n_s$, $s=\dim U$, such that
\[
  \T^{n_1} (U^*) \ot \cdots \ot \T^{n_s}(U^*) \twoheadrightarrow W
\]
is a $G$-module epimorphism. But clearly,
\[
  \T^{n_1} (U^*) \ot \cdots \ot \T^{n_s}(U^*) \cong \T^m(U^*) \quad \hbox{for $m=n_1 + \cdots + n_s$}
\]
as $G$-modules, yielding a $G$-module epimorphism $\T^m(U^*) \to W$.
By complete reducibility, this means that $W$ is isomorphic to an
irreducible component of $\T^m(U^*)$. But then $W^*$ is isomorphic
to an irreducible component of $(\T^m(U^*))^* \cong \T^m(U^{* \, *})
\cong \T^m(U)$. Finally, we apply the algebraic version of the
Peter-Weyl Theorem, which says that every finite-dimensional
representation of $G$ occurs as a submodule of $k[G]$,
\cite[Cor.~4.2.8]{GW09}. Since $V$ is irreducible if and only if $V^*$
is so, we can apply the above argument to $W=V^*$, and in this way
finish the proof of the theorem. \qed
\end{proof}

The following lemma is a generalization of the second part of the proof of \cite[Prop.~1.2]{CM}.

\begin{lem}\label{lem:gen}
Let $\frl$ be a finite-dimensional Lie algebra. Let $U, V, W$ be finite-dimensional $\frl$-modules with $U$
completely reducible and $V,W$ irreducible with\linebreak $\Hom_\frl(U^{\ot m} \ot V, W) \ne 0$ for some $m\in \NN_+$. Then there exists a finite sequence $V=V_0, V_1, \ldots, V_m=V$ of irreducible finite-dimensional $\frl$-modules such that
\[
  \Hom_\frl(U \ot V_i, V_{i+1}) \ne 0\quad \text{for } 0 \le i < m.
\]
\end{lem}

\begin{proof}
We prove the result by induction on $m\in \NN_+$, the case $m=1$ being obvious. For $m>1$ we have
\[
  0 \ne \Hom_\frl(U^{\ot m} \ot V, W) \cong \Hom_\frl (U^{\ot (m-1)}\ot V, U^* \ot W).
\]
Since both $U^{\ot m-1} \ot V$ and $U^* \ot W$ are completely
reducible by \cite[\S6.5, Cor.~1 du Th.~4]{Bou:lie}, the above implies that
there exists an irreducible finite-dimensional $\frl$-module $X$
which is an $\frl$-submodule of both $U^{\ot (m-1)} \ot V$ and $U^*
\ot W$. But then $\Hom_\frl(U^{\ot (m-1)} \ot V, X) \ne 0$ and
$\Hom_\frl(X, U^* \ot W) \cong \Hom_\frl(U \ot X, W) \ne 0$.
Applying the induction hypothesis to $\Hom_\frl(U^{\ot (m-1)} \ot V,
X) \ne 0$ and putting $X=V_{m-1}$, finishes the proof. \qed
\end{proof}

For the remainder of the appendix, let $\frs$ be a semisimple finite-dimensional Lie algebra with weight lattice, root lattice and set of dominant integral weights $P,Q,P^+$ respectively.  Let $W(\frs)$ be the Weyl group of $\frs$.  Recall that for a finite-dimensional $\frs$-module $U$, $\wt(U)$ denotes the set of weights of $U$.

\begin{lem} \label{lem:hom-space-root-lattice}
Suppose $U$ is a finite-dimensional $\frs$-module and $\lambda, \mu \in P^+$ such that $\Hom_\frs(U \otimes V(\lambda), V(\mu)) \ne 0$.  Then $\mu - w_0(\lambda)$ and $\mu - \lambda$ are both elements of $\wt(U) + Q$.
\end{lem}

\begin{proof}
First recall that $V(\la)^*$ is an irreducible $\frs$-module of highest weight $-w_0(\la)$, where $w_0$ is the longest element in $W(\frs)$.  Since
\[
  \Hom_\frs(U,V(\lambda)^* \otimes V(\mu)) = \Hom_\frs(U \otimes V(\lambda), V(\mu)) \ne 0,
\]
we have $\mu - w_0(\lambda) - \omega \in \wt(U)$ for some $\omega \in Q$.  Thus $\mu - w_0(\lambda) \in \wt(U) + Q$.  Furthermore,
\[
  \mu - \lambda = (\mu - w_0(\lambda)) + (w_0(\lambda)-\lambda) \in \wt(U)+Q,
\]
since $w(\xi) - \xi \in Q$ for all $\xi \in P$ and $w \in W(\frs)$ by \cite[Ch.~VI, \S1.9, Prop.~27]{Bou81}. \qed
\end{proof}

\begin{cor}\label{cor:kumar}
Let $U$ be a finite-dimensional \emph{faithful} $\frs$-module.  Then $Q \subseteq \Span_\Z \wt(U) \subseteq P$.  Furthermore, for $\la, \mu \in P^+$, the following two conditions are equivalent:
\begin{enumerate}
  \item\label{cor-item:weight-chain} There exists a sequence $\la = \la_0, \la_1, \ldots, \la_n = \mu $ of weights $\la_i\in P^+$ such that
      \[
        \Hom_\frs\big(U \ot V(\la_{i}), V(\la_{i+1})\big) \ne 0 \quad \text{for } 0 \le i < n.
      \]

  \item\label{cor-item:mu-lambda-in-U} $\mu - \la\in \Span_\Z \wt(U)$.
\end{enumerate}
\end{cor}

\begin{proof}
That $Q \subseteq \Span_\Z \wt(U) \subseteq P$ is known;
see for example \cite[Exercise~21.5]{Hum72}. \lv{Details: It is clear
that $\Span_\Z \wt(U) \subseteq P$. To see that $Q \subseteq \Span_\Z \wt(U)$, let $\al$
be a root of $\frs$ and let $0 \ne x_\al \in \frs_\al$. Also, let $\rho$
be the representation affording $U$. Since $\rho(x_\al) \ne 0$ by
faithfulness, there exists a weight space $V_\mu$ of $V$ such that
$0 \ne \rho(x_\al) V_\mu \subset V_{\mu + \al}$. Hence $\mu$ and
$\mu + \al$ are weights and thus $\al = (\mu + \al) - \al \in
\Span_\Z \wt(U)$.}
Assume \eqref{cor-item:weight-chain}.  By Lemma~\ref{lem:hom-space-root-lattice}, we have
\begin{multline*}
  -w_0 (\la) + ( \la_1 - w_0(\la_1)) + \cdots + (\la_{n-1} - w_0(\la_{n-1})) + \mu  \\
  \qquad =  (\la_1 - w_0(\la_0))  + (\la_2 - w_0( \la_1)) + \cdots + (\la_n - w_0(\la_{n-1}) )\in \Span_\Z \wt(U) + Q .
\end{multline*}
Now using that $\xi - w\xi \in Q$ for $\xi \in P$ and $w\in W(\frs)$ (\cite[Ch.~VI, \S1.9, Prop.~27]{Bou81}), we see that $\mu - w_0(\la) \in \Span_\Z \wt(U)+Q$. But this is equivalent to \eqref{cor-item:mu-lambda-in-U} since $\mu - \la = (\mu - w_0(\la)) + (w_0(\la) - \la)$ and $w_0(\la) - \la \in Q$.

To prove that~\eqref{cor-item:mu-lambda-in-U} implies~\eqref{cor-item:weight-chain}, we will use some standard facts from the theory of Chevalley groups, for which the reader is referred to
\cite{Ste68}. Assume \eqref{cor-item:mu-lambda-in-U} is true and let $G$ be the Chevalley group corresponding to the representation $U$ of $\frs$. This is a semisimple algebraic
$k$-group (\cite[Th.~6]{Ste68}), whose weight lattice (group of characters of a maximal
torus) is $\Span_\Z \wt(U)$ (\cite[p.~60]{Ste68}).  The $\frs$-module $U$ is
canonically a faithful rational $G$-module, also denoted $U$. The
$\frs$-module $V(\la)^* \ot V(\mu)$ contains a highest weight vector
of weight $\mu - w_0(\la)$, hence also an irreducible $\frs$-module
$X$ of highest weight $\mu - w_0(\la)\in (\Span_\Z \wt(U)) \cap P^+$. It
integrates to an irreducible $G$-module of highest weight $\mu -
w_0(\la)$, also denoted $X$ (\cite[Th.~39]{Ste68} and the remark on
p.~211 of loc.~cit.).  We can now apply Theorem~\ref{prop:kumar}
and conclude that there exists $m \in \NN$ such that $X$ is
isomorphic to an irreducible $G$-submodule of $U^{\ot m}$, i.e.,
$\Hom_G( U^{\ot m}, X) \ne 0$. Since $\Hom_G(U^{\ot m}, X) \cong
\Hom_\frs( U^{\ot m}, X)$, we have $\Hom_\frs(U^{\ot m} \ot V(\la), V(\mu)) \cong \Hom_\frs(U^{\ot m}, V(\la)^* \ot V(\mu)) \ne 0$. Now \eqref{cor-item:weight-chain} follows from Lemma~\ref{lem:gen}. \qed
\end{proof}

\begin{rem} \label{rem:Kostant-Chari-Moura}
The special case $U=\frs$ (so $\Span_\Z \wt(U) = Q$) of Corollary~\ref{cor:kumar} is proven in \cite[Prop.~1.2]{CM}, using a result of Kostant's instead of Theorem~\ref{prop:kumar}.

Note that Corollary~\ref{cor:kumar} applies in the following setting: $\g$ a simple finite-dimensional Lie algebra with an automorphism of order $2$, $\g = \g_0 \oplus \g_1$ the corresponding eigenspace decomposition, $\frs=\g_0$ semisimple (see \cite[Chapter~X, \S5, Table~II]{Hel01} for a list of the cases in which this condition is fulfilled) and $U=\g_1$, which is a faithful $\frs$-module (see Example~\ref{eg:order-two}\eqref{eg-item:order-two:a}).
\end{rem}


\printbibliography

\end{document}